\documentclass[11pt]{amsart}
\usepackage{amsmath,amssymb}
\usepackage{thmtools}
\usepackage[margin=2.5cm, marginparwidth=2cm]{geometry}
\usepackage{mathtools}
\usepackage{extarrows}
\usepackage{enumerate}
\usepackage{commath}
\usepackage{float}
\usepackage{tikz, tikz-cd}
\usepackage{eucal}
\usepackage{graphicx}
\usepackage{hyperref}
\usepackage[nameinlink, capitalise]{cleveref}

\newcommand{\dd}{\partial}
\newcommand{\R}{\mathbb{R}}
\newcommand{\N}{\mathbb{N}}
\newcommand{\Z}{\mathbb{Z}}
\newcommand{\CC}{\mathbb{C}}
\newcommand{\suchthat}{\;\ifnum\currentgrouptype=16 \middle\fi|\;}
\newcommand{\topp}{\bigwedge\nolimits^{\mathrm{top}}}
\newcommand{\gad}{\rotatebox[origin=c]{180}{$\dag$}}

\DeclareMathOperator{\im}{im}
\DeclareMathOperator{\id}{id}
\DeclareMathOperator{\CZ}{CZ}
\DeclareMathOperator{\ind}{ind}
\DeclareMathOperator{\codim}{codim}
\DeclareMathOperator{\coker}{coker}

\DeclareMathOperator{\ev}{ev}

\newcommand*{\defeq}{\mathrel{\vcenter{\baselineskip0.65ex \lineskiplimit0pt
                     \hbox{$\raisebox{-0.200ex}{\scriptsize.}$}\hbox{\scriptsize.}}}
                     =}
\newcommand*{\eqdef}{=\mathrel{\vcenter{\baselineskip0.45ex \lineskiplimit0pt
                     \hbox{$\raisebox{0.10ex}{\scriptsize.}$}\hbox{\scriptsize.}}}
                     }
\newcommand{\co}{\colon\thinspace}
\newcommand{\setmin}{\mathbin{\vcenter{\hbox{$\rotatebox{20}{\raisebox{0.6ex}{$\scriptscriptstyle\mathrlap{\setminus}{\hspace{.2pt}}\setminus$}}$}}}}

\newcommand{\imp}{\ \Rightarrow \ }
\newcommand{\paren}[1]{\left( #1 \right)}
\newcommand{\parenb}[1]{\left[ #1 \right]}
\newcommand{\parenm}[1]{\left\{ #1 \right\}}
\newcommand{\ip}[2]{\left \langle #1, #2 \right \rangle}

\declaretheorem[style=plain, name=Theorem, numberwithin=section]{thm}

\declaretheorem[style=plain, name=Proposition, sibling=thm]{prp}
\declaretheorem[style=plain, name=Corollary, sibling=thm]{cor}
\declaretheorem[style=plain, name=Lemma, sibling=thm]{lma}
\declaretheorem[style=definition, name=Remark, sibling=thm]{rmk}
\declaretheorem[style=definition, name=Example, sibling=thm]{ex}
\numberwithin{equation}{section}

\title{Fiber Floer cohomology and conormal stops}
\author{Johan Asplund}
\address{Department of Mathematics, Uppsala University, 751 06 Uppsala, Sweden}
\email{johan.asplund@math.uu.se}

\begin{document}
\begin{abstract}
	Let $S$ be a closed orientable spin manifold. Let $K \subset S$ be a submanifold and denote its complement by $M_K$. In this paper we prove that there exists an isomorphism between partially wrapped Floer cochains of a cotangent fiber stopped by the unit conormal $\varLambda_K$ and chains of a Morse theoretic model of the based loop space of $M_K$, which intertwines the $A_\infty$-structure with the Pontryagin product. As an application, we restrict to codimension 2 spheres $K \subset S^n$ where $n = 5$ or $n\geq 7$. Then we show that there is a family of knots $K$ so that the partially wrapped Floer cohomology of a cotangent fiber is related to the Alexander invariant of $K$. A consequence of this relation is that the link $\varLambda_K \cup \varLambda_x$ is not Legendrian isotopic to $\varLambda_{\mathrm{unknot}} \cup \varLambda_x$ where $x\in M_K$.
\end{abstract}
\maketitle
\tableofcontents
\section{Introduction}
	In this paper we consider the wrapped Floer cohomology of a cotangent fiber with wrapping stopped by a conormal. We relate it to chains of based loops on the complement of a submanifold. Then we show that the Legendrian conormal knows about the smooth topology of the submanifold beyond the fundamental group.

	Let $S$ be a closed orientable spin manifold. Let $K \subset S$ be a submanifold and denote its complement by $M_K$. Consider the disk cotangent bundle $DT^\ast S$ equipped with the canonical Liouville form $\lambda = pdq$. The ideal contact boundary of the Weinstein domain $DT^\ast S$ is the unit cotangent bundle $ST^\ast S$. Associated to $K$ are the conormal bundle
	\[
		L_K = \parenm{(q,p) \in T^\ast S \suchthat q \in K, \, \langle p, T_qK \rangle = 0} \subset DT^\ast S\, ,
	\]
	and the unit conormal $\varLambda_K = L_K \cap ST^\ast S$. Consider a cotangent fiber $F = DT^\ast_\xi S$ at $\xi \in M_K$ and let $CW^\ast_{\varLambda_K}(F,F)$ be the partially wrapped Floer cochains on $F$ with wrapping stopped by $\varLambda_K$. Let $BM_K$ denote the space of piecewise geodesic loops in $M_K$ based at $\xi$. Consider the space $C^{\text{cell}}_{-\ast}(BM_K)$ of cellular chains of $BM_K$ equipped with the Pontryagin product. Then we have the following result:
	\begin{thm}[\cref{thm:isomorphism_cw_and_cell_bmk} and \cref{thm:extending_Y1_to_Zpi_module_iso}]\label{thm:theorem_A}
		There exists a geometrically defined isomorphism of $A_\infty$-algebras $\varPsi \co CW^\ast_{\varLambda_K}(F,F) \longrightarrow C^{\text{cell}}_{-\ast}(BM_K)$.

		Moreover, $\varPsi$ induces an isomorphism $HW^\ast_{\varLambda_K}(F,F) \longrightarrow H_{-\ast}(\varOmega_\xi M_K)$ of $\Z[\pi_1(M_K)]$-modules.
	\end{thm}
	We define $CW^\ast_{\varLambda_K}(F,F)$ using a surgery approach similar to \cite[Appendix B]{ekholm2017duality} and \cite[Section 6]{ekholm2016completearXiv} (see \cref{sec:partially_wrapped_floer_cohomology_using_a_surgery_approach} for details).
	The outline of the surgery approach is the following. We attach a handle modeled on $D_\varepsilon T^\ast ([0,\infty) \times \varLambda_K)$ to $DT^\ast S$ along a neighborhood of $\varLambda_K$. We denote the resulting Liouville sector by $W_K$ (with terminology as in \cite{ganatra2017covariantly}). Then $CW^\ast_{\varLambda_K}(F,F)$ is the wrapped Floer cochain complex of $F$ in $W_K$. The skeleton of $W_K$ is $L_K \cup S$ with clean intersection $L_K \cap S = K$. By performing Lagrangian surgery along the clean intersection, we obtain an exact Lagrangian submanifold $M_K \subset W_K$ which is diffeomorphic to the complement $S \setmin K$ (see \cref{sec:partially_wrapped_floer_cohomology_using_a_surgery_approach} for details). 

 Let $\varOmega_\xi M_K$ denote the space of loops in $M_K$ based at $\xi$. Consider singular chains on the space of based loops $C_{-\ast}(\varOmega_\xi M_K)$. We give it the structure of an $A_\infty$-algebra by equipping it with the Pontryagin product and all higher products equal to zero. See \cref{sub:based_loops_on_} and \cref{sub:filtration_on_based_loops} for a more detailed discussion about the model of the based loop space we use.

	In the spirit of Cieliebak--Latschev \cite{cieliebak2009role} and Abouzaid \cite{abouzaid2012wrapped}, we have a geometrically defined $A_{\infty}$-homomorphism $\varPsi \co CW^\ast_{\varLambda_K}(F,F) \longrightarrow C_{-\ast}(\varOmega_\xi M_K)$. By analyzing the action filtrations, we show that $\varPsi$ is diagonal with respect to the action filtrations. A key point in proving that $\varPsi$ is an isomorphism is showing that the disks contributing to the diagonal are transversely cut out. The solutions of the linearized Floer equation are precisely those vector fields along the disk that restricts to broken Jacobi fields along $\gamma$ on which the Hessian of the energy functional is negative definite.
 
	In the surgery approach we attach a handle modeled on $D_\varepsilon T^\ast([0,\infty) \times \varLambda_K)$, with skeleton $[0,\infty) \times \varLambda_K$. We consider a generic product metric on $[0,\infty) \times \varLambda_K$ such that the metric on $\varLambda_K$ is scaled by a positive function with strictly negative derivative (warped product metric), see \eqref{eq:metric_definition}. By the genericity of the metric, there is a natural one-to-one correspondence between Reeb chords and geodesics, see \cref{lma:1-1-correspondence_geodeiscs_reeb_of_index_k} for details. 

	Since $W_K$ and $M_K$ are non-compact we use monotonicity of $J$-holomorphic curves to prove that relevant moduli spaces of $J$-holomorphic curves are compact, see \cref{sec:monotonicity_estimates} for details. 

	\subsection{Applications}
	 Let $Q$ be a smooth manifold and let $K \subset Q$ be a submanifold. Consider the cotangent bundle $T^\ast Q$ and the unit conormal bundle $\varLambda_K$. It is known in certain cases that the symplectic topology of $T^\ast Q$ knows about the smooth topology of $Q$ \cite{abouzaid2012framed,ekholm2016exact,ekholm2016lagrangian}. In some cases the contact topology of $\varLambda_K$ knows about the smooth topology of $K$. For instance, it is known that conormal tori $\varLambda_K \subset ST^\ast \R^3$ of knots $K \subset \R^3$ are complete knot invariants \cite{shende1604conormal,ekholm2016completearXiv}. The results of Ekholm--Ng--Shende fit nicely into the broader picture of partially wrapped Floer cohomology that we consider in this paper, and is summarized in \cite[Section 1.3]{ekholm2016completearXiv}. Specifically, in \cite{ekholm2016completearXiv} it is proven that there is a ring isomorphism
		\[
			HW^0_{\varLambda_K}(F,F) \cong \Z[\pi_1(M_K)]\, ,
		\]
		which is also obtained from \cref{thm:theorem_A} by restricting to degree 0. Furthermore there is a relation  between the knot contact homology of $K \subset \R^3$ and the Alexander polynomial of $K$ \cite{ng2008framed,ekholm2016completearXiv}.

		Let $K \subset S^n$ be a codimension 2 sphere. In this paper we show that the partially wrapped Floer cohomology of the fiber is related to the Alexander invariant. The Alexander invariant is $H_\ast(\widetilde M_K)$ regarded as a $\Z[\pi_1(M_K)]$-module, where $\widetilde M_K$ denotes the infinite cyclic cover of $M_K$, see \cref{sub:using_the_leray--serre_spectral_sequence} for details. Denote by $\varLambda_{\mathrm{unknot}}$ the unit conormal of the standard embedded $S^{n-2} \subset S^n$. As an application of \cref{thm:theorem_A} we have the following theorem.
		\begin{thm}[\cref{thm:exists_codim_2_knot_with_conormal_not_leg_iso_to_unknot}]\label{thm:thm_B}
			Let $n = 5$ or $n\geq 7$. Let $x \in M_K$ be a point. Then there exists a codimension 2 knot $K \subset S^{n}$ with $\pi_1(M_K) \cong \Z$, such that $\varLambda_{K} \cup \varLambda_x$ is not Legendrian isotopic to $\varLambda_{\mathrm{unknot}} \cup \varLambda_x$.
		\end{thm}
	\subsection{Relation to other results}
		Let $Q$ be a closed smooth manifold and consider the exact symplectic manifold $(T^\ast Q, d \lambda)$ where $\lambda$ is the canonical Liouville form $\lambda = pdq$. Abbondandolo--Schwarz proved that the wrapped Floer cohomology of a cotangent fiber $T_\xi^\ast Q$ is isomorphic to the homology of the based loop space of $Q$ \cite{abbondandolo2006floer}. Abouzaid extended this to an $A_\infty$-quasi-isomorphism in \cite{abouzaid2012wrapped} where the loop space is equipped with the Pontryagin product. Recently, Ganatra--Pardon--Shende proved that this result continues to hold even when $Q$ is not assumed to be compact as a consequence of a deeper relationship between the wrapped Fukaya category of a Liouville sector and a certain category of sheaves \cite{ganatra2018microlocal}.

		In this paper, we consider a similar $J$-holomorphic curve setup to the one used by Abouzaid in \cite{abouzaid2012wrapped}, but instead we work in the context of the partially wrapped Fukaya category of $T^\ast S$ stopped by the unit conormal $\varLambda_K$.

		\begin{rmk}\label{rmk:handle_attachment_remark}
		Another interesting geometric point of view which motivates \cref{thm:theorem_A} is the following. Consider the wrapped Fukaya category of $T^\ast M_K$ \cite{ganatra2017covariantly,bae2019wrapped}. By \cite[Corollary 6.1]{ganatra2018microlocal} we have an $A_{\infty}$-quasi-isomorphism
		\[
			CW^\ast(F,F) \cong C_{-\ast}(\varOmega_\xi M_K)\, ,
		\]
		where $F \subset T^\ast M_K$ is the cotangent fiber at $\xi \in M_K$. We realize $W_K$ as the result of attaching a handle to $T^\ast M_K$ as follows: Take a tubular neighborhood $N(K) \subset S$ of $K$ and consider $N'(K) \defeq N(K) \cap M_K$. Then remove $L_{N'(K)} \subset T^\ast M_K$ and replace it with $L_{N(K)}$, identifying their common boundaries $\varLambda_{N(K)} = \varLambda_{N'(K)}$.
		\begin{figure}[H]
			\centering
			\includegraphics{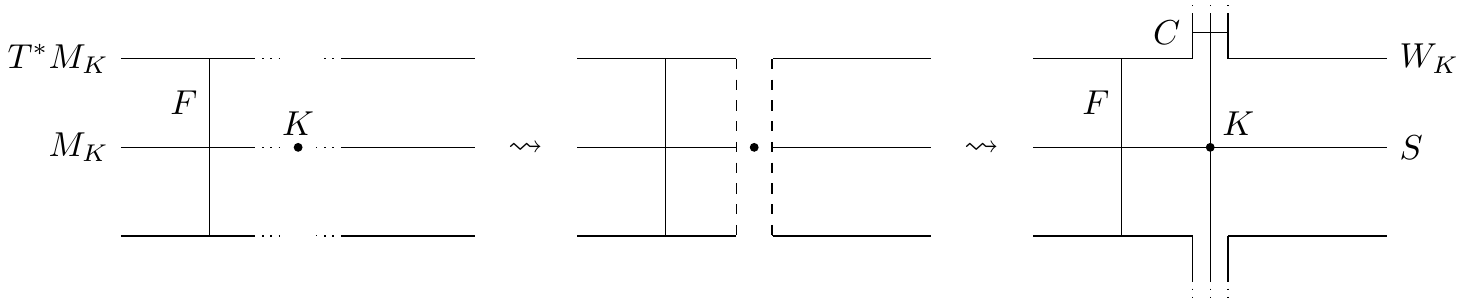}
			\caption{The figure shows the construction of $W_K$ via handle attachment on $T^\ast M_K$.}
		\end{figure}
	 From the point of view of handle attachment, there is a new generator of the wrapped Fukaya category, namely the cocore disk $C$. Because of this, the wrapped Floer cohomology of the fiber $F$ will change on the level of chains. However, if we push $C$ very far out in the punctured handle by a Lagrangian isotopy, we look at filtered $A_{\infty}$-algebras and yield a chain isomorphism
		\[
			\mathcal F_L CW^\ast(F,F)_{W_K} \cong \mathcal F_L CW^\ast(F,F)_{T^\ast M_K}\, ,
		\]
		where $\mathcal F_L$ means we only consider generators of action less than $L$. A standard filtration argument then shows that the wrapped Floer cohomology of $F$ is unaffected by this type of handle attachment and thus $HW^\ast(F,F)_{W_K} \cong HW^\ast(F,F)_{T^\ast M_K}$. Hence we obtain an indirect proof of the isomorphism
		\[
			HW^\ast_{\varLambda_K}(F,F) = HW^{\ast}(F,F)_{W_K} \cong HW^\ast(F,F)_{T^\ast M_K} \cong H_{-\ast}(\varOmega M_K)
		\]
		in \cref{thm:theorem_A}.
	\end{rmk}
	\subsection{Organization of the paper}
		In \cref{sec:wrapped_floer_cohomology_wo_ham} we describe the version of wrapped Floer cohomology defined without Hamiltonian perturbations which we use in this paper. In \cref{sec:partially_wrapped_floer_cohomology_and_chains_of_based_loops} we first discuss the surgery approach to define partially wrapped Floer cohomology. Then we define the operations $\varPsi = \parenm{\varPsi_m}_{k=1}^\infty$ between $CW^\ast_{\varLambda_K}(F,F)$ and $C_{-\ast}(\varOmega_\xi M_K)$ and show that $\varPsi$ is an $A_{\infty}$-homomorphism. \cref{sec:chain_map_is_a_quasi-isomorphism} is devoted to proving that $\varPsi$ is a isomorphism between $CW^\ast_{\varLambda_K}(F,F)$ and the Morse theoretic model of chains of based loops. Lastly, in \cref{sec:applications} we equip $CW^\ast_{\varLambda_K}(F,F)$ and $C_{-\ast}(\varOmega_\xi M_K)$ with $\Z[\pi_1(M_K)]$-module structures relate $HW^\ast_{\varLambda_K}(F,F)$ to the Alexander invariant $H_\ast(\widetilde M_K)$ for certain families of codimension 2 knots $K \subset S^n$. Then we show that this relation is used to show that $\varLambda_K \cup \varLambda_x$ is not Legendrian isotopic to $\varLambda_{\mathrm{unknot}} \cup \varLambda_x$, where $x\in S^n \setmin K$ is a point.
	\subsection*{Acknowledgments}
		The author would like to thank his PhD advisor Tobias Ekholm for all his guidance and helpful discussions. He would also like to thank Georgios Dimitroglou Rizell for useful discussions regarding \cref{rmk:handle_attachment_remark}. Finally, the author would like to thank the anonymous referee whose many comments has improved the exposition of the paper. The author was supported by the Knut and Alice Wallenberg Foundation.
\section{Wrapped Floer cohomology without Hamiltonian}
	\label{sec:wrapped_floer_cohomology_wo_ham}
	In this paper, we consider a version of wrapped Floer cohomology defined without Hamiltonian perturbations. Wrapped Floer cohomology without Hamiltonian has been studied in e.g.\@ \cite{ekholm2012rational,rizell2013lifting,ekholm2017duality} and in particular it is useful in proving various surgery formulas involving the wrapped Floer cohomology \cite{bourgeois2012effect,ekholm2017duality,ekholm2019holomorphic}. It has also been used to study knots via knot contact homology from which there is a relationship to string topology and the cord algebra \cite{ekholm2016completearXiv, ekholm2013knot, cieliebak2017knot}.

	\begin{rmk}
		The relationship between wrapped Floer cohomology defined with and without Hamiltonians has also been studied. The version without Hamiltonian is known to be quasi-isomorphic to the version defined with Hamiltonians by counting strips with a Hamiltonian term that is turned on as one goes from the positive end to the negative end \cite[Theorem 7.2]{ekholm2016legendrian}. Such $J$-holomorphic maps with a Hamiltonian term that turns on has been more systematically studied in \cite{ekholm2017symplectic} and it is proven in \cite[Lemma 68, 69]{ekholm2017duality} that the two versions of wrapped Floer cohomology are $A_{\infty}$-quasi-isomorphic.

		When working with wrapped Floer cohomology without Hamiltonian we have a priori bubbling issues. This is circumvented by considering parallel copies, which also removes the possibility of having multiply covered curves, see \cite[Section 3.3]{ekholm2017duality}. Furthermore we need to count anchored curves \cite[Section 2.2]{bourgeois2012effect} \cite[Section A.1]{ekholm2017duality}. A specific perturbation scheme involving anchored curves is constructed in \cite{ekholm2019holomorphic}, and we fix such perturbation scheme so that all relevant moduli spaces are transversely cut out.
	\end{rmk}

	We give a brief description of the wrapped Floer cohomology without Hamiltonian by following \cite[Appendix A-B]{ekholm2017duality}. We consider a Weinstein domain $M$ together with a smooth exact Lagrangian submaniold $(M,\omega \defeq d\lambda, L)$. Let $Y \defeq \dd M$ and $\varLambda \defeq L \cap Y$ be its Legendrian boundary. The boundary $(Y, \alpha \defeq \eval[0]\lambda_{Y})$ is a contact manifold. We consider the completion of $M$ and $L$ by attaching cylindrical ends $[0,\infty) \times Y$ to $Y$ and $[0,\infty) \times \varLambda$ to $\varLambda$. Then we pick a system of parallel copies of $L$ as in \cite[Section 3.3]{ekholm2017duality}. Consider a family $(H_k,h_k)_{k=1}^\infty$ of pairs of Morse functions, $H_k \co L \longrightarrow \R$ and $h_k \co \varLambda \longrightarrow \R$. Let $L_k$ be the time-1 flow of $L$ of the Hamiltonian vector field $X_{H_k}$, and let $\varLambda_k \defeq L_k \cap Y$. Then we call $\parenm{L_k}_{k=0}^\infty$ a \emph{system of parallel copies of $L$} where $L_0 \defeq L$. Let $\overline L \defeq \bigcup_{k=0}^\infty L_k$ and $\overline \varLambda \defeq \bigcup_{k=0}^\infty \varLambda_k$.

	Note that in this paper, $L$ is a cotangent fiber. Therefore we choose the Morse functions $H_k$ in such a way that all of them only have one minimum, since $L \cong D^n$.
	\subsection{$A_{\infty}$-structure and moduli space of disks}
		\label{sub:a_infty_structure_and_moduli_space_of_disks_cw_wo_ham}
		Let $(M, \lambda)$ be a spin Weinstein domain. Let $L \subset M$ be an orientable exact Lagrangian with vanishing Maslov class (see \cite{arnol1967characteristic} for a definition of the Maslov class). Let $\overline L = \parenm{L_k}_{k=0}^\infty$ be the corresponding system of parallel copies of $L$ as in the previous section.

		First we define $CW^\ast(L,L)$ as a $\Z$-graded module over $\Z$. Note that, for each Reeb chord $c'$ starting at $\varLambda_{i}$ and ending at $\varLambda_{j}$, there is a unique Reeb chord $c$ of $\varLambda$ close to $c'$. Similarly, for each transverse intersection point $a'$ in $L_{i} \cap L_{j}$, there is a unique transverse intersection point $a \in L_{0} \cap L_{1}$. We implicitly fix an identification of $c'$ with $c$, and $a'$ with $a$. We then define $CW^\ast(L,L)$ to be the $\Z$-graded module over $\Z$, which is generated by Reeb chords of $\varLambda$ and intersection points $L_0 \cap L_1$. The grading is given by the Maslov index (see \cref{rmk:description_of_maslov_index_of_reeb_chord_generators} below for a more precise definition).

		We now describe how we equip $CW^\ast(L,L)$ with a $A_{\infty}$-structure $\parenm{\mu^i}_{i=1}^\infty$ which is defined by $J$-holomorphic curve counts. Let $D_{m} \subset \CC$ denote the positively oriented unit disk, with $m$ points along the boundary removed. We denote the boundary punctures in $D_m$ by $\zeta_{1},\ldots,\zeta_{m}$, one of which is distinguished. These boundary punctures subdivide the boundary of $D_m$ into $m$ arcs. We enumerate these arcs by $\kappa_{1}, \ldots, \kappa_{m}$, according to the boundary orientation, starting from the distinguished boundary puncture. We call $\kappa \defeq \parenm{\kappa_{i}}_{i=1}^{m}$ a \emph{boundary numbering} of $D_{m}$. If the sequence $\kappa$ is decreasing (increasing), we say that the disk $D_{m}$ has decreasing (increasing) boundary numbering $\kappa$. If $\kappa_{i-1} \leq \kappa_{i}$ ($\kappa_{i-1} \geq \kappa_{i}$), we say that the puncture $\zeta_{i}$ is increasing (decreasing), and if $\kappa_{i-1} = \kappa_{i}$ we say that $\zeta_{i}$ is a constant puncture. 

		We equip the boundary punctures $\zeta_j \in \partial D_m$ with both a positive and a negative strip-like end. Namely, we pick biholomorphisms
		\[
			\begin{cases}
				\varepsilon_+^i \co (0,\infty) \times [0,1] \longrightarrow N(\zeta_i)\\
				\varepsilon_-^i \co (-\infty,0) \times [0,1] \longrightarrow N(\zeta_i)
			\end{cases} \forall i\in \parenm{1,\ldots, m}\, ,
		\]
		where $N(\zeta_i)$ is a neighborhood of the boundary puncture $\zeta_i\in \partial D_m$. 

		Using notation as in \cite{ekholm2017duality}, we are interested in the moduli spaces of $J$-holomorphic disks which are denoted by $\mathcal M^{\text{fi}}(\boldsymbol{c};\kappa)$, $\mathcal M^{\text{sy}}(\boldsymbol{c};\kappa)$ and $\mathcal M^{\text{pb}}(\boldsymbol{c};\kappa)$. These moduli spaces consist of \emph{filling disks}, \emph{symplectization disks} and \emph{partial holomorphic buildings} respectively, and we define them below.
		\begin{description}
			\item[Filling disks] 
				Consider $D_{m}$ equipped with a strictly decreasing boundary numbering. Note that every puncture is strictly decreasing except for the distinguished puncture, which is strictly increasing. We let $\boldsymbol c = c_{1} \cdots c_{m}$ be a word of generators of $CW^{\ast}(L,L)$. Then we define $\mathcal M^{\text{fi}}(\boldsymbol c; \kappa)$ to be the moduli space of $J$-holomorphic maps $u\co (D_{m}, \dd D_{m}) \longrightarrow (M, \overline L)$ such that
				\begin{itemize}
					\item near the boundary puncture $\zeta_{i}$, $u$ is asymptotic to the generator $c_{i}$, that is
					\[
						\begin{cases}
							\lim_{s\to \pm \infty} u(\varepsilon_{\pm}^i(s,t)) = c_i, & \text{ if } c_i \text{ is an intersection generator} \\
							\lim_{s\to \pm \infty} u(\varepsilon_{\pm}^i(s,t)) = (\infty,c_i), & \text{ if } c_i \text{ is a Reeb chord generator.}
						\end{cases}
					\]
					The sign in the above formulas is equal to $-$ if $i = j$, and $+$ otherwise.
					\item $u$ maps the boundary arc labeled by $\kappa_{j}$ to the component $L_{\kappa_{j}}$ of $\overline L$.
				\end{itemize}
				\begin{figure}[H]
					\centering
					\includegraphics[scale=0.9]{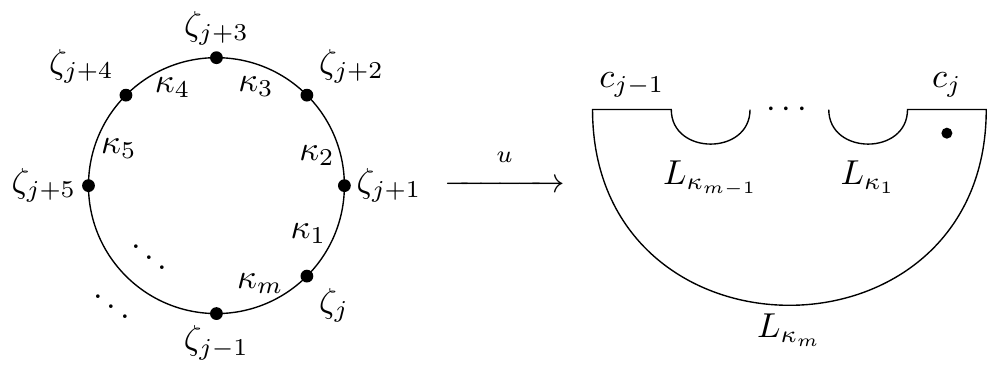}
					\caption{A $J$-holomorphic disk in $\mathcal M^{\text{fi}}(\boldsymbol c;\kappa)$. The dot on the right hand side indicates that $\zeta_j$ (near which, $u$ is asymptotic to $c_j$) is the distinguished puncture.}
				\end{figure}
			\item[Symplectization disks]
				Consider $D_{m}$ equipped with a decreasing boundary numbering (not necessarily strictly decreasing). Let $D_{m,k} \defeq D_m \setmin \parenm{\zeta^{\text{in}}_1,\ldots,\zeta^{\text{in}}_k}$, where each $\zeta^{\text{in}}_i$ is a point in the interior of $D_m$. We equip each $\zeta^{\text{in}}_i$ with a negative cylinder-like end. That is a biholomorphism
				\[
					\varphi_-^i \co (0,\infty) \times S^1	\longrightarrow N(\zeta^{\text{in}}_i)\quad \forall i \in \parenm{1,\ldots,k}\, .
				\]
				 We let $\boldsymbol c = c_{1}^{\sigma_{1}} \cdots c_{m}^{\sigma_{m}}$ be a word of signed Reeb chord generators of $CW^\ast(L,L)$, where $\sigma_{i} \in \parenm{+,-}$ for every $i$. We also let $\boldsymbol \gamma = \gamma_1 \cdots \gamma_k$ be a word of Reeb orbits in $Y$, each of which is equipped with an \emph{asymptotic marker}, i.e.\@ a point $p_i \in \im \gamma_i$. The distinguished boundary puncture $\zeta_j$ induces an asymptotic marker for each interior puncture $\zeta^{\text{in}}_i$, which is a half-ray $\varphi^{-1}((-\infty,0) \times \parenm{x_i})$ near $\zeta_i$ \cite[Section 2.1]{ekholm2017symplectic}. By abuse of notation we say that $x_i \in S^1$ is the asymptotic marker of $\zeta^{\text{in}}_i$. Then we define $\mathcal M^{\text{neg}}(\boldsymbol c, \boldsymbol \gamma; \kappa)$ to be the moduli space of $J$-olomorphic maps $u \co (D_{m,k}, \partial D_{m,k}) \longrightarrow (\R \times Y, \R \times \overline \varLambda)$ such that
				\begin{itemize}
					\item near the boundary puncture $\zeta_{i}$, $u$ is asymptotic to the Reeb chord $c_{i}^{\sigma_i}$ of $\varLambda$ at $\pm \infty$, depending on the sign $\sigma_{i}$, that is
						\[
							\lim_{s\to \pm \infty} u(\varepsilon_{\pm}^i(s,t)) = (\pm \infty,c_i)\, .
						\]
					\item near the interior puncture $\zeta^{\text{in}}_i$, $u$ is asymptotic to the Reeb orbit $\gamma_i$ in $Y$ at $-\infty$ respecting the asymptotic markers, that is
					\[
						\begin{cases}
							\lim_{s\to -\infty} u(\varphi_-^i(s,t)) = (-\infty, \gamma_i) \\
							\lim_{s\to -\infty} u(\varphi_-^i(s,x_i)) = (-\infty, p_i)\, .
						\end{cases}
					\]
					\item $u$ maps the boundary arc labeled by $\kappa_{j}$ to the component $\R \times \varLambda_{\kappa_{j}}$ of $\R \times \overline \varLambda$, and
					\item if $\zeta_{i}$ is a constant puncture, we require $\zeta_{i}$ to be a \emph{negative puncture} (i.e.\@ asymptotic to a Reeb chord of $\varLambda$ at $-\infty$).
				\end{itemize}
				\begin{figure}[H]
					\centering
					\includegraphics[scale=0.9]{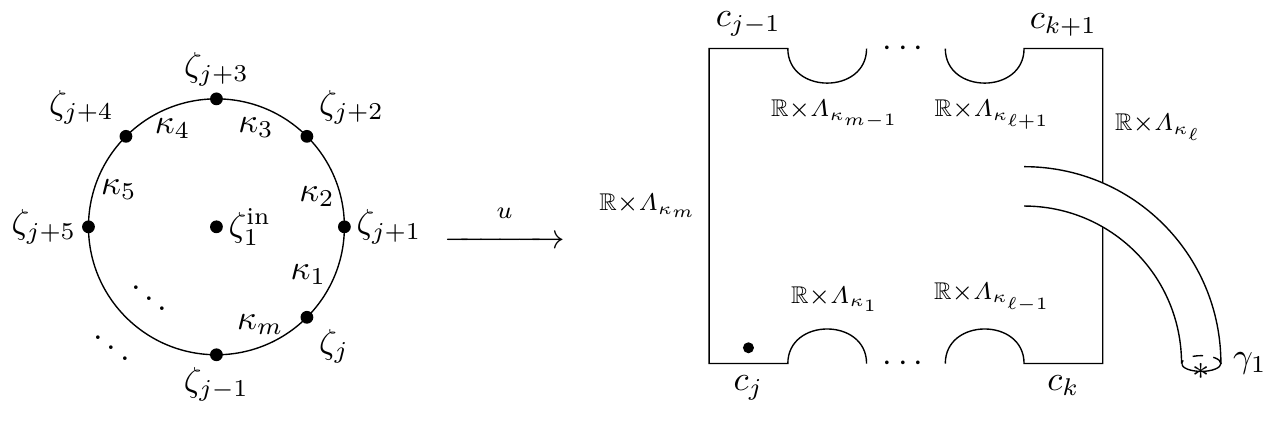}
					\caption{A $J$-holomorphic disk in $\mathcal M^{\text{neg}}(\boldsymbol c, \gamma_1;\kappa)$. The dot on the right hand side indicates that the puncture $\zeta_j$ is the distinguished puncture. The $\ast$ on the right hand side is the asymptotic marker $p_1 \in \im \gamma_1$.} \label{fig:moduli_neg}
				\end{figure}

				Let $\gamma$ be a Reeb orbit in $Y$, equipped with the asymptotic marker $p \in \im \gamma$. Let $S$ denote $S^2$ with one puncture $\zeta\in S^2$, with a fixed choice of asymptotic marker $x$ at $\zeta$. Equip $\zeta$ with a positive cylinder-like end
				\[
					\varphi_+\co (0,\infty) \times S^1 \longrightarrow N(\zeta)\, .
				\]
				Let $\mathcal M^\lambda(\gamma)$ be the $\lambda$-perturbed moduli space of $J_\lambda$-holomorphic maps $u \co S \longrightarrow X$ with notation as in \cite[Theorem 1.1]{ekholm2019holomorphic}, satisfying
				\[
					\begin{cases}
						\lim_{s\to \infty} u(\varphi_+(s,t)) = (\infty, \gamma) \\
						\lim_{s\to \infty} u(\varphi_+(s,x)) = (\infty, p) \, .
					\end{cases}
				\]
				Then we define
				\[
					\mathcal M^{\text{sy}}(\boldsymbol c; \kappa) \defeq \bigcup_{\boldsymbol \gamma} \paren{\mathcal M^{\text{neg}}(\boldsymbol c, \boldsymbol \gamma; \kappa) \times \prod_{\gamma_i \in \boldsymbol \gamma} \mathcal M^{\lambda}(\gamma_i)}\, ,
				\]
				See \cite{ekholm2019holomorphic} and \cite[Appendix A.1]{ekholm2017duality} for more deatils. Each curve in $\mathcal M^{\text{sy}}(\boldsymbol c; \kappa)$ should be interpreted as curves shown in \cref{fig:moduli_neg}, but with all Reeb orbits capped off by punctured $J_\lambda$-holomorphic spheres.
				\begin{figure}[H]
					\centering
					\includegraphics[scale=0.9]{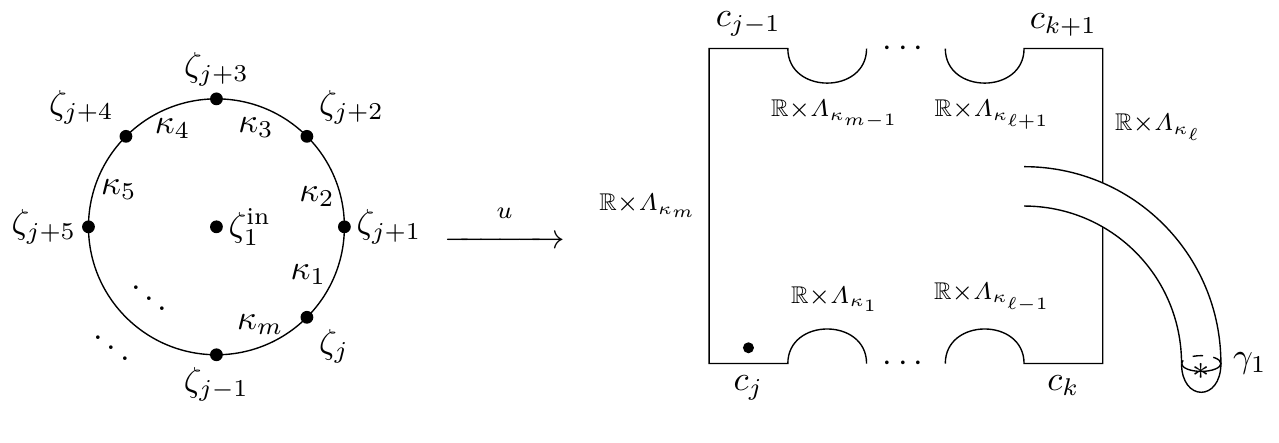}
					\caption{A $J$-holomorphic disk in $\mathcal M^{\text{sy}}(\boldsymbol c;\kappa)$. The dot on the right hand side indicates that the puncture $\zeta_j$ is the distinguished puncture. The $\ast$ on the right hand side is the asymptotic marker $p_1 \in \im \gamma_1$.}
				\end{figure}
			\item[Partial holomorphic buildings]
				The domain of a partial holomorphic building is a possibly broken disk with $m+1$ boundary punctures, see \cref{fig:phb}. We denote this (possibly broken) disk by $D_{m+1}$ and equip it with a decreasing boundary numbering $\kappa$. In the target, the partial holomorphic building consists of a two-level $J$-holomorphic building, with exactly one symplectization disk (called the \emph{primary disk}), and multiple filling disks (called \emph{secondary disks}). We require that the distinguished puncture (which is the only increasing puncture), is a negative puncture of the primary disk. If the primary disk only has one negative puncture, the primary disk is the only component, and the disk is not broken. If the primary disk has more than 1 negative puncture, each additional negative puncture has a secondary disk attached to it, at the distinguished puncture of the secondary disks disk. If $\boldsymbol c = c_{0}c_{1} \cdots c_{m}$ is a word of generators of $CW^{\ast}(L,L)$, where $c_{0}$ is the generator to which the distinguished puncture is asymptotic to, we denote the moduli space of partial holomorphic buildings by $\mathcal M^{\text{pb}}(\boldsymbol c; \kappa)$.
				\begin{figure}[H]
					\centering
					\includegraphics[scale=0.9]{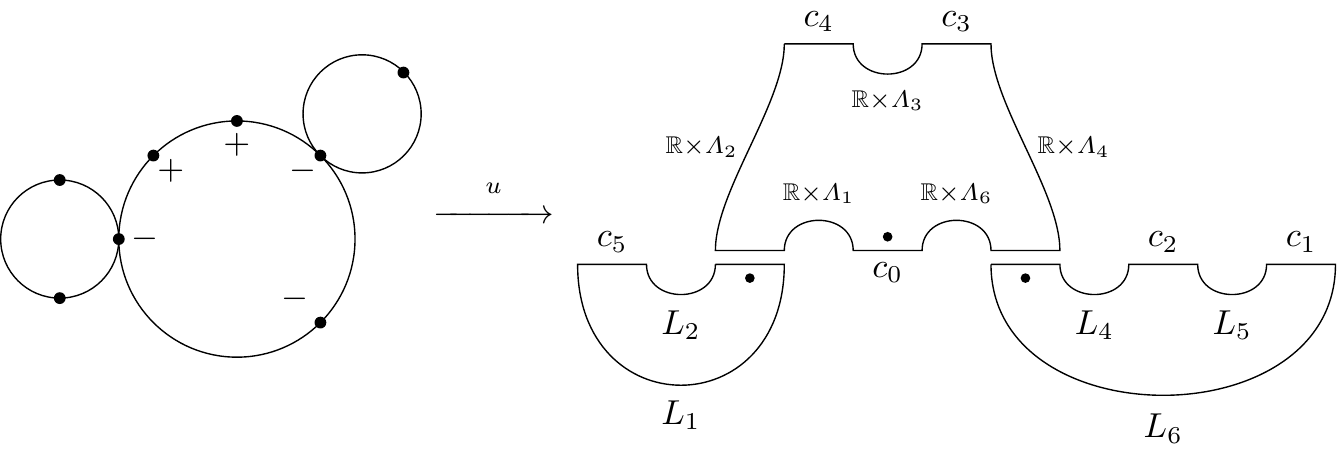}
					\caption{A partial holomorphic building in $\mathcal M^{\text{pb}}(\boldsymbol c;\kappa)$. The dots on the right hand side indicate the distinguished punctures of the corresponding disks. The signs on the left hand side indiciate the sign of the punctures of the primary disk.}
					\label{fig:phb}
				\end{figure}
				\begin{rmk}
					Take note that we might have additional negative punctures of the symplectization disk, at which there are constant filling disks with only 1 positive puncture attached. We have not depicted these above, but they should nonetheless be taken into account.
				\end{rmk}
		\end{description}
		By \cite[Theorem 63,65]{ekholm2017duality} and \cite[Theorem 1.1]{ekholm2019holomorphic}, $\mathcal M^{\mathrm{fi}}(\boldsymbol c; \kappa)$ and $\mathcal M^{\text{sy}}(\boldsymbol c; \kappa)$ are transversely cut out smooth manifolds that are independent of the boundary numbering $\kappa$ up to diffeomorphism. This follows from the observation that disks in $\mathcal M^{\text{fi}}(\boldsymbol c; \kappa)$ or $\mathcal M^{\text{sy}}(\boldsymbol c; \kappa)$ can not be multiply covered for topological reasons. Transversality is then proved using standard techniques as in \cite{ekholm2007legendrian}. Furthermore the moduli spaces admit compactifications that consists of $J$-holomorphic buildings of several levels.
		
		\begin{rmk}\label{rmk:description_of_maslov_index_of_reeb_chord_generators}
			For Reeb chord generators the grading $\abs{a}$ is more explicitly described as follows. Suppose that $a \co [0,\ell] \longrightarrow Y$, then we first define the Conley--Zehnder index $\CZ(a)$ by following \cite[Section 2.2]{ekholm2005contact}. Namely, let $a^-$ and $a^+$ be the start and endpoints of the Reeb chord $a$, respectively. Then pick a capping path $\gamma_c\co [0,1] \longrightarrow \varLambda \subset Y$ so that $\gamma_c(0) = a^+$, $\gamma_c(1) = a^-$. Let $\alpha = \eval[0]{\lambda}_{\dd M}$ and $\xi = \ker \alpha$. Then $T_{a^+}\varLambda \subset \xi_{a^+}$ is a Lagrangian submanifold. By parallel transport along $\gamma_c$ and via the linearized Reeb flow we get a path of Lagrangian submanifolds in the contact planes $\xi \subset T Y$. If we close this path up by \emph{positive close-up} in the contact planes we obtain a loop of Lagrangian submanifolds in $\xi$ denoted by $\varGamma_a$. We then define the Conley--Zehnder index of $a$ to be the Maslov index of $\varGamma_a$ (in the sense of \cite{robbin1993maslov}),
			\[
				\CZ(a) \defeq \mu(\varGamma_a)\, .
			\]
			Then we define
			\[
				\abs{a} = -\CZ(a) + (n-1)\, .
			\]
			For Lagrangian intersection generators $x\in L_{0} \cap L_{1}$ we use the choice of graded lifts of $L_0$ and $L_1$ to obtain a path starting at $T_x L_1$ and ending at $T_x L_0$. We close this path up in $T_x M$ by a positive rotation. This gives a loop of Lagrangian submanifolds denoted by $\varGamma_x$, which starts and ends at $T_x L_0 \subset T_x M$. Then define the grading of $x$ as the Maslov index of this loop \cite[p.~89]{ekholm2017duality} \cite[Appendix A]{cieliebak2010compactness}
			\[
				\abs x \defeq \mu(\varGamma_x)\, .
			\]
		\end{rmk}
		The dimension of the moduli space $\mathcal M^\text{fi}(\boldsymbol a; \kappa)$ is dependent on whether the distinguished puncture is a Reeb chord or a Lagrangian intersection puncture. To emphasize the differences, we introduce some more notation.
		\begin{itemize}
			\item If the distinguished puncture is a Reeb chord generator we denote it by $\mathcal M^{\text{fi,Reeb}}(\boldsymbol a; \kappa)$, and
			\item if the distinguished puncture is an intersection generator we denote it by $\mathcal M^{\text{fi,Lag}}(\boldsymbol a; \kappa)$.
		\end{itemize}
		\begin{thm}\label{thm:dimension_of_various_moduli_spaces}
			Let $\boldsymbol a = ca_{2} \cdots a_{m}$ be a word of generators of $CW^\ast(L,L)$. Assume that $c$ is the distinguished puncture and that it is a Reeb chord generator. Then the dimension of the moduli space $\mathcal M^{\mathrm{fi,Reeb}}(\boldsymbol a; \kappa)$ is
			\[
				\dim \paren{\mathcal M^{\mathrm{fi,Reeb}}(\boldsymbol a; \kappa)} = (n-3) + m - \abs c - \sum_{j=2}^{m} \abs{a_{j}}\, .
			\]
			Let $\boldsymbol a = xa_{2} \cdots a_{m}$ be a word of generators of $CW^\ast(L,L)$. Assume that $x$ is the distinguished puncture and that it is a Lagrangian intersection generator. Then the dimension of the moduli space $\mathcal M^{\mathrm{fi,Lag}}(\boldsymbol a; \kappa)$ is
			\[
				\dim \paren{\mathcal M^{\mathrm{fi,Lag}}(\boldsymbol a; \kappa)} = -3 + m - \abs x - \sum_{j=2}^{m} \abs{a_{j}}\, .
			\]
			For any word of Reeb chord generators $\boldsymbol c = c_{1} \cdots c_{m}$, the dimension of the moduli space $\mathcal M^{\mathrm{sy}}(\boldsymbol c; \kappa)$ is
			\[
				\dim \paren{\mathcal M^{\mathrm{sy}}(\boldsymbol c; \kappa)} = (n-3) + m + \sum_{\sigma_{j} = -} (\abs{c_{j}} - (n-1)) - \sum_{\sigma_{j} = +} \abs{c_{j}}\, .
			\]
			For any word of Reeb chord generators $\boldsymbol c = c_{0} \cdots c_{m}$, the dimension of the moduli space $\mathcal M^{\mathrm{pb}}(\boldsymbol c; \kappa)$ is
			\[
				\dim \paren{\mathcal M^{\mathrm{pb}}(\boldsymbol c; \kappa)} = -1 + m + \abs{c_{0}} - \sum_{j=1}^{m}\abs{c_{j}}\, .
			\]
		\end{thm}
		\begin{proof}
			The theorem follows from applying \cite[Theorem A.1]{cieliebak2010compactness}, and the fact that the index of a several-level $J$-holomorphic building is the sum of the indices of the disks at each level. Let $\boldsymbol a = a_1 \cdots a_m$ be a word of generators of $CW^\ast(L,L)$. Let either $u\in \mathcal M^{\text{sy}}(\boldsymbol a; \kappa)$ or $u\in \mathcal M^{\text{fi}}(\boldsymbol a; \kappa)$. Let $\widehat D_m$ be the unit disk in $\CC$ together with $\zeta_1,\ldots,\zeta_m$ regarded as marked points (and not punctures). The boundary of $\widehat D_m$ is equal to the union of closed boundary arcs $C$ such that the interiors of all the boundary arcs $C$ are pairwise disjoint, and only missing the marked points $\parenm{\zeta_1,\ldots,\zeta_m}$.
			\begin{enumerate}[(T1)]
				\item For all Reeb chord generators $a_i \in \boldsymbol a$, fix a complex trivialization $Z_{a_i}$ of the contact structure $\xi$ along $a_i$, such that the linearized Reeb flow along the chord $a_i$ expressed in $Z_{a_i}$ is constantly equal to the identity.
				\item For each boundary arc $C$ in $\widehat D_m$, fix a complex trivialization $Z_C$ of $u^\ast TM$ (if $u\in \mathcal M^{\text{fi}}(\boldsymbol a; \kappa)$) or $u^\ast T(\R \times Y)$ (if $u\in \mathcal M^{\text{sy}}(\boldsymbol a; \kappa)$) with the following properties:
				\begin{enumerate}
					\item If an endpoint of $C$ is a puncture $\zeta_i$ asymptotic to a Reeb chord $a_i$, then $Z_C = Z_{a_i}$.
					\item If an endpoint of $C$ is a puncture $\zeta_i$ asymptotic to an intersection generator $x_i \in L_{\kappa_i} \cap L_{\kappa_{i+1}}$, then $Z_C = Z_{C'}$ where $\zeta_i$ is the common endpoint of the boundary arcs $C$ and $C'$.
				\end{enumerate}
			\end{enumerate} 
			Items (T1) and (T2) above give a complex trivialization $Z_{\partial_ju}$ of $u^\ast TM$ (or $u^\ast T(\R \times Y)$) over the $j^{\text{th}}$ boundary arc $C_j$ of $\widehat D_m$. For each boundary arc $C_j$, let $C'_j$ be the complement of its endpoints in $C_j$. The tangent planes of $L$ along all $f(C'_j)$ expressed in the trivialization $Z_{\partial_j u}$ gives a collection of paths of Lagrangian subspaces in $\CC^n$. We close up this path to a loop as follows. For each Reeb chord $a_i \in \boldsymbol a$, denote its start and endpoints by $a_i^\pm$ respectively.
			\begin{enumerate}[(C1)]
				\item For each positive puncture $\zeta_i$ near which $u$ is asymptotic to the Reeb chord $a_i$, the tangent planes of $L = \R \times \varLambda$ are connected by the product of the linearized Reeb flow along $a_i$ in $\xi$, and the identity in the $\R$-factor, followed by negative close-up in the contact plane in $\xi_{a_i^+} \times \CC$ (cf.\@ \cref{rmk:description_of_maslov_index_of_reeb_chord_generators}). Denote this path of Lagrangian subspaces by $g_{a_{i}}^{+}$.
				\item For each negative puncture $\zeta_i$ near which $u$ is asymptotic to the Reeb chord $a_i$, the tangent planes of $L = \R \times \varLambda$ are connected by the product of the backwards linearized Reeb flow along $a_i$ in $\xi$, and the identity in the $\R$-factor, followed by negative close-up in the contact plane in $\xi_{a_i^-} \times \CC$ (cf.\@ \cref{rmk:description_of_maslov_index_of_reeb_chord_generators}). Denote this path of Lagrangian subspaces by $g_{a_{i}}^{-}$.
				\item For each puncture $\zeta_i$ near which $u$ is asymptotic to the intersection generator $x_i \in L_{\kappa_i} \cap L_{\kappa_{i+1}}$, connect the planes $T_{x_i} L_{\kappa_i}$ and $T_{x_i} L_{\kappa_{i+1}}$ by a negative rotation taking $T_{x_i}L_{\kappa_i}$ to $T_{x_i}L_{\kappa_{i+1}}$ in $\CC^n$ (cf.\@ \cref{rmk:description_of_maslov_index_of_reeb_chord_generators} and \cite[Remark A.1]{cieliebak2010compactness}). Denote this path of Lagrangian subspaces by $g_{x_{i}}^{\cap}$.
			\end{enumerate}
			Define $\mu(\partial u, Z_{\partial u})$ to be the Maslov index of the loop of Lagrangian subspaces in $\CC^n$ which is constructed by closing up paths of Lagrangian subspaces as described in (C1), (C2) and (C3). For the moduli spaces of filling disks and symplectization disks, we then have by \cite[Theorem A.1.]{cieliebak2010compactness} that
			\begin{align*}
				\dim \paren{\mathcal M^{\text{fi}}(\boldsymbol a; \kappa)} &= (n-3) + m + \mu(\dd u, Z_{\dd u})\\
				\dim \paren{\mathcal M^{\text{sy}}(\boldsymbol a; \kappa)} &= (n-3) + m + \mu(\dd u, Z_{\dd u})\, .
			\end{align*}
			Since $L$ is assumed to have vanishing Maslov class, the contribution to $\mu(\partial u, Z_{\partial u})$ is equal to the sum of each contribution at every boundary puncture of $D_{m}$. Next we describe each of these contributions in terms of the grading of each generator. First let $u \in \mathcal M^{\text{sy}}(\boldsymbol a; \kappa)$.
			\begin{enumerate}[(sy1)]
				\item If $\zeta_{i}$ is a positive puncture near which $u$ is asymptotic to the Reeb chord $a_{i}$ then
				\[
					\mu(g_{a_{i}}^{+} \circ (\varGamma_{a_{i}})^{-1}) = -(n-1) \Leftrightarrow \mu(g_{a_{i}}^{+}) = \mu(\varGamma_{a_{i}}) - (n-1) = -\abs{a_{i}}\, .
				\]
				\item If $\zeta_{i}$ is a negative puncture near which $u$ is asymptotic to the Reeb chord $a_{i}$ then
				\[
					\mu(g_{a_{i}}^{-} \circ \varGamma_{a_{i}}) = 0 \Leftrightarrow \mu(g_{a_{i}}^{-}) = -\mu(\varGamma_{a_{i}}) = \abs{a_{i}} - (n-1)\, .
				\]
			\end{enumerate}
			Then let $u \in \mathcal M^{\text{fi}}(\boldsymbol a; \kappa)$.
			\begin{enumerate}[({fi}1)]
				\item Let $\zeta_{i}$ be a puncture near which $u$ is asymptotic to the Reeb chord $a_{i}$.
				\begin{enumerate}
					\item If $\zeta_{i}$ is the distinguished puncture then
					\[
						\mu(g_{a_{i}}^{+} \circ \varGamma_{a_{i}}) = 0 \Leftrightarrow \mu(g_{a_{i}}^{+}) = -\mu(\varGamma_{a_{i}}) = \abs{a_{i}} - (n-1)\, .
					\]
					\item If $\zeta_{i}$ is not the distinguished puncture then
					\[
						\mu(g_{a_{i}}^{+} \circ (\varGamma_{a_{i}})^{-1}) = -(n-1) \Leftrightarrow \mu(g_{a_{i}}^{+}) = \mu(\varGamma_{a_{i}}) - (n-1) = -\abs{a_{i}}\, .
					\]
				\end{enumerate}
				\item Let $\zeta_{i}$ be a puncture near which $u$ is asymptotic to the intersection generator $x_{i}$
				\begin{enumerate}
					\item If $\zeta_{i}$ is the distinguished puncture then
					\[
						\mu(g_{x_{i}}^{\cap} \circ (\varGamma_{x_{i}})^{-1}) = -n \Leftrightarrow \mu(g_{x_{i}}^{\cap}) = \mu(\varGamma_{x_{i}})-n = \abs{x_{i}}-n\, .
					\]
					\item If $\zeta_{i}$ is not the distinguished puncture then
					\[
						\mu(g_{x_{i}}^{\cap} \circ \varGamma_{x_{i}}) = 0 \Leftrightarrow \mu(g_{x_{i}}^{\cap}) = -\mu(\varGamma_{x_{i}}) = -\abs{x_{i}}\, .
					\]
				\end{enumerate}
			\end{enumerate}

			From (sy1) and (sy2) we obtain
			\begin{align*}
				\dim \paren{\mathcal M^{\text{sy}}(\boldsymbol c; \kappa)} &= (n-3) + m + \sum_{\sigma_j = -} \paren{\abs{c_j} - (n-1)} - \sum_{\sigma_j = +} \abs{c_j} \, .
			\end{align*}
			From (fi1)(a), (fi1)(b) and (fi2)(b) we obtain
			\[
				\dim \paren{\mathcal M^{\text{fi,Reeb}}(\boldsymbol a; \kappa)} = (n-3) + m - \abs c - \sum_{j=2}^m \abs{a_j}\, .
			\]
			From (fi2)(a), (fi1)(b) and (fi2)(b) we obtain
			\[
				\dim \paren{\mathcal M^{\text{fi,Lag}}(\boldsymbol a; \kappa)} = (n-3) + m + (\abs x - n)-\sum_{j=2}^m \abs{a_j} =  - 3 + m + \abs x - \sum_{j=2}^m \abs{a_j}\, .
			\]
			For a partial holomorphic building, let $a_1, \ldots, a_p$ be the positive punctures of the primary disk, let $d_0$ be the distinguished negative puncture of the primary disk and let $d_1,\ldots,d_q$ be the remaining negative punctures. Let $b_1,\ldots,b_r$ be all the non-distinguished punctures of all the secondary disks, see \cref{fig:phb_dim}. Each secondary disk lies in $\mathcal M^{\text{fi,Reeb}}(\boldsymbol a; \kappa)$.
			We may then compute the dimension by taking sums, that is
		\begin{align*}
			\dim \paren{\mathcal M^{\text{pb}}(\boldsymbol c; \kappa)} &= \parenb{(n-3) - \sum_{j=1}^p \paren{\abs{a_j}-1} + \sum_{j=0}^q \paren{\abs{d_j}-(n-2)}}\\
			&\qquad + \parenb{\sum_{j=1}^q \paren{(n-3) - \paren{\abs{d_j}-1}} + \sum_{j=1}^r -\paren{\abs{b_j}-1}} \, .
		\end{align*}
		\begin{figure}[H]
			\centering
			\includegraphics{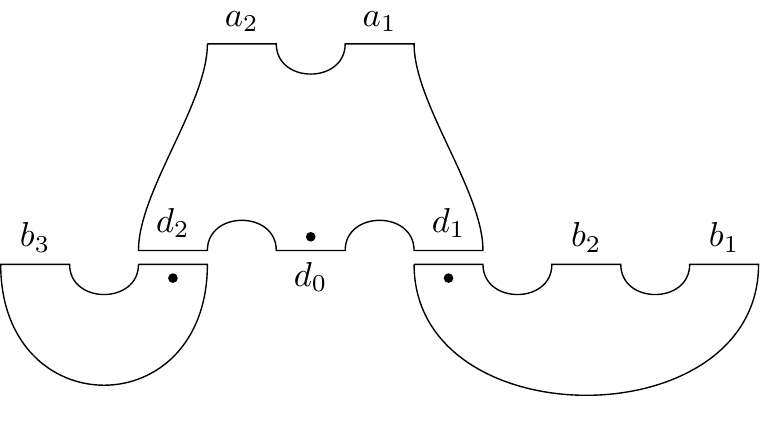}
			\caption{}
			\label{fig:phb_dim}
		\end{figure}
		After canceling we get
		\begin{align*}
			\dim \paren{\mathcal M^{\text{pb}}(\boldsymbol c; \kappa)} &= (n-3) - \sum_{j=1}^p \paren{\abs{a_j}-1} + \paren{\abs{d_0} - (n-2)} - \sum_{j=1}^r \paren{\abs{b_j}-1} \\
			&= -1 + (p+r) + \abs{d_0} - \sum_{j=1}^p \abs{a_j} - \sum_{j=1}^r \abs{b_j}\, .
		\end{align*}
		Now let $c_0 \defeq d_0$ and let $\boldsymbol c$ be the word of $m\defeq p+r$ letters corresponding to all the generators $a_1,\ldots,a_p,b_1,\ldots,b_r$ in the appropriate order. Therefore
		\[
			\dim \paren{\mathcal M^{\text{pb}}(\boldsymbol c; \kappa)} = m-1 + \abs{c_0}- \sum_{j=1}^m \abs{c_j}\, .
		\]
		\end{proof}
		We now define operations, one for each $i\geq 1$,
		\[
			\mu^i \co CW^{\ast}(L_{\kappa_{i-1}}, L_{\kappa_{i}}) \otimes \cdots \otimes CW^{\ast}(L_{\kappa_{1}}, L_{\kappa_{2}}) \longrightarrow CW^{\ast}(L_{\kappa_{1}}, L_{\kappa_{i}})\, ,
		\]
		that counts various $J$-holomorphic disks discussed above. We split it as a sum $\mu^i = \mu^i_{\text{Lag}} + \mu^i_{\text{Reeb}}$, where $\mu^i_{\text{Lag}}$ takes values in Lagrangian intersection generators and $\mu^i_{\text{Reeb}}$ takes values in Reeb chord generators.
		
		First we consider $\mu^i_{\text{Lag}}$. Let $\boldsymbol c' = c_1 \cdots c_i$ be a word of generators of $CW^\ast(L,L)$. Then
		\[
			\mu^i_{\text{Lag}}(c_i \otimes \cdots \otimes c_1)\defeq \sum_{\abs{c_0} = \abs{\boldsymbol c'} + (2-i)} \abs{\mathcal M^{\text{fi,Lag}}(c_0 \boldsymbol c'; \kappa)} c_0\, .
		\]
		The sum is taken over all Lagrangian intersection generators $c_0$ so that $\dim \paren{\mathcal M^{\text{fi,Lag}}(c_0 \boldsymbol c'; \kappa)} = 0$.
		
		To define $\mu^i_{\text{Reeb}}$, consider a word of generators $\boldsymbol c' = c_1 \cdots c_i$. Then
		\[
			\mu^i_{\text{Reeb}}(c_i \otimes \cdots \otimes c_1) \defeq \sum_{\abs{c_0} = \abs{\boldsymbol c'} + (2-i)} \abs{\mathcal M^{\text{pb}}(c_0 \boldsymbol c'; \kappa)} c_0\, .
		\]
		The sum is taken over all Reeb chords $c_0$ so that $\dim \paren{\mathcal M^{\text{pb}}(c_0 \boldsymbol c'; \kappa)} = 0$. The total operation $\mu^i$ is then defined as
		\begin{equation}\label{eq:def_of_a_infty_operations_wo_ham}
			\mu^i(c_i \otimes \cdots \otimes c_1) = (-1)^{\diamond}\paren{\mu^i_{\text{Lag}}(c_i \otimes \cdots \otimes c_1)+\mu^i_{\text{Reeb}}(c_i \otimes \cdots \otimes c_1)}
		\end{equation}
		where
		\[
			\diamond = \sum_{j=1}^i j \abs{c_j}\, .
		\]
		\begin{lma}
			With the sign conventions as in~\cite{seidel2008fukaya}, $\paren{CW^{\ast}(L,L), \parenm{\mu^i}_{i=1}^\infty}$ forms an $A_{\infty}$-algebra, that is
			\[
				\sum_{\substack{d_1+d_2 = d+1 \\ 0 \leq k < d_1}} (-1)^{\maltese_k} \mu^{d_1}(c_d,\ldots,c_{k+d_2+1}, \mu^{d_2}(c_{k+d_2},\ldots,c_{k+1}),c_k,\ldots,c_1) = 0\, ,
			\]
			where
			\[
				\maltese_k = k + \sum_{j=1}^k \abs{c_j}\, .
			\]
		\end{lma}
		\begin{proof}
			See \cite[Lemma 67]{ekholm2017duality}.
		\end{proof}
\section{Partially wrapped Floer cohomology and chains of based loops}
\label{sec:partially_wrapped_floer_cohomology_and_chains_of_based_loops}
	Let $S$ be any closed orientable spin manifold and $K \subset S$ any submanifold. The purpose of this section is to describe the surgery approach to compute the partially wrapped Floer cohomology of a cotangent fiber in the Weinstein domain $(DT^\ast S, \lambda = p dq)$ stopped by the unit conormal $\varLambda_K$. We then define a chain map relating the partially wrapped Floer cohomology of a fiber to chains of based loops on a Lagrangian submanifold $M_K$ that is diffeomorphic to the complement $S \setmin K$.

	In \cref{sec:partially_wrapped_floer_cohomology_and_chains_of_based_loops} we describe the surgery approach in more detail, and also construct the Lagrangian $M_K$. In \cref{sub:based_loops_on_} we describe the model we use for the chains of based loops on $M_K$, and equip it with the Pontryagin product. Then in \cref{sub:moduli_space_of_half_strips} we describe the moduli space of half strips which we need in order to to define an $A_{\infty}$-homomorphism between the partially wrapped Floer cocomplex and the chains of based loops on $M_K$. The construction of the $A_{\infty}$-homomorphism is carried out in \cref{sub:the_evaluation_map_etc}.
	\subsection{Partially wrapped Floer cohomology using a surgery approach}\label{sec:partially_wrapped_floer_cohomology_using_a_surgery_approach}
		Following~\cite[Appendix B]{ekholm2017duality} and~\cite[Section 6]{ekholm2016completearXiv} we will now describe the surgery approach. We consider the disk cotangent bundle $DT^\ast S$ the conormal bundle of $K$
		\[
			L_K = \parenm{(q,p) \in DT^{\ast} S \suchthat q \in K, \, \ip{p}{T_{q}K} = 0}\, .
		\] 
		Let $\varLambda_K \defeq L_K \cap ST^\ast S$ be the unit conormal of $K$. We take a tubular neighborhood $U$ of $\varLambda_K$ in $ST^\ast S$ and we attach a handle modeled on $D_{\varepsilon}T^{\ast}([0,\infty) \times \varLambda)$ to $U$. After handle attachment and after smoothing out corners, the Liouville vector field is equal to $p\dd_p$ in $D_{\varepsilon}T^{\ast}([T,\infty) \times \varLambda)$ (for $T \geq 0$ large enough) for coordinates $(q,p)$ in the handle. We call the resulting manifold $W_K$, see \cref{fig:wk}.

		We then consider a cotangent fiber $F \cong DT^\ast_\xi S$ at $\xi \in M_K$ in $W_K$. Denote the wrapped Floer cochains of $F$ in $W_K$ as described in \cref{sec:wrapped_floer_cohomology_wo_ham} by $CW^\ast_{\varLambda_K}(F,F)$.
		\begin{rmk}
			In the language of Sylvan~\cite{sylvan2016partially}, we obtain a \emph{stop} $\sigma_{\varLambda_K}$ from $\varLambda_K$ as follows. Pick a tubular neighborhood $U \supset \varLambda_K$ in $ST^\ast S$, and a strict contactomorphism $\varphi\co (U, \eval[0]{\lambda}_{U}) \longrightarrow (V, dz- ydx)$ where $V$ is a tubular neighborhood of $\varLambda_K \subset J^{1}(\varLambda_K) = T^{\ast} \varLambda_K \times \R$, viewed as the zero section. Then the Liouville hypersurface $\sigma_{\varLambda_K} \defeq \varphi^{-1}(T^{\ast}\varLambda_K \cap V) \subset U$ is a stop.

			Another point of view, is to remove the tubular neighborhood $U$ from $ST^\ast S$, and take the Liouville completion of $\paren{DT^\ast S} \setmin U$ to obtain a Liouville sector as defined~\cite{ganatra2017covariantly}. The wrapped Fukaya category of this Liouville sector coincides with the wrapped Fukaya category associated to the pair $(M, \sigma_{\varLambda_K})$, and also with the Fukaya category associated to $W_K$ \cite{ekholm2017duality,ganatra2017covariantly,ganatra2018structural}.
		\end{rmk}
		\begin{figure}[H]
			\centering
			\includegraphics{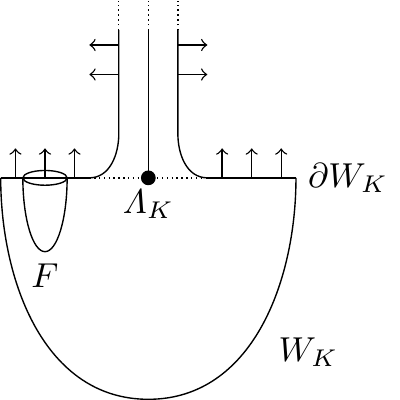}
			\caption{The Liouville sector $W_K$.}
			\label{fig:wk}
		\end{figure}
		To construct the complement Lagrangian $M_K$, we perform Lagrangian surgery of $L_K$ and $S$ which intersect cleanly along $K$. Above each point of $K$, the intersection $L_K \cap S$ looks like the transverse intersection of two Lagrangian disks of dimension $k$. We perform Lagrangian surgery along $K$ as in~\cite[Section 2.2.2]{mak2018dehn} \cite{aganagic2014topological}. We denote the result of the surgery by $M_K \cong S \setmin K$ (cf.\@ \cite{aganagic2014topological}).

		\begin{rmk}\label{rmk:maslov_class_of_complement}
			Note that the Maslov class of $M_K$ vanishes, because it is the result of surgery of $S \subset W$ and $L_K \subset W$, both of which have vanishing Maslov class. In particular, consider the following model. We pick a $\CC^n$-neighborhood around $p \in K$ such that $L_K = i \R^n$ and $S = \R^n$. Following the discussion in \cite[Section 2.2]{ekholm2016exact}, we have a phase function $\phi\co H \longrightarrow \R$ which is unique up to an additive constant on the handle $H$, so that $\eval[0]{\phi}_{S \cap H} = 0$ and $\eval[0]{\phi}_{L_K \cap H} = n-1$.

			Any loop that is based at any point outside of the handle pass through the entire handle an even number of times, which means that the total Maslov index of the loop is zero.
		\end{rmk}
	\subsection{Based loops on $M_{K}$}
	\label{sub:based_loops_on_}
		Consider the Moore loop space of $M_K$, based at $\xi$
		\[
			\varOmega_{\xi} M_{K} = \parenm{\gamma\co [0,R]\longrightarrow M_{K} \suchthat \gamma(0) = \gamma(R) = \xi}\, .
		\]
		We use a cubical model for chains of based loops as in \cite{abouzaid2012wrapped, ekholm2017duality}.

		A \emph{singular $k$-cube} is a smooth map $\sigma\co [0,1]^{k} \longrightarrow \varOmega_{\xi} M_{K}$ and it is called \emph{degenerate} if $\sigma(x_{1},\ldots,x_{k})$ is constant in at least one of the coordinates. We define the space of cubical $k$-chains by
		\[
			C_{k}(\varOmega_{\xi}M_{K}) = \frac{\Z[\text{singular $k$-cubes}]}{\Z[\text{degenerate singular $k$-cubes}]}\, .
		\]
		We equip $C_{\ast}(\varOmega_{\xi}M_K)$ with the differential
		\begin{equation}\label{eq:differential_on_cubical_chains}
			\partial \sigma \defeq \sum_{i=1}^k \sum_{\varepsilon=0}^1 (-1)^{i+\varepsilon} \sigma(\delta_{i,\varepsilon}(x_1,\ldots,x_k))\, ,
		\end{equation}
		where
		\[
			\delta_{i,\varepsilon}(x_1,\ldots,x_k) = (x_1,\ldots,x_{i-1},\varepsilon,x_{i+1},\ldots,x_{k}), \; \varepsilon \in \parenm{0,1}
		\]
		is the map that replaces the $i$-th coordinate with $\varepsilon$.
		
		The Pontryagin product $P$ is defined as the following composition:
		\begin{equation}\label{eq:def_pontryagin_product}
			\begin{tikzcd}[row sep=tiny, column sep=scriptsize]
				C_{k}(\varOmega_{\xi}M_{K}) \otimes C_{\ell}(\varOmega_{\xi}M_{K}) \rar & C_{k+\ell}\paren{\paren{\varOmega_{\xi}M_{K}}^{2}} \rar & C_{k+\ell}(\varOmega_{\xi}M_{K}) \\
				\sigma_2 \otimes \sigma_1 \rar[mapsto] & (-1)^{\abs{\sigma_1}} \sigma_2 \times \sigma_1 \rar[mapsto] & (-1)^{\abs{\sigma_1}} \sigma_1 \circ \sigma_2
			\end{tikzcd}
		\end{equation}
		The cross product of a singular $i$-cube $\sigma_1$ and a $j$-cube $\sigma_2$ is the $(i+j)$-cube
		\begin{align*}
			\sigma_{1} \times \sigma_{2} \co [0,1]^{i+j} &\longrightarrow \varOmega_{\xi}M_{K} \times \varOmega_{\xi}M_{K} \\
			(x_1,\ldots,x_{i+j}) &\longmapsto (\sigma_1(x_1,\ldots,x_i), \sigma_2(x_{i+1},\ldots,x_{i+j}))\, .
		\end{align*}
		The map $\circ$ is pointwise concatenation of loops where we first follow $\sigma_1(x_1,\ldots,x_i)$, and then $\sigma_2(x_{i+1},\ldots,x_{i+j})$. That is $(\sigma_1 \circ \sigma_2)(x) = \sigma_1(x_1,\ldots,x_i) \circ \sigma_2(x_{i+1},\ldots,x_{i+j})$, where
		\[
			(\sigma_1(x_1,\ldots,x_i) \circ \sigma_2(x_{i+1},\ldots,x_{i+j}))(t) \defeq \begin{cases}
				\sigma_1(x_1,\ldots,x_i)(t), & t\in [0,R_1] \\
				\sigma_2(x_{i+1},\ldots,x_{i+j})(t-R_1), & t \in [R_1,R_1+R_2]
			\end{cases}\, .
		\]
		From the definitions of $P$ and $\dd$ we see that for any two singular cubes $\sigma_1 \in C_{k}(\varOmega_\xi M_K)$ and $\sigma_2 \in C_{\ell}(\varOmega_\xi M_K)$ we have
		\[
			\dd(\sigma_1 \circ \sigma_2) = (-1)^k (\sigma_1 \circ \dd \sigma_2) + \dd \sigma_1 \circ \sigma_2\, .
		\]
		This leads via \eqref{eq:def_pontryagin_product} to 
		\[
			\dd P(\sigma_2 \otimes \sigma_1) + P(\sigma_2 \otimes \dd \sigma_1) + (-1)^{k+1} P(\dd \sigma_2 \otimes \sigma_1) = 0\, .
		\]
		Hence $(C_{\ast}(\varOmega_{\xi}M_{K}), \dd, P)$ is an $A_{\infty}$-algebra with all higher operations being zero with sign conventions as in \cite{abouzaid2010open,seidel2008fukaya}.
	\subsection{Moduli space of half strips}
	\label{sub:moduli_space_of_half_strips}
		Consider the cotangent fiber $F \cong T^\ast_{\xi}S \subset W_K$ at $\xi\in M_K$ defined in \cref{sec:partially_wrapped_floer_cohomology_using_a_surgery_approach} and consider a system of parallel copies of $F$ as in \cref{sec:wrapped_floer_cohomology_wo_ham}. In this section we construct a moduli spaces of $J$-holomorphic half strips similar to \cite{abouzaid2012wrapped}. This moduli space is used to define a chain map between $CW^\ast_{\varLambda_K}(F,F)$ and $C_{-\ast}(\varOmega_\xi M_K)$. By non-compactness of $W_K$ in the horizontal direction, we use monotonicity for $J$-holomorphic half strips to establish compactness of moduli spaces, see \cref{sec:monotonicity_estimates} for details.

		Let $D_{3} \subset \CC$ be the positively oriented unit disk with three boundary punctures $\zeta_{+}, \zeta_-, \zeta_{1}$. Then $D_3$ is biholomorphic to
		\[
			T \defeq ([0,\infty) \times [0,1]) \setmin \parenm{\zeta_+, \zeta_-} \subset \CC\, ,
		\]
		where $\zeta_+ = (0,1) \in \CC$ and $\zeta_- = (0,0) \in \CC$. The boundary segment between $\zeta_+$ and $\zeta_-$ is called the \emph{outgoing segment}.
		\begin{figure}[H]
			\centering
			\includegraphics{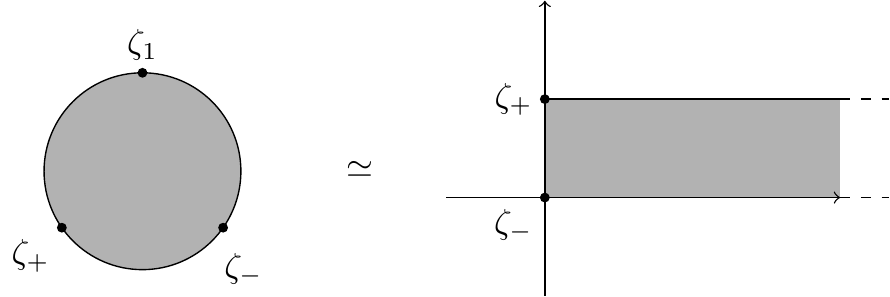}
		\end{figure}
		Define
		\[
			\begin{cases}
				Z_- = (-\infty,0) \times [0,1] \subset \CC\\
				Z_+ = (0,\infty) \times [0,1] \subset \CC\, ,
			\end{cases}
		\]
		equipped with the standard complex structure $j$ on $\CC$. We pick a positive strip-like end $\varepsilon^{+}$ near $\zeta_{+}$, and a negative strip-like end $\varepsilon_-$ near $\zeta_-$. That is, $\varepsilon_\pm$ are maps
		\begin{align*}
			\varepsilon_{+}\co Z_+ \longrightarrow T \\
			\varepsilon_- \co Z_- \longrightarrow T
		\end{align*}
		defined in neighborhoods of $\zeta_+$ and $\zeta_-$ respectively. Fix a family $\parenm{J_t}_{t\in [0,1]} \subset \mathcal J(W_k, \omega)$ of $\omega$-compatible almost complex structures, parametrized by $t\in [0,1]$. Then consider a map
		\[
			J_T \co T \longrightarrow \mathcal J(W_K, \omega)
		\]
		which satisfies
		\[
			\begin{cases}
				J_T(s,t) = J_t, & s > N \text{ for some } N > 0\\
				(\varepsilon_-)^\ast J_T = J_t, & \text{near }\zeta_- \\
				(\varepsilon_{+})^\ast J_T = J_t, & \text{near }\zeta_+\, .
			\end{cases}
		\]
		Given a generator $a\in CW^\ast_{\varLambda_K}(F,F)$ we consider maps
		\[
			u\co T \longrightarrow W_K
		\]
		that satisfies the following Floer equation:
		\begin{equation}\label{eq:floer_eqn_half_strips}
			\begin{cases}
				du + J_T \circ du \circ j = 0\\
				\lim_{s\to \infty} u(s,t) = a(t), &\forall t\in [0,1] \\
				\lim_{s\to \infty} u(\varepsilon_{+}(s,t)) = \xi, & \forall t \in [0,1] \\
				\lim_{s\to -\infty} u(\varepsilon_{-}(s,t)) = \xi, & \forall t \in [0,1]
			\end{cases}
		\end{equation}
		where the boundary conditions on $u$ is indicated in \cref{fig:pseudo_hol_disk_in_one_positive_punc} below.
		\begin{figure}[H]
			\centering
			\includegraphics{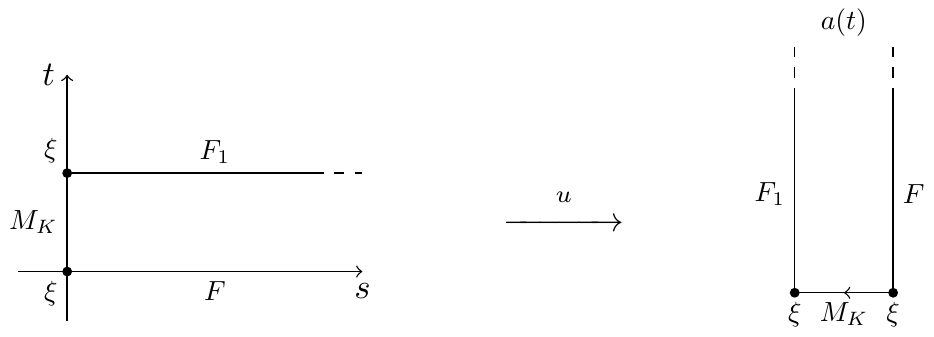}
			\caption{A $J_T$-holomorphic disk in $\mathcal M(a)$.}
			\label{fig:pseudo_hol_disk_in_one_positive_punc}
		\end{figure}
		For a generator $a\in CW^\ast_{\varLambda_K}(F,F)$ we define $\mathcal M(a)$ to be the moduli space of $J_T$-holomorphic maps $u\co T \longrightarrow W_K$ that satisfies \eqref{eq:floer_eqn_half_strips}.

		Analogous to \cite[Theorem 63]{ekholm2017duality} and \cite[Lemma 4.2]{abouzaid2012wrapped} we have
		\begin{lma}
			For generic choices of almost complex structure $J_T$, the moduli space $\mathcal M(a)$ is a smooth orientable manifold of dimension
			\[
				\dim \mathcal M(a) = - \abs a\, .
			\]
		\end{lma}
		\begin{proof}
			See the proof of \cref{lma:transversality_of_moduli_space_of_half_strips} for the proof of the statement about the dimension. Note that, because we work with a system of parallel copies of $F$, $J_T$-holomorphic curves can not be multiply covered, and transversality for such is achieved using standard methods as in \cite{ekholm2007legendrian,ekholm2017duality}.
		\end{proof}

		Let $D_{m+2} \subset \CC$ be the positively oriented unit disk with $m+2$ boundary punctures which we denote by $\zeta_{-}, \zeta_{1},\ldots, \zeta_{m}, \zeta_+$. Let $\mathcal R_m$ be the Deligne--Mumford space of unit disks in the complex plane with $m+1$ boundary punctures that are oriented counterclockwise. Let $\overline{\mathcal R}_m$ denote the Deligne--Mumford compactification of $\mathcal R_m$ as in~\cite[Section C.1]{abouzaid2010geometric} and~\cite[Section (9f)]{seidel2008fukaya}. Also define $\mathcal H_m$ to be the Deligne--Mumford space of unit disks in the complex plane with $m+2$ boundary punctures that are oriented counterclockwise. Its Deligne--Mumford compactification is denoted by $\overline{\mathcal H}_m$. The boundary of $\overline{\mathcal H}_m$ is obtained by adding broken disks and hence the codimension one boundary of $\overline{\mathcal H}_{m}$ is covered by the following spaces
		\begin{align}
			\label{eq:stratification_of_deligne_mumford_compactification_1}
			\overline{\mathcal H}_{m_1} &\times \overline{\mathcal H}_{m_2}, \, m_1+m_2 = m \\
			\label{eq:stratification_of_deligne_mumford_compactification_2}
			\overline{\mathcal H}_{m_1} &\times \overline{\mathcal R}_{m_2}, \, m_1+m_2 = m+1
		\end{align}
		where we regard each stratum as being included in $\overline{\mathcal H}_m$ via the natural inclusion.

		Consider a word of generators $a_k \in CW^\ast_{\varLambda_K}(F_{k-1},F_k)$
		\[
			\boldsymbol a = a_1 \cdots a_m\, .
		\]
		Then we define the moduli space $\mathcal M(\boldsymbol a)$ to be maps
		\[
			u\co T \longrightarrow W_K\, ,
		\]
		where $T\in \overline{\mathcal H}_m$, and so that $u$ satisfies the following Floer equation
		\[
			\begin{cases}
				du + J_T \circ du \circ j = 0\\
				\lim_{s\to \infty} u(\varepsilon^k(s,t)) = a_k(t), &\forall t\in [0,1] \text{ and } k \in \parenm{1,\ldots,m} \\
				\lim_{s\to \infty} u(\varepsilon_{+}(s,t)) = \xi, & \forall t \in [0,1] \\
				\lim_{s\to -\infty} u(\varepsilon_{-}(s,t)) = \xi, & \forall t \in [0,1]
			\end{cases}
		\]
		where $\varepsilon_{}{\pm}\co Z_{\pm} \longrightarrow T$ and $\varepsilon^k \co Z_+ \longrightarrow T$ are strip-like ends near each puncture $\zeta^{\pm}$ and $\zeta^k$ for $k\in \parenm{1,\ldots,m}$. The boundary conditions of $u$ is indicated in \cref{fig:a_inf_moduli_space_disk} below
		\begin{figure}[H]
			\centering
			\includegraphics{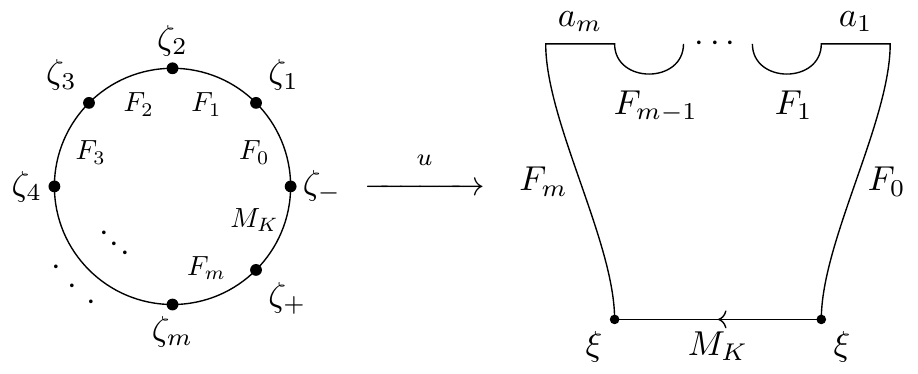}
			\caption{A $J_T$-holomorphic disk in $\mathcal M(\boldsymbol a)$.}
			\label{fig:a_inf_moduli_space_disk}
		\end{figure}
		Again, analogous to \cite[Theorem 63]{ekholm2017duality} and \cite[Lemma 4.7]{abouzaid2012wrapped} we have the following standard transversality result.
		\begin{lma}\label{lma:transversality_of_moduli_space_of_half_strips}
			For a generic choice of almost complex structure, $\mathcal M(\boldsymbol a)$ is a smooth orientable manifold of dimension
			\[
				\dim \mathcal M(\boldsymbol a) = -1+m - \sum_{j=1}^{m} \abs{a_{j}}\, .
			\]
		\end{lma}
		\begin{proof}
			We first observe that disks in $\mathcal M(\boldsymbol a)$ have switching boundary condition which implies that they can not be multiply covered for topological reasons. Then transversality is proved using standard techniques as in \cite{ekholm2007legendrian,ekholm2017duality}.

			We now prove the statement about the dimension. The proof is similar to the proof of \cref{thm:dimension_of_various_moduli_spaces}. By \cite[Theorem A.1.]{cieliebak2010compactness} we have
			\[
				\dim \mathcal M(\boldsymbol a) = (n-3) + m+2 + \mu(\dd u, Z_{\dd u})\, ,
			\]
			where $\mu(\dd u, Z_{\dd u})$ is defined as in the proof of \cref{thm:dimension_of_various_moduli_spaces}. There is a new type of contribution coming from the Lagrangian intersection punctures $\zeta_{\pm}$. By definition of $\mu(\partial u, Z_{\partial u})$ we see that the sum of the contributions from both $\zeta_{\pm}$ is equal to $-n$. The Maslov class of $M_K$ vanishes (see \cref{rmk:maslov_class_of_complement}), so the only contributions to $\mu(\partial u, Z_{\partial u})$ comes from the generators $\boldsymbol a$ and the Lagrangian intersection pucntures $\zeta_{\pm}$. Therefore
			\[
				\dim \mathcal M(\boldsymbol a) = (n-1) + m - n - \sum_{j=1}^m \abs{a_j} = -1 + m - \sum_{j=1}^m \abs{a_j}\, .
			\]
			Furthermore, by vanishing of the Maslov class of $M_{K}$ (see \cref{rmk:maslov_class_of_complement}) it allows us to find a coherent orientation of the moduli spaces. See \cref{sec:signs_gradings_and_orientations_of_moduli_spaces} for a more general discussion about orientations.
		\end{proof}
		Since $W_K$ is non-compact, we use monotonicity together with a generically chosen metric to make sure $J_T$-holomorphic half strips do not escape to horizontal infinity, see \cref{sec:monotonicity_estimates} and in particular \cref{thm:holomorphic_strips_stays_in_compact_part}. This gives that $\mathcal M(\boldsymbol a)$ can be compactified by adding several-level curves and we denote the compactification by $\overline{\mathcal M}(\boldsymbol a)$. Similar to~\cite[Lemma 4.9]{abouzaid2012wrapped} and by \eqref{eq:stratification_of_deligne_mumford_compactification_1}, \eqref{eq:stratification_of_deligne_mumford_compactification_2} the codimension one boundary of $\overline{\mathcal M}(\boldsymbol a)$ is stratified as
		\begin{equation}\label{eq:compactification_stratification}
			\dd \overline{\mathcal M}(\boldsymbol a) = \coprod_{\substack{\tilde{\boldsymbol a} \subset \boldsymbol a \\ t+s+r = m}} \overline{\mathcal M}(\boldsymbol a \setmin \tilde{\boldsymbol a}) \times \mathcal M^{\mathrm{cw}}(\tilde{\boldsymbol a}) \amalg \coprod_{\substack{\boldsymbol a' \boldsymbol a'' = \boldsymbol a \\ m_1 + m_2 = m}} \overline{\mathcal M}(\boldsymbol a') \times \overline{\mathcal M}(\boldsymbol a'')\, .
		\end{equation}
		Note that we define $\mathcal M^{\mathrm{cw}}(\tilde{\boldsymbol a})$ to mean either $\mathcal M^{\mathrm{pb}}(\tilde{\boldsymbol a})$ or $\mathcal M^{\mathrm{fi,Lag}}(\tilde{\boldsymbol a})$ as in \cref{sec:wrapped_floer_cohomology_wo_ham}, depending on whether the breaking happens at a Reeb chord or a Lagrangian intersection generator.

		To be more precise, \eqref{eq:compactification_stratification} means that the codimension one boundary of $\dd \overline{\mathcal M}(\boldsymbol a)$ is covered by images of the natural inclusions of products $\overline{\mathcal M}(\boldsymbol a \setmin \tilde{\boldsymbol a}) \times \mathcal M^{\mathrm{cw}}(\tilde{\boldsymbol a})$ for subwords $\tilde{\boldsymbol a} \subset \boldsymbol a$ and $\overline{\mathcal M}(\boldsymbol a') \times \overline{\mathcal M}(\boldsymbol a'')$ for partitions
		\[
			\boldsymbol a = \underbrace{a_1 \cdots a_{m_1}}_{=\boldsymbol a'} \underbrace{a_{m_1+1} \cdots a_{m_1+m_2}}_{= \boldsymbol a''}\, .
		\]
		Note that $\boldsymbol a \setmin \tilde{\boldsymbol a}$ is the word of generators obtained by starting with the word $\boldsymbol a$ and replacing the subword $\tilde{\boldsymbol a}$ with an auxiliary generator $y$, see \cref{fig:broken_disks_1} and \cref{fig:broken_disks_2}. If $\boldsymbol a = a_1 \cdots a_m$ and $\tilde{\boldsymbol a} = a_{t+1} \cdots a_{t+s} \subset \boldsymbol a$ then
		\begin{equation}\label{eq:a_minus_subword_a_tilde_def}
			\boldsymbol a \setmin \tilde{\boldsymbol a} \defeq a_1 \cdots a_t y a_{t+s+1} \cdots a_m = (\boldsymbol a \setmin \tilde{\boldsymbol a})_1 y (\boldsymbol a \setmin \tilde{\boldsymbol a})_2\, .
		\end{equation}
		In this case, where the auxiliary generator $y$ is placed at position $t+1$ in \eqref{eq:a_minus_subword_a_tilde_def} we say that $\tilde{\boldsymbol a} \subset \boldsymbol a$ is a \emph{subword of $\boldsymbol a$ at position $t+1$}.

		We now have a lemma of how the orientation of the different strata compares to the boundary orientation. See \cref{sec:signs_gradings_and_orientations_of_moduli_spaces} for a general discussion about orientations of moduli spaces.
		\begin{lma}\label{lma:product_ori_differs_from_boundary_ori}
			The product orientation on $\overline{\mathcal M}(\boldsymbol a') \times \overline{\mathcal M}(\boldsymbol a'')$ differs from the boundary orientation on $\dd \overline{\mathcal M}(\boldsymbol a)$ by $(-1)^{\ddag_1}$ where
			\begin{equation}\label{eq:sign_difference_HH_strata}
				\ddag_1 = (m_2+1) \paren{\sum_{i=1}^{m_1} \abs{a_i}} + m_1\, ,
			\end{equation}
			while the product orientation on $\overline{\mathcal M}(\boldsymbol a \setmin \tilde{\boldsymbol a}) \times \mathcal M^{\mathrm{cw}}(\tilde{\boldsymbol a})$ differs from the boundary orientation on $\dd \overline{\mathcal M}(\boldsymbol a)$ by $(-1)^{\ddag_2}$ where
			\begin{equation}\label{eq:sign_difference_HR_strata}
				\ddag_2 = s \paren{\abs \xi + \sum_{i=1}^{t+s}\abs{a_i}} + s(m-t)+t+s\, ,
			\end{equation}
			whenever $\mathcal M^{\mathrm{cw}}(\tilde{\boldsymbol a})$ is rigid. Here $\tilde{\boldsymbol a}$ is a subword of $\boldsymbol a$ at position $t+1$ as in \eqref{eq:a_minus_subword_a_tilde_def}.
		\end{lma}
		\begin{proof}
			See \cref{sec:signs_gradings_and_orientations_of_moduli_spaces}.
		\end{proof}
		\begin{lma}\label{lma:choice_of_fundamental_chain}
			There exists a family of fundamental chains $[\overline{\mathcal M}(\boldsymbol a)]\in C_{\ast}(\overline{\mathcal M}(\boldsymbol a))$ such that
			\begin{equation}\label{eq:fund_class_of_bdry}
				\dd[\overline{\mathcal M}(\boldsymbol a)] =\sum_{\boldsymbol a' \boldsymbol a'' = \boldsymbol a} (-1)^{\ddag_1}[\overline{\mathcal M}(\boldsymbol a')] \times [\overline{\mathcal M}(\boldsymbol a'')] + \sum_{\tilde{\boldsymbol a} \subset \boldsymbol a} (-1)^{\ddag_2}[\overline{\mathcal M}(\boldsymbol a \setmin \tilde{\boldsymbol a})] \times [\mathcal M^{\mathrm{cw}}(\tilde{\boldsymbol a})]\, ,
			\end{equation}
			where $\ddag_1$ and $\ddag_2$ are as in \eqref{eq:sign_difference_HH_strata} and \eqref{eq:sign_difference_HR_strata} respectively.
		\end{lma}
		\begin{proof}
			See \cite[Lemma 4.11]{abouzaid2012wrapped}.
		\end{proof}
	\subsection{The evaluation map and construction of the $A_{\infty}$-homomorphism}
	\label{sub:the_evaluation_map_etc}
		In this section we construct the evaluation map used to define the $A_\infty$-homomorphism between $CW^\ast_{\varLambda_K}(F,F)$ and $C_{-\ast}(\varOmega_\xi M_K)$. 

		First pick any smooth, orientation reversing map $r\co \R \longrightarrow D_{m+2}$ which parametrizes the outgoing segment. (That is, the boundary arc of $D_{m+2}$ that lies between $\zeta_+$ and $\zeta_-$.)
		
		Pick two strip-like ends
			\begin{align*}
				\varepsilon_{\pm}\co (0,\infty) \times [0,1] \longrightarrow U_{\pm}\, ,
			\end{align*}
			where $U_{\pm} \subset D_{m+2}$ are neighborhoods of $\zeta_{\pm} \in D_{m+2}$. We pick the strip-like ends so that $\varepsilon_{\pm}((0,\infty) \times \parenm{0}) \subset U_{\pm}$ are the parts of the boundary of $D_{m+2}$ that points towards $\zeta_{\pm}$ (according to the boundary orientation on $D_{m+2}$), and $\varepsilon_{\pm}((0,\infty) \times \parenm{1}) \subset U_{\pm}$ are the parts of the boundary of $D_{m+2}$ that points away from $\zeta_{\pm}$.
		
		Assume that $r \co \R \longrightarrow D_{m+2}$ satisfies the following
		\begin{equation}\label{eq:paramterization_of_boundary_of_D}
			\begin{cases}
				\lim_{t\to \pm \infty} r(t) = \zeta_{\mp} \\
				\sup_{\abs t \geq \tilde t} \abs{(\varepsilon_{\pm}^{-1} \circ r)^{(n)}(t)} < \infty, \quad \exists \tilde t > 0, \;\forall n \geq 1 \, .
			\end{cases}
		\end{equation}
		Then $u\circ r \co \R \longrightarrow M_K$ is a map so that $\lim_{t \to \pm \infty} (u\circ r)(t) = \xi$. We reparametrize $r$ by arc length with respect to a Riemannian metric on $M_{K}$ (see \eqref{eq:metric_definition}), and compactify the domain. In doing so, we get a smooth, orientation reversing map
		\[
			\tilde r\co [0,R] \longrightarrow D_{m+2}\, ,
		\]
		that satisfies $(u\circ \tilde r)(0) = (u \circ \tilde r)(R) = \xi$, which means $u\circ \tilde r \in \varOmega_{\xi}M_K$. We then define the \emph{evaluation map} as
		\begin{align}\label{eq:evaluation_map_def}
			\ev \co \mathcal M(\boldsymbol a) &\longrightarrow \varOmega_{\xi}M_K \\
			u &\longmapsto u \circ \tilde r \nonumber \, .
		\end{align}	
		\begin{lma}\label{lma:eval_exp_decay}
			Let $u\co D_{m+2} \longrightarrow W_K$ be a $J$-holomorphic disk and take $r\co \R \longrightarrow D_{m+2}$ so that \eqref{eq:paramterization_of_boundary_of_D} holds. Then $\dd_s u \circ r$ decays exponentially in the $C^\infty$-topology.
		\end{lma}
		\begin{proof}
			Pick strip-like ends
			\begin{align*}
				\varepsilon_{\pm}\co (0,\infty) \times [0,1] \longrightarrow U_{\pm}\, ,
			\end{align*}
			as above. By~\cite[Theorem A]{robbin2001asymptotic} we have that $\dd_s u \circ \varepsilon_{\pm}$ decays exponentially in the $C^{\infty}$-topology. When we say that a function decays exponentially in the $C^{\infty}$-toplogy we mean that there are constants $\delta, c_0, c_1, c_2, \ldots > 0$ so that $\forall k \in \N$ and for every $t_0\in (0,\infty)$ we have
			\begin{equation}\label{eq:exponential_decay_of_hol_strip}
				\norm{\dd_s u \circ \varepsilon_{\pm}}_{C^k([t_0,\infty) \times [0,1])} \leq c_k e^{-\delta t_0}\, .
			\end{equation}
			
			Next consider $r\co \R \longrightarrow D_{m+2}$ which satisfies \eqref{eq:paramterization_of_boundary_of_D}, where $\tilde t > 0$ is large enough so that $r(t) \in U_{\pm}$ for $\abs t > \tilde t$. This also gives
			\[
				(u \circ r)(t) = \begin{cases}
					\paren{u \circ \varepsilon_-} \circ (\varepsilon_-^{-1} \circ r)(t), & t \geq \tilde t\\
					\paren{u \circ \varepsilon_+}\circ (\varepsilon_+^{-1} \circ r)(t), & t \leq -\tilde t
				\end{cases}
			\]
			where $\varepsilon_{\pm}^{-1} \circ r \co \R \longrightarrow (0,\infty) \times [0,1]$ are maps so that
			\[
				\begin{cases}
					(\varepsilon_-^{-1} \circ r)(t) \subset (0,\infty) \times \parenm{1} \\
					(\varepsilon_+^{-1} \circ r)(t) \subset (0,\infty) \times \parenm{0}\, ,
				\end{cases}
			\]
			and
			\[
				\begin{cases}
					\lim_{t \to \infty} (\varepsilon_{-}^{-1} \circ r)(t) = (\infty, 1) \\
					\lim_{t \to -\infty} (\varepsilon_{+}^{-1} \circ r)(t) = (\infty, 0)\, .
				\end{cases}
			\]
			
			Then we have constants $\delta, c_0,c_1,c_2, \ldots > 0$ so that $\forall k \geq 0$
			\begin{align*}
				\norm{\dd_s u \circ r}_{C^k([\tilde t, \infty))} &= \sum_{\abs \alpha \leq k} \sup_{\abs t \geq \tilde t} \abs{D^\alpha(\dd_s u \circ r)} \\
				&= \sum_{\abs \alpha \leq k} \sup_{\abs t \geq \tilde t} \abs{D^\alpha	\parenb{\dd_s u \circ \varepsilon_{\pm}}(\varepsilon_{\pm}^{-1}(r(t))) \cdot D^\alpha \parenb{\varepsilon_{\pm}^{-1} \circ r}(t)} \\
				&= \sum_{\abs \alpha \leq k} \parenb{\sup_{[t_0,\infty) \times [0,1]} \abs{D^\alpha(\dd_s u \circ \varepsilon_{\pm})}} \parenb{\sup_{\abs t \geq \tilde t} \abs{D^\alpha(\varepsilon_{\pm}^{-1} \circ r)}} \, .
			\end{align*}
			Here $D^\alpha$ denotes derivative with respect to the multi-index $\alpha$. Because of~\eqref{eq:paramterization_of_boundary_of_D} we have
			\[
				\sup_{\abs t \geq \tilde t} \abs{D^\alpha(\varepsilon_{\pm}^{-1} \circ r)} \leq A_{\alpha}\, ,
			\]
			where $A_\alpha$ is some constant depending on $\alpha$. We conclude
			\[
				\norm{\dd_s u\circ r}_{C^k([\tilde t, \infty))} \leq A_k\norm{\dd_s u \circ \varepsilon_{\pm}}_{C^k([t_0, \infty) \times [0,1])} \leq A_k \cdot c_k e^{- \delta t_0}\, ,
			\]
			by~\eqref{eq:exponential_decay_of_hol_strip}, where $A_k \defeq \max_{\abs a \leq k} A_{\alpha}$. Furthermore we note that $\tilde t > 0$ is large enough so that for $\abs t \geq \tilde t$ we have
			\[
				\begin{cases}
					(\varepsilon_{-}^{-1} \circ r)(t) \subset [t_0, \infty) \times \parenm{1} \\
					(\varepsilon_{+}^{-1} \circ r)(t) \subset [t_0, \infty) \times \parenm{0}\, .
				\end{cases}
			\]
		\end{proof}
			The previous lemma enables us to extend the evaluation map to the compactification of the moduli space of half strips.
		\begin{lma}\label{lma:ev_map_def_on_compactification}
			There is an extension of the evaluation map $\ev$ to a continuous map on the compactification of $\mathcal M(\boldsymbol a)$,
			\[
				\ev \co \overline{\mathcal M}(\boldsymbol a) \longrightarrow \varOmega_{\xi}M_K\, ,
			\]
			such that the following diagram commutes up to an overall sign of $(-1)^{\ddag_1}$, where $\ddag_1$ is defined in \eqref{eq:sign_difference_HH_strata}.
			\[
				\begin{tikzcd}[row sep=scriptsize, column sep=scriptsize]
					\overline{\mathcal M}(\boldsymbol a') \times \overline{\mathcal M}(\boldsymbol a'') \rar{\iota} \dar{\ev \times \ev} & \overline{\mathcal M}(\boldsymbol a) \dar{\ev} \\
					\varOmega_{\xi}M_K \times \varOmega_{\xi}M_K \rar{\circ} & \varOmega_{\xi}M_K\, ,
				\end{tikzcd}
			\]
			The map $\iota$ in the top row is inclusion as in~\eqref{eq:compactification_stratification}. The map in the bottom row is concatenation of loops.
		\end{lma}
		\begin{proof}
			For this proof, we follow the idea outlined in \cite[p.\@ 37]{abouzaid2012wrapped}.
		\begin{description}
			\item[Extension of $\ev$ to the compactification]
				It is obvious how to extend it to the boundary strata $\overline{\mathcal M}(\boldsymbol a \setmin \tilde{\boldsymbol a}) \times \mathcal M^{\text{cw}}(\tilde{\boldsymbol a})$; we define the evaluation map of such broken disk to be the same as the evaluation map when we forget about the factor $\mathcal M^{\text{cw}}(\tilde{\boldsymbol a})$. However, if we have a sequence $\parenm{u^\nu}_{\nu=0}^\infty \subset \mathcal M(\boldsymbol a)$ which Gromov converges to a broken disk in any of the boundary strata $\overline{\mathcal M}(\boldsymbol a') \times \overline{\mathcal M}(\boldsymbol a'')$, then the Gromov limit is a stable $J$-holomorphic map (a broken disk), consisting of two $J$-holomorphic disks $u_i \co D_{k_i} \longrightarrow W_K$ where $k_1+k_2-2 = m+2$, and two boundary punctures $z_1 \in \dd D_{k_1}$, $z_2 \in \dd D_{k_2}$ so that we either have $(z_1,z_2) = (\zeta_-, \zeta_+)$ or $(z_1,z_2) = (\zeta_+, \zeta_-)$~\cite{frauenfelder2008gromov}. More precisely, it means that there are two families of M{\"o}bius transformations of the unit disk $D \subset \CC$
				\[
					\varphi_{1}^{\nu}, \varphi_{2}^{\nu} \co D \longrightarrow D, \text{ where } \nu\in \N\, ,
				\]
				so that
		\begin{equation}\label{eq:extension_of_ev_compact_möbius}
			\begin{cases}
				u^\nu \circ \varphi_1^\nu \longrightarrow u_1 & \text{in } C^{\infty}_{\text{loc}}(D_{k_1} \setmin \parenm{z_1})\\
				u^\nu \circ \varphi_2^\nu \longrightarrow u_2 & \text{in } C^{\infty}_{\text{loc}}(D_{k_2} \setmin \parenm{z_2})\, ,
			\end{cases}
		\end{equation}
		and
		\[
			\begin{cases}
				(\varphi_1^\nu)^{-1} \circ \varphi_2^\nu \longrightarrow z_1 & \text{in } C^{\infty}_{\text{loc}}(D_{k_1} \setmin \parenm{z_1}) \\
				(\varphi_2^\nu)^{-1} \circ \varphi_1^\nu \longrightarrow z_2 & \text{in } C^{\infty}_{\text{loc}}(D_{k_2} \setmin \parenm{z_2})\, .
			\end{cases}
		\]
		Recall that convergence in $C^{\infty}_{\text{loc}}(X)$ means $C^{\infty}$-convergence on every compact subset $K \subset X$.
		
		Define parametrizations
		\begin{align*}
			r_1\co \R &\longrightarrow D_{k_1} \\
			r_2 \co \R &\longrightarrow D_{k_2}
		\end{align*}
		so that $r_1$ and $r_2$ satisfy \eqref{eq:paramterization_of_boundary_of_D}. Then the two maps $u_i \circ r_i \co \R \longrightarrow M_K$ are smooth maps so that $\dd_s u_i \circ r_i$ decay exponentially in the $C^{\infty}$-topology by \cref{lma:eval_exp_decay}. Hence the composition of two smooth loops $u_i \circ r_i$ is again a smooth loop. There are two cases, depending on whether the two components of the broken disk have the puncture $\zeta_+$ or $\zeta_-$ in common. That is, we either have $(z_1,z_2) = (\zeta_-, \zeta_+)$ or $(z_1,z_2) = (\zeta_+, \zeta_-)$. In the first case when $(z_1,z_2) = (\zeta_-, \zeta_+)$, we define a map $\gamma\co \R \longrightarrow M_K$ as
		\begin{equation}\label{eq:composition_of_loops_exp_decay}
			\gamma(t) \defeq \begin{cases}
				(u_1 \circ r_1)\paren{t-\frac 1t}, & t < 0 \\
				\xi, & t = 0 \\
				(u_2 \circ r_2)\paren{t-\frac 1t}, & t > 0\, .
			\end{cases}
		\end{equation}
		In the second case when $(z_1,z_2) = (\zeta_+, \zeta_-)$ we swap places of $u_1 \circ r_1$ and $u_2 \circ r_2$ in the above definition of $\gamma$.
		
		We then claim that this map is smooth and has exponentially decaying derivatives in the $C^{\infty}$-topology as $t \to \pm \infty$. Since $u_1 \circ r_1$ and $u_2 \circ r_2$ are smooth maps with exponentially decaying derivatives in the $C^{\infty}$-topology as $t \to \pm \infty$, it suffices to show that all derivatives of $\gamma$ at $t = 0$ exists. This follows from the exponential decay of every derivative of $u_1 \circ r_1$ and $u_2 \circ r_2$ in the $C^{\infty}$-topology. We may then reparametrize $\gamma$ by arc length and compactify the domain to obtain a map $\tilde \gamma \co [0,R] \longrightarrow M_K$ so that $\tilde \gamma(0) = \tilde \gamma(R) = \xi$, that is $\tilde \gamma \in \varOmega_{\xi}M_K$, and we define $\ev((u_1,u_2)) \defeq \tilde \gamma$.
		
		\item[Commutativity of the diagram]
			It follows almost immediately from the definition of the evaluation map
			\[
				\ev \co \overline{\mathcal M}(\boldsymbol a) \longrightarrow \varOmega_\xi M_K
			\]
			that the diagram
			\[
				\begin{tikzcd}[row sep=scriptsize, column sep=scriptsize]
					\overline{\mathcal M}(\boldsymbol a') \times \overline{\mathcal M}(\boldsymbol a') \rar{\iota} \dar{\ev \times \ev}& \overline{\mathcal M}(\boldsymbol a) \dar{\ev} \\
					\varOmega_{\xi}M_K \times \varOmega_{\xi}M_K \rar{\circ} & \varOmega_{\xi}M_K
				\end{tikzcd}
			\]
			commutes, since $\gamma$ in \eqref{eq:composition_of_loops_exp_decay} is essentially defined as the concatenation of $u_1 \circ r_1$ and $u_2\circ r_2$. More precisely, we consider $u_i \circ r_i \co \R \longrightarrow M_K$ for $i\in \parenm{1,2}$ as above. Then reparametrize $r_1$ and $r_2$ by arc length so that we obtain two maps
			\[
				u_i \circ \tilde r_i \co [0,R_i] \longrightarrow M_K\, .
			\]
			These maps are so that $(u_i \circ \tilde r_i)(0) = (u_i \circ \tilde r_i)(R_i) = \xi$ for $i\in \parenm{1,2}$, and the concatenation of these maps yields a map $\psi \co [0,R_1+R_2] \longrightarrow M_K$ defined by
			\begin{align*}
				\psi(t) = \begin{cases}
					(u_1 \circ \tilde r_1)(t), & t \in [0,R_1] \\
					(u_2 \circ \tilde r_2)(t-R_1), & t \in [R_1,R_1+R_2]
				\end{cases}
			\end{align*}
			which coincides with the map $\tilde \gamma \co [0,R] \longrightarrow M_K$ obtained by parametrizing $\gamma$ defined in~\eqref{eq:composition_of_loops_exp_decay} by arc length. The overall sign $(-1)^{\ddag_1}$ comes from \cref{lma:product_ori_differs_from_boundary_ori}, see \cref{sec:signs_gradings_and_orientations_of_moduli_spaces} for a discussion about sign and orientations.
		\item[Continuity of $\ev$]
			We claim that $\ev$ is a continuous map, meaning that if $\parenm{u^\nu}_{\nu=0}^\infty \subset \mathcal M(\boldsymbol a)$ is a Gromov convergent sequence of $J$-holomorphic disks, then the map $\tilde \gamma(t)$ defined in \eqref{eq:composition_of_loops_exp_decay} is realized as a limit of loops in the compact-open topology of $\varOmega_{\xi}M_K$.
			
			Pick a family of smooth maps $\parenm{r^\nu\co \R \longrightarrow D_{m+2}}_{\nu=0}^\infty$ which satisfies~\eqref{eq:paramterization_of_boundary_of_D}. Then we have that $\parenm{u^\nu \circ r^\nu}_{\nu=0}^\infty$ is a family of smooth maps with exponentially decaying derivatives as $t\to \pm \infty$ in the $C^{\infty}$-topology by \cref{lma:eval_exp_decay}. From \eqref{eq:extension_of_ev_compact_möbius} we have two families of M{\"o}bius transformations $\parenm{\varphi_1^\nu}_{\nu=0}^\infty$ and $\parenm{\varphi_2^\nu}_{\nu=0}^\infty$ such that
			\[
				u^\nu \circ \varphi_i^{\nu} \to u_i, \quad \text{in } C^{\infty}_{\text{loc}}(D_{k_i} \setmin \parenm{z_i})\, ,
			\]
			for $i\in \parenm{1,2}$. We also have that $\varphi_i^{\nu}$ preserves the boundary of $D_m$ and that $\paren{\varphi_i^{\nu}}^{-1}$ preserves boundary marked points in the sense that $\lim_{\nu\to \infty}\paren{\varphi_i^{\nu}}^{-1}(\zeta_j) = \zeta_j$. Then we have
			\begin{equation}\label{eq:convergece_of_parametrizations}
				\begin{cases}
					\paren{\varphi_1^{\nu}}^{-1} \circ r^\nu \longrightarrow r_1 , \quad \text{in } C^{\infty}_{\text{loc}}(\R_{< 0})\\
					\paren{\varphi_2^{\nu}}^{-1} \circ r^\nu \longrightarrow r_2 , \quad \text{in } C^{\infty}_{\text{loc}}(\R_{> 0})\, .
				\end{cases}
			\end{equation}
			Hence for any multi-index $\alpha$ and $i\in \parenm{1,2}$ we have
			\begin{align} \label{eq:alpha_derivatives_difference}
				\abs{D^{\alpha}(u^\nu \circ r^\nu) - D^\alpha(u_i \circ r_i)} &= \abs{D^{\alpha}(u^\nu \circ \varphi_i^\nu \circ \paren{\varphi_i^\nu}^{-1} \circ r^\nu) - D^\alpha(u_i \circ r_i)}\\
				&= \abs{D^\alpha \paren{u^\nu \circ \varphi_i^\nu} D^\alpha \parenb{\paren{\varphi_i^\nu}^{-1} \circ r^\nu} - D^\alpha (u_i) D^\alpha(r_i)}\nonumber \\
				&\leq \abs{D^\alpha(u^\nu \circ \varphi_i^\nu) \cdot \paren{D^\alpha \parenb{\paren{\varphi_i^\nu}^{-1} \circ r^\nu}-D^\alpha(r_i)}} \nonumber \\
				&\qquad + \abs{D^\alpha(r_i) \cdot \parenb{D^\alpha \paren{u^\nu \circ \varphi_i^\nu} - D^\alpha(u_i)}}\nonumber \, .
			\end{align}
			Let $i\in \parenm{1,2}$ and define $\R_1 \defeq \R_{< 0}$ and $\R_2 \defeq \R_{> 0}$. Inserting suprema over suitable compact sets $A \subset \R_i$ and $K \subset D_{k_i} \setmin \parenm{z_i}$ gives
			\begin{align}\label{eq:supremum_estimate_parametrization}
				&\sup_{A \subset \R_{i}}\abs{D^{\alpha}(u^\nu \circ r^\nu) - D^\alpha(u_i \circ r_i)} \\
				&\overset{\eqref{eq:alpha_derivatives_difference}}{\leq} \sup_{K \subset D_{k_i} \setmin \parenm{z_i}}\abs{D^\alpha(u^\nu \circ \varphi_i^\nu)} \sup_{A \subset \R_{i}}\abs{D^\alpha \paren{\paren{\varphi_i^\nu}^{-1} \circ r^\nu}-D^\alpha(r_i)} \nonumber \\
				&\qquad + \sup_{A \subset \R_{i}}\abs{D^\alpha (r_i)} \sup_{K \subset D_{k_i} \setmin \parenm{z_i}}\abs{D^\alpha \paren{u^\nu \circ \varphi_i^\nu} - D^\alpha(u_i)} \nonumber \\
				&\leq C_1 \underbrace{\sup_{A \subset \R_{i}}\abs{D^\alpha \paren{\paren{\varphi_i^\nu}^{-1} \circ r^\nu}-D^\alpha(r_i)}}_{\longrightarrow 0} + C_2 \underbrace{\sup_{K \subset D_{k_i} \setmin \parenm{z_i}}\abs{D^\alpha \paren{u^\nu \circ \varphi_i^\nu} - D^\alpha(u_i)}}_{\longrightarrow 0} \, , \nonumber
			\end{align}
			Here we have used that $u^\nu \circ \varphi_i^\nu \longrightarrow u_i$ in $C^{\infty}_{\mathrm{loc}}(D_{k_i} \setmin \parenm{z_i})$ and hence that $u^\nu \circ \varphi_i^\nu$ is also bounded in this topology. Furthermore we have used that $\paren{\varphi_i^{\nu}}^{-1} \circ r^\nu \longrightarrow r_i$ in $C^\infty_{\text{loc}}(\R_i)$ by \eqref{eq:convergece_of_parametrizations}.

			Then by recalling the definition of $\gamma(t)$ in~\eqref{eq:composition_of_loops_exp_decay}, we have
			\begin{align*}
				&\sup_{A \subset \R} \abs{D^\alpha(u^\nu \circ r^\nu) - D^\alpha(\gamma)} \\
				&\leq \sup_{A \subset \R_{< 0}} \abs{D^\alpha(u^\nu \circ r^\nu) - D^\alpha(\gamma)} + \sup_{A \subset \R_{> 0}} \abs{D^\alpha(u^\nu \circ r^\nu) - D^\alpha(\gamma)} \\
				&= \sup_{A \subset \R_{< 0}} \abs{D^\alpha(u^\nu \circ r^\nu) - D^\alpha(u_1 \circ r_1)} + \sup_{A \subset \R_{> 0}} \abs{D^\alpha(u^\nu \circ r^\nu) - D^\alpha(u_2 \circ r_2)}
			\end{align*}
			By~\eqref{eq:supremum_estimate_parametrization}, we get $u^\nu \circ r^\nu \longrightarrow \gamma$ in $C^{\infty}_{\mathrm{loc}}(\R)$, and thus by passing to arc length parametrizations we get $\ev(u^\nu) \longrightarrow \ev(u)$ in the compact-open topology on $\varOmega_{\xi}M_K$.
		\end{description}
		\end{proof}
		
 		The evaluation map
 		\[
 			\ev \co \overline{\mathcal M}(\boldsymbol a) \longrightarrow \varOmega_\xi M_K
 		\]
 		induces a map on chains $\ev_{\ast} \co C_{-\ast}(\overline{\mathcal M}(\boldsymbol a)) \longrightarrow C_{-\ast}(\varOmega_{\xi}M_K)$. We then pick a fundamental chain $[\overline{\mathcal M}(\boldsymbol a)]$ by \cref{lma:choice_of_fundamental_chain} so that~\eqref{eq:fund_class_of_bdry} holds, and define a family of maps $\parenm{\varPsi_m}_{m=1}^\infty$
		\begin{align}\label{eq:def_of_chain_maps}
			\varPsi_m\co CW^\ast_{\varLambda_K}(F_{m-1},F_m) \otimes \cdots \otimes CW^\ast_{\varLambda_K}(F_0,F_1) &\longrightarrow C_{-\ast}(\varOmega_{\xi}M_K) \\
			a_m \otimes \cdots \otimes a_1&\longmapsto (-1)^{\S} \ev_{\ast}[\overline{\mathcal M}(\boldsymbol a)] \nonumber\, ,
		\end{align}
		where
		\[
			\S = \sum_{j=1}^m j \abs{a_j} + (m+1)\abs\xi + (\abs \xi + m) \dim \overline{\mathcal M}(\boldsymbol a) = \sum_{j=1}^m j \abs{a_j} + (\abs \xi + m) \sum_{j=1}^m \abs{a_j} \pmod 2\, .
		\]
		Note that $\abs \xi$ means the grading of $\xi$ regarded as an intersection generator of $CW^\ast_{\varLambda_K}(F_0,F_1)$ as in \cref{sub:a_infty_structure_and_moduli_space_of_disks_cw_wo_ham}.
		\begin{lma}\label{lma:chain_diagrams}
			The following diagram commutes,
			\[
				\begin{tikzcd}[row sep=scriptsize, column sep=scriptsize]
					C_{-k}(\overline{\mathcal M}(\boldsymbol a)) \dar{\ev_{\ast}} \rar{\dd} &  C_{-k+1}(\overline{\mathcal M}(\boldsymbol a)) \dar{\ev_{\ast}}\\
					C_{-k}(\varOmega_{\xi}M_K) \rar{\dd} & C_{-k+1}(\varOmega_{\xi}M_K)
				\end{tikzcd}
			\]
			and the following diagram commutes up to an overall sign of $(-1)^{\ddag_1 + \dim \overline{\mathcal M}(\boldsymbol a')}$, where $\ddag_1$ is defined in \eqref{eq:sign_difference_HH_strata}.
			\[
				\begin{tikzcd}[row sep=scriptsize, column sep=scriptsize]
					C_{-k}(\overline{\mathcal M}(\boldsymbol a')) \otimes C_{-\ell}(\overline{\mathcal M}(\boldsymbol a'')) \rar{\iota_{\ast} \circ \times} \dar{\ev_{\ast} \otimes \ev_{\ast}} & C_{-(k+\ell)}(\overline{\mathcal M}(\boldsymbol a)) \dar{\ev_{\ast}} \\
					C_{-k}(\varOmega_{\xi}M_K) \otimes C_{-\ell}(\varOmega_{\xi}M_K) \rar{P} & C_{-(k+\ell)}(\varOmega_{\xi}M_K)
				\end{tikzcd}
			\]
			In the latter diagram we have the subdivision $\boldsymbol a = \boldsymbol a' \boldsymbol a''$, and the map $\iota_\ast$ is the composition of the map induced by the inclusion
			\[
				\iota\co \overline{\mathcal M}(\boldsymbol a') \times \overline{\mathcal M}(\boldsymbol a'') \longrightarrow \overline{\mathcal M}(\boldsymbol a)\, .
			\]
		\end{lma}
		\begin{proof}
			That the first diagram commutes follows more or less by definition. Namely, let $A \in C_{-k}(\overline{\mathcal M}(\boldsymbol a))$. Then $\ev_{\ast}(A) = \ev \circ A$, and by using the definition of $\dd$ in \eqref{eq:differential_on_cubical_chains} and the definition of $\ev$ in \eqref{eq:evaluation_map_def} we get
			\begin{align*}
				\dd(\ev_{\ast}(A)) &= \sum_{i=1}^k \sum_{\varepsilon = 0}^1 (-1)^{i+\varepsilon} (\ev \circ A)(\delta_{i,\varepsilon}(x)) = \sum_{i=1}^k \sum_{\varepsilon = 0}^1 (-1)^{i+\varepsilon} (A(\delta_{i,\varepsilon}(x)) \circ \tilde r) \\
				&= \paren{\sum_{i=1}^k \sum_{\varepsilon=0}^1(-1)^{i+ \varepsilon} A(\delta_{i,\varepsilon}(x))} \circ \tilde r \\
				&= \ev_{\ast}(\dd A)\, .
			\end{align*}
			The second diagram is split up into the following digram
			\[
				\begin{tikzcd}[row sep=scriptsize, column sep=scriptsize]
					C_{-k}(\overline{\mathcal M}(\boldsymbol a')) \otimes C_{-\ell}(\overline{\mathcal M}(\boldsymbol a'')) \dar{\ev_{\ast} \otimes \ev_{\ast}} \rar{\times} & C_{-(k+\ell)}(\overline{\mathcal M}(\boldsymbol a') \times \overline{\mathcal M}(\boldsymbol a'')) \rar{\iota_{\ast}} \dar{\ev_{\ast} \times \ev_{\ast}} & C_{-(k+\ell)}(\overline{\mathcal M}(\boldsymbol a)) \dar{\ev_{\ast}} \\
					C_{-k}(\varOmega_{\xi}M_K)\otimes C_{-\ell}(\varOmega_{\xi}M_K) \rar{\times} \ar[rr, bend right, swap, looseness=0.2, "P"] & C_{-(k+\ell)}((\varOmega_{\xi}M_K)^2) \rar{\circ} & C_{-(k+\ell)}(\varOmega_{\xi}M_K)\, .
				\end{tikzcd}
			\]
			The right square commutes, since the corresponding diagram before application of $C_{-\ast}$ commutes, by \cref{lma:ev_map_def_on_compactification}, and the maps $\iota_\ast$, $\ev_\ast$ and $\circ$ on chains are defined pointwise. The left square also commutes, because $\ev_{\ast} \otimes \ev_{\ast}$ and $\ev_{\ast} \times \ev_{\ast}$ act componentwise. Hence the outer square also commutes.

			The overall sign $(-1)^{\ddag_1 + \dim(\overline{\mathcal M}(\boldsymbol a'))}$ comes from the definition of $P$ in \eqref{eq:def_pontryagin_product}, and from the inclusion
			\[
				\iota\co \overline{\mathcal M}(\boldsymbol a') \times \overline{\mathcal M}(\boldsymbol a'') \longrightarrow \overline{\mathcal M}(\boldsymbol a)\, ,
			\]
			of $\overline{\mathcal M}(\boldsymbol a') \times \overline{\mathcal M}(\boldsymbol a'')$ as a boundary stratum of $\overline{\mathcal M}(\boldsymbol a)$ as in \eqref{eq:fund_class_of_bdry}.
		\end{proof}
		\begin{lma}\label{lma:varpsi_a_infty_morphism}
			The maps $\parenm{\varPsi_m}_{m=1}^\infty$ form an $A_{\infty}$-homomorphism. That is,
			\[
				\dd \varPsi_m + \sum_{m_1+m_2 = m} P(\varPsi_{m_2} \otimes \varPsi_{m_1}) = \sum_{r+s+t=m} (-1)^{\maltese_t} \varPsi_{r+1+t}(\id^{\otimes r} \otimes \mu^s \otimes \id^{\otimes t})\, ,
			\]
			where
			\[
				\maltese_t = t + \sum_{j=1}^t \abs{x_j}\, .
			\]
		\end{lma}
		\begin{proof}
			From \cref{lma:transversality_of_moduli_space_of_half_strips} it is clear that $\varPsi_m$ has degree $1-m$.
			
			We first ignore signs and prove the statement modulo 2. We look at the codimension one boundary of $\overline{\mathcal M}(\boldsymbol a)$ of dimension $d$. It consists of two types of broken $J$-holomorphic curves as in~\eqref{eq:fund_class_of_bdry}, and we analyze each boundary term separately. 
			\begin{enumerate}
				\item The first boundary term is
				\[
					\coprod_{\tilde{\boldsymbol a} \subset \boldsymbol a} \overline{\mathcal M}(\boldsymbol a \setmin \tilde{\boldsymbol a}) \times \mathcal M^{\mathrm{cw}}(\tilde{\boldsymbol a})\, ,
				\]
				where $\tilde{\boldsymbol a} \subset \boldsymbol a$ is a subword at position $t+1$ of $\boldsymbol a$.
				\item The second boundary term is
				\[
					\coprod_{\boldsymbol a' \boldsymbol a'' = \boldsymbol a} \overline{\mathcal M}(\boldsymbol a') \times \overline{\mathcal M}(\boldsymbol a'')\, ,
				\]
				and it consists of broken half strips that is broken at the Lagrangian intersection point $\xi$.
			\end{enumerate}
			In view of~\eqref{eq:fund_class_of_bdry}, we consider the fundamental chain of $\dd \overline{\mathcal M}(\boldsymbol a)$. Consider the natural inclusions of the boundary strata
			\begin{align*}
				\iota \co\overline{\mathcal M}(\boldsymbol a') \times \overline{\mathcal M}(\boldsymbol a'') &\longrightarrow \overline{\mathcal M}(\boldsymbol a)\\
				\iota\co \overline{\mathcal M}(\boldsymbol a \setmin \tilde{\boldsymbol a}) \times \mathcal M^{\mathrm{cw}}(\tilde{\boldsymbol a}) &\longrightarrow \overline{\mathcal M}(\boldsymbol a)\, .
			\end{align*}
			We consider $\dd \varPsi_m(a_m \otimes \cdots \otimes a_1)$ and use \cref{lma:chain_diagrams}. Then
			\begin{align}\label{eq:ev_incl_applied_to_chain_of_boundary}
				\dd \varPsi_{m}(a_m \otimes \cdots \otimes a_1) &= \dd \ev_{\ast}[\overline{\mathcal M}(\boldsymbol a)] = \ev_{\ast} \dd [\overline{\mathcal M}(\boldsymbol a)] \\
				&= \sum_{\boldsymbol a' \boldsymbol a'' = \boldsymbol a} \ev_{\ast}\paren{\iota_\ast \paren{[\overline{\mathcal M}(\boldsymbol a')] \times [\overline{\mathcal M}(\boldsymbol a'')]}} \nonumber \\
				&\quad + \sum_{\tilde{\boldsymbol a} \subset \boldsymbol a} \ev_{\ast} \paren{\iota_\ast \paren{[\overline{\mathcal M}(\boldsymbol a \setmin \tilde{\boldsymbol a})] \times [\mathcal M^{\mathrm{cw}}(\tilde{\boldsymbol a})]}} \nonumber
			\end{align}
			We start by considering boundary terms of type (1). The evaluation applied to these terms is
			\[
				\ev_{\ast}\iota_\ast \paren{[\overline{\mathcal M}(\boldsymbol a \setmin \tilde{\boldsymbol a})] \times [\mathcal M^{\mathrm{cw}}(\tilde{\boldsymbol a})]} = \ev_{\ast}[\overline{\mathcal M}((\boldsymbol a \setmin \tilde{\boldsymbol a})_{1} \mu^{s}(\tilde{\boldsymbol a}) (\boldsymbol a \setmin \tilde{\boldsymbol a})_{2})]\, ,
			\]
			because of the definition of $\ev$ on these boundary strata. Note that if $\mathcal M^{\mathrm{cw}}(\tilde{\boldsymbol a})$ is not rigid, then the image $\ev_{\ast}\iota_\ast \paren{[\overline{\mathcal M}(\boldsymbol a \setmin \tilde{\boldsymbol a})] \times [\mathcal M^{\mathrm{cw}}(\tilde{\boldsymbol a})]}$ would be degenerate in $C_{-\ast}(\varOmega_\xi M_K)$, and hence does not contribute. In figures we illustrate this equality as follows
			\[
				\ev_{\ast} \iota_{\ast} \parenb{\raisebox{-0.4\height}{\includegraphics{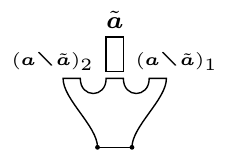}}} = \ev_{\ast} \parenb{\raisebox{-0.4\height}{\includegraphics{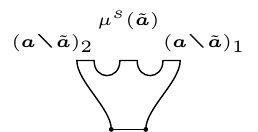}}}\, .
			\]
			The word $(\boldsymbol a \setmin \tilde{\boldsymbol a})_{1} \mu^{s}(\tilde{\boldsymbol a}) (\boldsymbol a \setmin \tilde{\boldsymbol a})_{2}$ is the word obtained from $\boldsymbol a$, by replacing the word $\tilde{\boldsymbol a}$ with $\mu^{s}(\tilde{\boldsymbol a})$. Therefore 
			\begin{equation}\label{eq:broken_disks_of_type_1}
				\ev_{\ast} \paren{\iota_\ast \paren{[\overline{\mathcal M}(\boldsymbol a \setmin \tilde{\boldsymbol a})] \times [\mathcal M^{\mathrm{cw}}(\tilde{\boldsymbol a})]}} = \varPsi_{r+1+t}\paren{a_{m} \otimes \cdots \otimes a_{t+s+1} \otimes \mu^{s}(\tilde{\boldsymbol a}) \otimes a_{t} \otimes \cdots \otimes a_{1}}\, ,
			\end{equation}
			where
			\begin{align*}
				\begin{cases}
					t = \text{length of the word } (\boldsymbol a \setmin \tilde{\boldsymbol a})_{1} \\
					s = \text{length of the word } \boldsymbol{\tilde a} \\
					r = \text{length of the word } (\boldsymbol a \setmin \tilde{\boldsymbol a})_{2}\, .
				\end{cases}
			\end{align*}
			This means that the broken disks of type (1) correspond to terms of the form $\varPsi_{r+1+t}(\id^{\otimes r} \otimes \mu^s \otimes \id^{\otimes t})$ where $r+s+t = m$.
			\begin{figure}[H]
				\centering
				\includegraphics{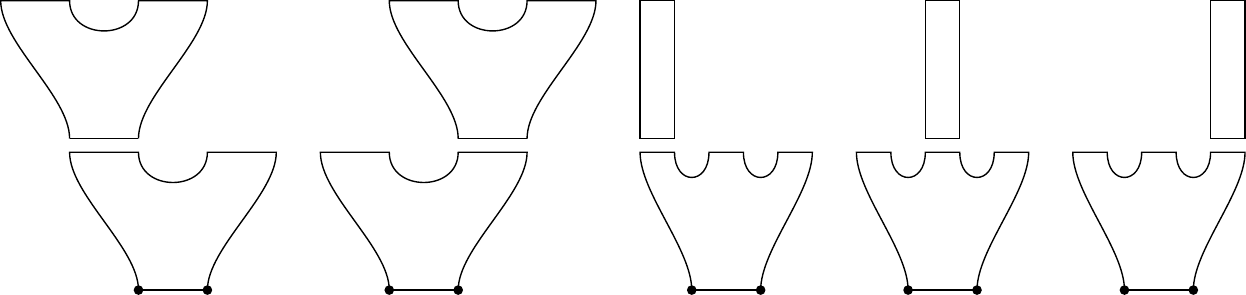}
				\caption{All broken disks of type (1) in the case of $m = 3$.}
				\label{fig:broken_disks_1}
			\end{figure}
			Similarly, for the first terms in~\eqref{eq:ev_incl_applied_to_chain_of_boundary} which correspond to broken disks of type (2) we apply \cref{lma:chain_diagrams} to get
			\begin{equation}\label{eq:broken_disks_of_type_2}
				\ev_{\ast}\paren{\iota_\ast \paren{[\overline{\mathcal M}(\boldsymbol a')] \times [\overline{\mathcal M}(\boldsymbol a'')]}} = P \paren{\ev_{\ast} [\overline{\mathcal M}(\boldsymbol a'')] \otimes \ev_{\ast}[\overline{\mathcal M}(\boldsymbol a')]} = P(\varPsi_{m_{2}}(\boldsymbol a'') \otimes \varPsi_{m_{1}}(\boldsymbol a'))\, ,
			\end{equation}
			so that the broken disks of type (2) correspond to terms of the form $P(\varPsi_{m_2} \otimes \varPsi_{m_1})$ where $m_1+m_2=m$.
			\begin{figure}[H]
				\centering
				\includegraphics{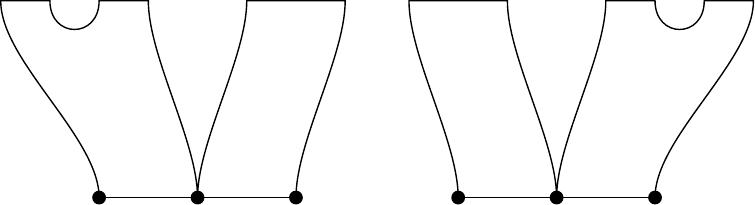}
				\caption{All broken disks of type (2) in the case of $m=3$.}
				\label{fig:broken_disks_2}
			\end{figure}
			Therefore via \eqref{eq:broken_disks_of_type_1} and \eqref{eq:broken_disks_of_type_2}, equation \eqref{eq:ev_incl_applied_to_chain_of_boundary} becomes
			\[
				\dd \varPsi_m + \sum_{m_1+m_2=m} P(\varPsi_{m_2} \otimes \varPsi_{m_1}) = \sum_{r+s+t = m} \varPsi_{r+1+t}(\id^{\otimes r} \otimes \mu^s \otimes \id^{\otimes t})\, ,
			\]
			and these are precisely the $A_{\infty}$-relations modulo 2. For confirmation of signs we refer the reader to \cref{sec:signs_gradings_and_orientations_of_moduli_spaces}.
		\end{proof}
\section{The chain map is an isomorphism}
	\label{sec:chain_map_is_a_quasi-isomorphism}
	This section is dedicated to the proof of \cref{thm:theorem_A}.
	\begin{thm}[\cref{thm:theorem_A}]\label{thm:quiso}
		There exists a geometrically defined isomorphism of $A_\infty$-algebras $\varPsi \co CW^\ast_{\varLambda_K}(F,F) \longrightarrow C^{\text{cell}}_{-\ast}(BM_K)$.
	\end{thm}
	The first step is to replace the full Moore loop space with a Morse theoretic model of it. It is the space of piecewise geodesic loops and we denote it by $BM_K$ (see \cref{sub:filtration_on_based_loops}). In the Morse theoretic model of the loop space, we have that the geodesics on $M_K$ are precisely critical points of the energy functional, with finite dimensional unstable manifolds, and infinite dimensional stable manifolds. There is a one-to-one correspondence between Reeb chords and oriented geodesics. Assuming that the metric is generic gives moreover that Reeb chords of degree $-\lambda$ are in one-to-one correspondence with geodesics of index $\lambda$ (see \cref{lma:1-1-correspondence_geodeiscs_reeb_of_index_k}). We will show that the evaluation map defined in \cref{sub:the_evaluation_map_etc} is transverse to the infinite dimensional stable manifolds, and that the kernel of the linearized operator $D_u$ has the same dimension as the unstable manifold.

	In \cref{sub:filtration_on_cw} we will define the action filtration on $CW^\ast_{\varLambda_K}(F,F)$, followed by \cref{sub:filtration_on_based_loops} where we first replace the full Moore loop space with the Morse theoretic model consisting of piecewise geodesic loops, and then we filter the space of loops by length. In \cref{sub:the_chain_map_Y1_respects_the_filtration} we prove that $\varPsi_1$ respects the action filtrations and in fact that $\varPsi_1$ is diagonal with respect to the action filtrations. In \cref{sub:the_filtration_argument} we prove that $CW^\ast_{\varLambda_K}(F,F)$ is isomorphic to the Morse theoretic model of the loop space in each filtration level, which allows us to pass to colimits.

	Consider $M_K \subset W_K$ and fix a generic Riemannian metric $g$ on $M_K$ such that in the handle $D_\varepsilon T^\ast([0,\infty) \times \varLambda_K)$ of $W_K$, the metric has the form 
	\[
		dt^2 + f(t) g\, ,
	\]
	where $t$ is the coordinate in the $[0,\infty)$-factor, and $f\co [0,\infty) \longrightarrow [0,\infty)$ satisfies $f'(0) = -1$, $f'(t) < 0$ and $f''(t) \geq 0$, see \eqref{eq:metric_definition} for details.
	\subsection{Length filtration on $CW^{\ast}_{\varLambda_{K}}(F,F)$}
		\label{sub:filtration_on_cw}
		For a Reeb chord generator $c\in CW^\ast_{\varLambda_K}(F,F)$ define its \emph{action} by
		\[
			\mathfrak a(c) \defeq \int_0^\ell c^\ast \lambda \, .
		\]
		In our case with $F \cong T^\ast_\xi S \subset W_K$ for $\xi\in M_K$ we only have a single Lagrangian intersection generator $\xi$, whose action we define explicitly as $\mathfrak a(\xi) \defeq 0$.
		We then filter $CW^\ast_{\varLambda_K}(F,F)$ by this action, and use the notation
		\[
			\mathcal F_p CW^\ast_{\varLambda_K}(F,F) \defeq \parenm{ c\in CW^\ast_{\varLambda_K}(F,F) \suchthat \mathfrak a(c) < p}\, .
		\]
		Now, by applying Stokes' theorem to any $J$-holomorphic disk which contributes to $\mu^1(c)$ we get the following lemma. (Compare with e.g.\@ \cite[Lemma B.3]{ekholm2006rational}.)
		\begin{lma}
			The differential $\mu^1 \co CW^\ast_{\varLambda_K}(F,F) \longrightarrow CW^\ast_{\varLambda_K}(F,F)$ does not increase the action of generators. That is,
			\[
				\mathfrak a(c) \geq \mathfrak a(\mu^1(c))\, ,
			\]
			for any $c\in CW^\ast_{\varLambda_K}(F,F)$.
		\end{lma}
	\subsection{Length filtration on $C_{-\ast}(\varOmega_{\xi} M_{K})$}
		\label{sub:filtration_on_based_loops}
		In this section we review basic material on the Morse theory of loop spaces from \cite{milnor1963morse}. 

		One goal in this section is to replace the full Moore loop space $\varOmega_{\xi} M_K$ with a homotopy equivalent Morse theoretic model by approximating Moore loops by piecewise geodesic loops. The second goal is to in detail define the filtration on the model of chains of based loops we use.

		By abuse of notation, we denote by $\varOmega_{\xi} M_{K}$ the space of continuous based loops $\gamma\co [0,1] \longrightarrow M_{K}$ with fixed domain $[0,1]$. It is homotopy equivalent with the space of Moore loops as defined in \cref{sub:based_loops_on_}. With respect to the generic Riemannian metric $h$ on $M_K$ as described in \eqref{eq:metric_definition}, equip $\varOmega_\xi M_{K}$ with the supremum metric
		\[
			d^{\ast}(\gamma,\beta) \defeq \sup_{t\in [0,1]} h(\gamma(t),\beta(t)), \quad \gamma,\beta\in \varOmega_\xi M_{K}\, .
		\]
		The metric topology on $\varOmega_\xi M_{K}$ induced by $d^\ast$ then agrees with the compact-open topology. Define $\varOmega^{\text{pw}}M_{K}$ as the space of piecewise smooth loops, and equip it with the metric
		\[
			d(\gamma,\beta) \defeq d^{\ast}(\gamma,\beta) + \paren{\int_{0}^{1} \abs{\dot \gamma}^{2} - \abs{\dot \beta}^{2} dt}^{\frac 12},\quad \gamma,\beta \in \varOmega^{\text{pw}} M_{K}\, .
		\]
		By~\cite[Theorem 17.1]{milnor1963morse}, we have that the inclusion $i \co \varOmega^{\text{pw}} M_{K} \longrightarrow \varOmega_\xi M_{K}$ is a homotopy equivalence. We define the energy of $\gamma\in \varOmega^{\text{pw}}M_{K}$ by
		\begin{equation}\label{eq:energy_functional}
			E(\gamma) \defeq \int_{0}^{1} \abs{\dot \gamma}^{2} dt\, .
		\end{equation}
		Similarly we define the length of $\gamma \in \varOmega^{\text{pw}}M_K$ as 
		\[
			L(\gamma) \defeq \int_0^1 \abs{\dot \gamma} dt\, .
		\]
		Define
		\[
			\varOmega^{\text{pw},c} M_{K} \defeq \parenm{\gamma \in \varOmega^{\text{pw}}M_K \suchthat E(\gamma) < c^2}\, .
		\]
		Fix a subdivision of $[0,1]$,
		\[
			0 = t_{0} < t_{1} < t_{2} < \cdots < t_{m} = 1\, .
		\]
		Then define $BM_{K}$ to be the set of loops in $\varOmega^{\text{pw}}M_{K}$ that are geodesic in the time interval $[t_{i},t_{i+1}]$ for each $i\in \parenm{0,\ldots,m-1}$. Let 
		\[
			B^c M_K \defeq \parenm{\gamma\in BM_K \suchthat E(\gamma) < c^2}\, .
		\]
		Applying~\cite[Lemma 16.1]{milnor1963morse} then gives that for a sufficiently fine subdivision, $B^{c}M_{K}$ is a smooth finite dimensional manifold which is a natural submanifold of $(M_K)^{m-1}$. Moreover by~\cite[Theorem 16.2]{milnor1963morse}, $B^{c}M_{K}$ is a deformation retract of $\varOmega^{\text{pw},c}M_{K}$, and critical points of $\eval[0]E_{\varOmega^{\text{pw},c}M_{K}}$ are the same as the critical points of $\eval[0]E_{B^{c}M_{K}}$, and $\eval[0]E_{B^{c}M_{K}}$ is furthermore a Morse function.
	
		We consider another increasing filtration on $BM_K$ by filtering by length. Namely, define the length filtration of $BM_K$ by
		\[
			\mathcal F_c BM_K = \parenm{\gamma \in BM_K \suchthat L(\gamma) < c}\, ,
		\]
		and correspondingly
		\[
			\mathcal F_c \varOmega^{\text{pw}}M_K = \parenm{\gamma \in \varOmega^{\text{pw}}M_K \suchthat L(\gamma) < c}\, .
		\]

		By the same proof as~\cite[Theorem 16.2]{milnor1963morse}, we construct an explicit deformation retract of $\mathcal F_c \varOmega^{\text{pw}}M_K$ onto $\mathcal F_c BM_K$ (see \cref{lma:pw_loops_deformation_retract} below).
		
		If $\sigma \in C_{-k}(BM_K)$ is a cubical $k$-chain of piecewise geodesic loops, we define the action of $\sigma$ as
		\[
			\mathfrak a(\sigma) \defeq \max_{x\in [0,1]^k} L(\sigma(x))\, .
		\]
		We then define
		\[
			\mathcal F_c C_{-{\ast}}(BM_K) \defeq \parenm{\sigma \in C_{- \ast}(BM_K) \suchthat \mathfrak a(\sigma) < c}\, ,
		\]
		which gives us an increasing filtration on $C_{-{\ast}}(BM_K)$. Futhermore, we see by definition that $\mathfrak a(\dd \sigma) \leq \mathfrak a(\sigma)$.

		\begin{lma}\label{lma:pw_loops_deformation_retract}
			There is a deformation retract
			\[
				r\co \mathcal F_c\varOmega^{\text{pw}}M_K \longrightarrow \mathcal F_c BM_K\, ,
			\]
			which therefore induces a quasi-isomorphism
			\[
				r_\ast \co \mathcal F_c C_{-\ast}(\varOmega^{\text{pw}}M_K) \longrightarrow \mathcal F_c C_{-\ast}(BM_K)\, .
			\]
		\end{lma}
		\begin{proof}
			From the proof of~\cite[Theorem 16.2]{milnor1963morse}, we first define a retraction
			\[
				r\co \mathcal F_c\varOmega^{\text{pw}}M_K \longrightarrow \mathcal F_c BM_K\, ,
			\]
			as follows. Consider the closed ball with center $\xi\in M_K$ and radius $c$
			\[
				B(\xi, c) = \parenm{x \in M_K \suchthat h(x,\xi) \leq c}\, .
			\]
			
			For any $\gamma \in \mathcal F_c \varOmega^{\text{pw}}M_K$, fix a fine enough subdivision of $[0,1]$
			\[
				0 = t_0 < t_1 < \cdots < t_{k-1} < 1 = t_k\, ,
			\]
			so that $h(\gamma(t_{i-1}), \gamma(t_i)) < \varepsilon$ for some $\varepsilon > 0$ small enough so that there is a unique geodesic between $\gamma(t_{i-1})$ and $\gamma(t_i)$. Because $\gamma$ is contained in the ball $B(\xi,c)$, we have by~\cite[Corollary 10.8]{milnor1963morse} that there is a unique minimal geodesic between $\gamma(t_{i-1})$ and $\gamma(t_i)$ of length less than $\varepsilon$. Define $r(\gamma)$ so that for each $i \in \parenm{1,\ldots,k-1}$ we have
			\[
				\eval[0]{r(\gamma)}_{[t_{i-1},t_i]} = \text{ unique minimal geodesic of length less than } \varepsilon \text{ from } \gamma(t_{i-1}) \text{ to } \gamma(t_i)\, .
			\]
			Since geodesics are locally length minimizing, it is clear that $L(\gamma) \geq L(r(\gamma))$ and therefore that $r$ takes values in $\mathcal F_c BM_K$. For each $s\in[0,1]$ we define
			\[
				r_s \co \mathcal F_c \varOmega^{\text{pw}}M_K \longrightarrow \mathcal F_c BM_K\, ,
			\]
			in such a way that for $s\in [t_{i-1},t_i]$ and any $i\in \parenm{1,\ldots,k-1}$ the map $r_s$ is so that
			\[
				\begin{cases}
					\eval[0]{r_s(\gamma)}_{[0,t_{i-1}]} = \eval[0]{r(\gamma)}_{[0,t_{i-1}]} \\
					\eval[0]{r_s(\gamma)}_{[t_{i-1},s]} = \text{unique minimal geodesic from } \gamma(t_{i-1}) \text{ to } \gamma(s)\\
					\eval[0]{r_s(\gamma)}_{[s,1]} = \eval[0]{\gamma}_{[s,1]} \, .
				\end{cases}
			\]
			Then $\begin{cases}
				r_0(\gamma) = \gamma \\
				r_1(\gamma) = r(\gamma)
			\end{cases}$ and it is continuous in both $s$ and $\gamma$. Hence it shows that $\mathcal F_c BM_K$ is a deformation retract of $\mathcal F_c \varOmega^{\text{pw}}M_K$.

			It is now straightforward to see that this map is defined on singular chains. Namely, for any fixed $c > 0$, we pick a fine enough subdivision of $[0,1]$
			\[
				0 = t_0 < t_1 < \cdots < t_{N-1} < 1 = t_N\, ,
			\]
			so that for every $i\in \parenm{1,\ldots,N-1}$ we have
			\[
				\max_{x\in[0,1]^k} L \paren{\eval[0]{\sigma(x)}_{[t_{i-1},t_i]}} < \varepsilon\, .
			\]
			Hence for any $x\in [0,1]^k$, there is a unique geodesic from $\sigma(x)(t_{i-1})$ to $\sigma(x)(t_i)$. Then $r$ induces a map
			\begin{align}\label{eq:proof_r_ast_deformation_retract}
				r_\ast \co \mathcal F_c C_{-\ast}(\varOmega^{\text{pw}}M_K) &\longrightarrow \mathcal F_c C_{-\ast}(BM_K) \\
				\sigma &\longmapsto r \circ \sigma \nonumber \, .
			\end{align}
		\end{proof}

		By \cite[Theorem 16.3]{milnor1963morse}, $BM_K$ is a CW-complex with one cell of dimension $\lambda$ for each closed geodesic on $M_K$ of index $\lambda$. We consider the cellular chain complex $C^{\mathrm{cell}}_{-\ast}(BM_K)$. We think of the generators of $C_{-\lambda}^{\mathrm{cell}}(BM_K)$ as the unstable manifolds of geodesics of index $\lambda$ with respect to the energy functional $E$ on $BM_K$. We define the action of a $\lambda$-cell $e_\lambda$ as
		\[
			\mathfrak a(e_\lambda) = \max_{x\in [0,1]^\lambda} L(e_\lambda(x))\, .
		\]
		It is well known that singular chains and cellular chains on a CW-complex are homotopy equivalent. Denote the induced isomorphism on homology by
		\begin{equation}\label{eq:iso_cell_sing}
			s\co H_{-\ast}(BM_K) \xrightarrow{\cong} H_{-\ast}^{\text{cell}}(BM_K)\, .
		\end{equation}
		In particular by \cref{lma:pw_loops_deformation_retract} the map $r_\ast$ in \eqref{eq:proof_r_ast_deformation_retract} induces an isomorphism
		\begin{equation}\label{eq:r_deformation_retract_quiso}
			r_\ast\co H_{-\ast}(\varOmega^{\text{pw}}M_K) \xrightarrow{\cong} H_{-\ast}(BM_K)\, .
		\end{equation}
		
	\subsection{The chain map $\varPsi_1$ respects the filtration}
		\label{sub:the_chain_map_Y1_respects_the_filtration}
		The goal for this section is to prove that the chain map
		\[
			\varPsi_1 \co CW^\ast_{\varLambda_K}(F,F) \longrightarrow C_{-\ast}(\varOmega_{\xi}M_K)\, ,
		\]
		respects the filtrations $\mathcal F_c$ defined on $CW^\ast_{\varLambda_K}(F,F)$ and $C_{-\ast}(\varOmega_{\xi}M_K)$ in \cref{sub:filtration_on_cw,sub:filtration_on_based_loops} respectively. The plan is to follow and adapt the proof of \cite[Proposition 8.9]{cieliebak2017knot} to the current situation. The outline of the proof is to consider any $J$-holomorphic disk $u\in \overline{\mathcal M}(a)$ contributing to $\varPsi_1(a)$ and integrate the 2-form $d \lambda_\tau$ (defined in \eqref{eq:pullback_of_beta_tau} below) over the disk. Using Stokes' theorem we show that $0\leq \int_{u^{-1}(W_K)} u^{\ast}d \lambda_\tau = \mathfrak a(a) - L(\gamma)$.

		Consider a generator $a\in CW^\ast_{\varLambda_K}(F,F)$ and pick some loop $\gamma = \varPsi_1(a)(x) \co [0,1] \longrightarrow M_K$. Then pick a tubular neighborhood $N(M_K)$ of $M_K$ in $W_K$ and a symplectomorphism
		\begin{equation}\label{eq:weinstein_neighborhood_dfn}
			\varphi\co N(M_K) \longrightarrow D_\delta T^\ast M_K\, ,
		\end{equation}
		by the Lagrangian neighborhood theorem for some positive constant $\delta$. By a similar argument to that of the proof of \cite[Theorem 7.1]{weinstein1971symplectic}, we may assume that $\varphi$ sends the fiber $F \cap N(M_K)$ to a fiber of $D_\delta T^\ast M_K$.

		Recall that we use the metric on $M_K$ defined in \eqref{eq:metric_definition}. Pick coordinates $(q,p)$ in $T^\ast M_K$, and define the canonical 1-form $\beta = pdq$. Then let $\beta_1$ be a $1$-form on $T^\ast M_K$ that is given by
		\[
			\beta_1 = \frac{\delta p dq}{\abs p}\, .
		\]
		When we restrict to $S_\delta T^\ast M_K$, the Reeb vector field $R = p\dd_q$ and the contact structure $\xi = \ker \beta_1$ have the following expressions in these coordinates
		\[
			R = \sum_{i=1}^n p_i \dd_{q_i}, \quad \xi = \ker \beta_1 \cap \ker(pdp) = \paren{\mathrm{span} \parenm{R, p\dd_p}}^{\perp_{d \beta_1}}\, .
		\]
		Then we have the splitting $T_{(q,p)}T^\ast M_K = \mathrm{span}\parenm{R, p \dd_p} \oplus \xi$. We have picked an almost complex structure $J$ on $W_K$ which is compatible with $d \lambda$. The almost complex structure $J$ induces an almost complex structure $J'$ on $T^\ast M_K$ defined as
		\[
			J' \defeq (d \varphi) \circ J \circ (d \varphi)^{-1}\, ,
		\]
		which satisfies the following:
		\begin{enumerate}
			\item $J'$ is compatible with $dp \wedge dq$, and
			\item $J'$ preserves the splitting $T_{(q,p)}T^\ast M_K= \mathrm{span}\parenm{R, p \dd_p} \oplus \xi$.
		\end{enumerate}
		These two conditions ensure that the map 
		\[
			\varphi \circ u \co (u^{-1}(N(M_K)), j) \longrightarrow (D_\delta T^\ast M_K, J')
		\]
		is $J'$-holomorphic.

		By the proof of \cite[Lemma 8.8]{cieliebak2017knot} we have that $d \beta_1(v, J'v) \geq 0$. However, if we integrate $d \beta_1$ over the domain of $u\in \overline{\mathcal M}(a)$ we can not use Stokes' theorem directly since $\beta_1$ is singular along the zero section, so we have to make some further modifications to get rid of this singularity.

		Let
		\begin{equation}\label{eq:def_tau}
			\tau\co [0,\infty) \longrightarrow [0,1] \, ,
		\end{equation}
		be a smooth function so that
		\begin{itemize}
			\item $\tau(s) = 0$ near $s = 0$, and
			\item $\tau'(s) \geq 0$ for every $s$,
			\item $\tau(s) = 1$ for $s\geq \varepsilon$ for some small $\varepsilon < \delta$.
		\end{itemize}
		Then define
		\[
			\beta_{\tau} \defeq \frac{\delta \tau(\abs p)}{\abs p} pdq\, .
		\]
		\begin{lma}[{\cite[Lemma 8.8]{cieliebak2017knot}}]\label{lma:d_beta_tau_is_non-negative_outside_zero_sect}
			For any $v\in T_{(q,p)}T^\ast M_K$ outside of the zero section we have
			\[
				d \beta_{\tau}(v,J'v) \geq 0\, .
			\]
			For $\tau(\abs p) > 0$ and $\tau'(\abs p) > 0$ equality holds if and only if $v = 0$, whereas at points where $\tau(\abs p) > 0$ and $\tau'(\abs p) = 0$ equality holds if and only if $v$ is a linear combination of the Liouville vector field $p\dd_p$ and the Reeb vector field $R = p\dd_q$.
		\end{lma}
		
		Let $a\in CW^\ast_{\varLambda_K}(F,F)$ be a generator and consider $u\co D_3 \longrightarrow W_K$ in $\overline{\mathcal M}(a)$. Denote by $\gamma \defeq \ev(u)$. Using the symplectomorphism $\varphi$ in \eqref{eq:weinstein_neighborhood_dfn} we define an exact 2-form on $W_K$. Define
		\begin{equation}\label{eq:pullback_of_beta_tau}
			d\lambda_\tau \defeq \varphi^\ast d\beta_\tau\, ,
		\end{equation}
		 on $N(M_K) \subset W_K$. We may extend $d \lambda_\tau$ to the whole of $W_K$ by defining it to be
		\[
			d \lambda_\tau = \begin{cases}\varphi^\ast d \beta_\tau, & \text{in } N(M_K) \\ \omega, & \text{otherwise.}\end{cases}
		\]
		\begin{lma}\label{lma:d_lambda_tau_is_non-negative}
			The 2-form $d \lambda_\tau$ on $W_K$ defined above satisfies
			\[
				d \lambda_\tau(v,Jv) \geq 0\, .
			\]
		\end{lma}
		\begin{proof}
		In $N(M_K)$ we have $d \lambda_\tau = \varphi^\ast d \beta_\tau$, in which case the conclusion follows from \cref{lma:d_beta_tau_is_non-negative_outside_zero_sect}. Otherwise we have $d \lambda_\tau = \omega$ which is non-negative on complex lines, because $J$ is $\omega$-compatible.
		\end{proof}
		\begin{lma}\label{lma:image_of_fiber_is_exact_wrt_beta_1}
			Consider the exact Lagrangian fiber $F \cap N(M_K)$. Its image $F' \defeq \varphi(F \cap N(M_K)) \subset D_{\delta}T^\ast M_K$ under $\varphi$ is exact with respect to $\beta_1$.
		\end{lma}
		\begin{proof}
			This follows immediately from the assumption that $\varphi$ maps $F \cap N(M_K)$ to a fiber of $D_\delta T^\ast M_K$, say $F' = D_\delta T^\ast_xM_K$ for $x\in M_K$.
		\end{proof}

		\begin{prp}\label{prp:action_length_estimate}
			Let $a\in CW^\ast_{\varLambda_K}(F,F)$ be any generator and $u$ be any $J$-holomorphic half strip with positive puncture at $a$. Letting $\gamma \defeq \ev(u)$ we have
			\[
				\mathfrak a(a) \geq L(\gamma)\, ,
			\]
			with equality if and only if $u$ is a branched covering of a half strip over a Reeb chord.
		\end{prp}
		\begin{proof}
			Since $d \lambda_\tau(u,Ju) \geq 0$ by \cref{lma:d_lambda_tau_is_non-negative} we integrate it over the disk $u\co D_3 \longrightarrow W_K$ and use Stokes' theorem:
			\begin{align}\label{eq:integral_split_up_into_two_parts}
				0\leq \int_{u^{-1}(W_K)} u^\ast d \lambda_\tau &= \int_{u^{-1}(W_K \setmin N(M_K))} u^\ast d \lambda_\tau + \int_{u^{-1}(N(M_K))} u^\ast d \lambda_\tau \nonumber\\
				&= \underbrace{\int_{u^{-1}(W_K \setmin N(M_K))} u^\ast \omega}_{\eqdef I_{1}} + \underbrace{\int_{(\varphi \circ u)^{-1}(D_\delta T^\ast M_K)} (\varphi\circ u)^\ast d \beta_\tau}_{\eqdef I_{2}} \, .
			\end{align}

			For the remainder of this proof we follow the proof of \cite[Proposition 8.9]{cieliebak2017knot}. We start by computing $I_2$. To do this we consider $\beta_1 = \frac{\delta p dq}{\abs p}$. Then pick a biholomorphism
			\[
				\psi\co [0,\delta_0] \times [0,1] \longrightarrow U \subset D_{3}\, ,
			\]
			where $U \subset D_{3}$ is a neighborhood of the boundary arc between the boundary punctures $\zeta_\pm$ both of which are mapped to $\xi\in M_K \subset W_K$, so that $\psi(0,t)$ is a parametrization of the boundary arc between $\zeta_-$ and $\zeta_+$. We choose $\delta_0$ small enough so that $(\varphi \circ u \circ \psi) (\delta_0,t)$ does not hit $M_K \subset T^\ast M_K$. Let
			\[
				q(t) \defeq \varphi\circ u \circ \psi (0,t)\, .
			\]
			Since we have a non-flat metric $h$ on $M_K$ (see \eqref{eq:metric_definition}) we consider the splitting $T(T^\ast M_K) \cong V \oplus H$ and geodesic normal coordinates $(q,p)$ on $T^\ast M_K$. The almost complex structure $J$ then takes the vertical subspace to the horizontal and vice versa. Consider the Levi-Civita connection on $T(T^\ast M_K)$, and denote its associated Christoffel symbols by $\varGamma^k_{ij}$. Recall that in geodesic normal coordinates, the metric tensor at $(q,p)$ has components $h_{ij}(q,p) = \delta_{ij}$, where $\delta_{ij}$ is the Kronecker delta. In particular the Christoffel symbols vanish at $(q,p)$. For any $x$ in a neighborhood of $(q,p)$ it follows that $\varGamma^k_{ij}(x) = O(\abs x)$.

			The almost complex structure in a neighborhood of $(q,p)$ is
			\[
				\begin{cases}
					J(\dd_{p_i}) = \dd_{q_i} - \varGamma^k_{ij} p^j \dd_{p_k} \\ 
					J(\dd_{q_i}) = -\dd_{p_i} + \varGamma^k_{ij} p^j \dd_{q_k} - \varGamma^{m}_{ij} \varGamma^k_{mn} p^j p^n \dd_{p_k}\, .
				\end{cases}
			\]
			Since $u$ is $J$-holomorphic, we write
			\[
				\tilde u(s,t) \defeq (\varphi \circ u \circ \psi)(s,t) = (Q(s,t),P(s,t))\, ,
			\]
			where
			\[
				\begin{cases}
					\dd_s Q^k - \varGamma^k_{ij} \dd_t Q^i P^j + \dd_t P^k = 0 \\
					\dd_s P^k - \dd_t Q^k + \varGamma^k_{ij} \dd_t P^i P^j - \varGamma^{m}_{ij} \varGamma^k_{mn} \dd_t Q^i \dd_t Q^i P^j P^n = 0\, .
				\end{cases}
			\]
			Recall that in our geodesic normal coordinates we have $\varGamma^k_{ij}(x) = O(\abs x)$ where $x$ is in a neighborhood of $(q,p)$ , and hence with $x = \tilde u(s,t)$ we have
			\begin{equation}\label{eq:components_of_u_tilde}
				\begin{cases}
					\dd_s Q + \dd_t P + O(\abs x) = 0 \\
					\dd_s P - \dd_t Q + O(\abs x) = 0\, .
				\end{cases}
			\end{equation}
			If we write $Q(s,t) = q(t) + v(s,t)$ we get from the the second equation in \eqref{eq:components_of_u_tilde} that
			\[
				P(s,t) = s \paren{\dot q(t) + O(\abs x)} + w(s,t)\, ,
			\]
			where $w(s,t) \defeq \int_0^s \dd_t v(\sigma,t) d \sigma$. We now have $v(0,t) = 0 = w(0,t)$ and hence $\pd vt(0,t) = 0 = \pd wt(0,t)$. Setting $s = 0$ in \eqref{eq:components_of_u_tilde} gives $\pd vs(0,t) = O(\abs{\tilde u(0,t)}) = \pd ws(0,t)$. Next, from Taylor's formula we get $\pd vt (\delta_0,t) = O(\delta_0)$ and $w(\delta_0,t) = \delta_0 O(\abs{\tilde u(0,t)}) +O(\delta_0^2)$. Then we get
			\begin{align}\label{eq:computations_of_pullback_of_beta_1_close_to_geodesic}
				\eval[0]{\tilde u^\ast \beta_1}_{s = \delta_0} &= \frac{(\delta_0 \dot q(t) + w(\delta_0,t))(\dot q(t) + \pd vt(\delta_0,t))}{\abs{\delta_0 \dot q(t) + w(\delta_0,t)}}dt \\
				&= \frac{\ip{\delta_0 \dot q(t) + \delta_0 O(\abs{\tilde u(0,t)}) + O(\delta_0^2)}{\dot q(t) + O(\delta_0)}}{\abs{\delta_0 \dot q(t) + \delta_0 O(\abs{\tilde u(0,t)}) + O(\delta_0^2)}}dt \nonumber \\
				&= \frac{\ip{ \dot q(t) + O(\abs{\tilde u(0,t)}) + O(\delta_0)}{\dot q(t) + O(\delta_0)}}{\abs{\dot q(t) + O(\abs{\tilde u(0,t)}) + O(\delta_0)}}dt = \paren{\abs{\dot q(t)} + O(\delta_0)}dt \nonumber
			\end{align}
			Next, pick $\varepsilon > 0$ so that it is smaller than the minimal norm of the $p$-components of $(\varphi\circ u \circ \psi)(\delta_0,t)$ and pick a function $\tau \co [0,\infty) \longrightarrow [0,1]$ as in \eqref{eq:def_tau}. Namely, $\tau$ satisfies
			\[
				\begin{cases}
					\tau'(s)\geq 0, & \forall s \in [0,\infty) \\
					\tau(s) = 0, & \text{near } s = 0\\
					\tau(s) = 1, & s\geq \varepsilon\, .
				\end{cases}
			\]
			Consider $\beta_\tau = \frac{\delta \tau(\abs p)pdq}{\abs p}$. By \cref{lma:d_beta_tau_is_non-negative_outside_zero_sect} we have $(\varphi\circ u)^\ast d \beta_\tau \geq 0$, and also that $\beta_\tau$ agrees with $\beta_1$ in the set $\parenm{\abs p \geq \varepsilon} \subset T^\ast M_K$. Then we get
			\[
				\lim_{\delta_0 \to 0}\int_{\parenm{\delta_0} \times [0,1]} (\varphi \circ u \circ \psi)^\ast \beta_\tau = \lim_{\delta_0 \to 0}\int_{\parenm{\delta_0} \times [0,1]} \abs{\dot q(t)} + O(\delta_0) dt = \lim_{\delta_0 \to 0} L(\gamma) + O(\delta_0) = L(\gamma)\, .
			\]
			\begin{figure}[H]
				\centering
				\includegraphics{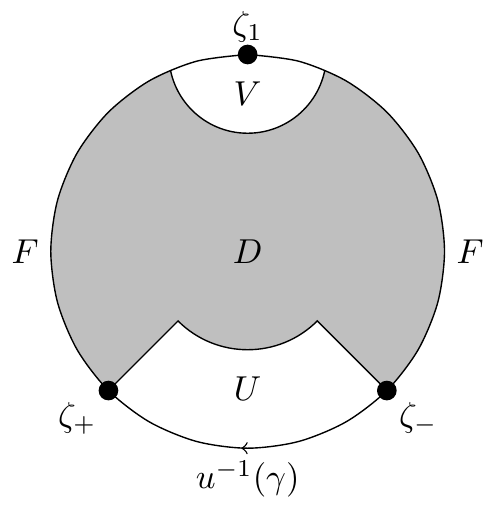}
				\caption{Domain of the $J$-holomorphic disk $u$ with neighborhoods around the outgoing segment between $\zeta_-$ and $\zeta_+$ and the positive puncture $\zeta_1$ marked in white.}
			\end{figure}
			By \cref{lma:image_of_fiber_is_exact_wrt_beta_1} we have that $F' = \varphi(F \cap N(M_K)) \subset D_\delta T^\ast M_K$ is exact with respect to $\beta_\tau$. Therefore we get
			\begin{align}\label{eq:computation_of_two}
				I_2 &= \int_{(\varphi \circ u)^{-1}(D_\delta T^\ast M_K)} (\varphi\circ u)^\ast d \beta_\tau \nonumber\\
				&= \int_{(\varphi \circ u)^{-1}(S_\delta T^\ast M_K)} (\varphi \circ u)^\ast \beta_\tau - \lim_{\delta_0 \to 0}\int_{\parenm{\delta_0} \times [0,1]} (\varphi \circ u \circ \psi)^\ast \beta_\tau \nonumber\\
				&= \int_{(\varphi \circ u)^{-1}(S_\delta T^\ast M_K)} (\varphi \circ u)^\ast \beta_\tau - L(\gamma)\, .
			\end{align}
			Finally, the integral $I_1$ in \eqref{eq:integral_split_up_into_two_parts} is computed by using Stokes' theorem and that $d \lambda_\tau = \omega$ outside of $N(M_K)$ by definition.
			\begin{align}\label{eq:computation_of_one}
				I_1 &= \int_{u^{-1}(W_K \setmin N(M_K))} u^\ast \omega = \mathfrak a(a) - \int_{u^{-1}(\dd N(M_K))} u^\ast \lambda
			\end{align}
			By combining \eqref{eq:computation_of_one} with \eqref{eq:computation_of_two} we get
			\begin{equation}\label{eq:computation_of_total_integral}
				0\leq \int_{u^{-1}(W_K)} u^\ast d \lambda_\tau =  - \int_{u^{-1}(\dd N(M_K))} u^\ast \lambda + \int_{(\varphi \circ u)^{-1}(S_\delta T^\ast M_K)} (\varphi \circ u)^\ast \beta_\tau+\mathfrak a(a)-L(\gamma) \, .
			\end{equation}
			Note that along $S_\delta T^\ast M_K$ we have $\beta_\tau = \beta = pdq$. Furthermore $\varphi$ is an exact symplectomorphism so we have $\varphi^\ast \beta - \lambda = d \theta$. Hence
			\[
				\int_{u^{-1}(\dd N(M_K))} u^\ast(\varphi^\ast \beta - \lambda) = \int_{u^{-1}(\dd N(M_K))} u^\ast d \theta = 0\, ,
			\]
			and therefore \eqref{eq:computation_of_total_integral} turns into
			\[
				0 \leq \int_{u^{-1}(W_K)} u^{\ast}d \lambda_\tau = \mathfrak a(a) - L(\gamma) \Leftrightarrow \mathfrak a(a) \geq L(\gamma)\, .
			\]
		\end{proof}
		\begin{cor}\label{cor:chain_map_preserves_filtrations}
			Let $a\in CW^\ast_{\varLambda_K}(F,F)$ be any generator and let $u\in \overline{\mathcal M}(a)$. Then
			\[
				\mathfrak a(a) \geq \mathfrak a(\varPsi_1 a)\geq \mathfrak a((r_\ast \circ \varPsi_1)(a))\, .
			\]
		\end{cor}
		\begin{proof}
			Fix a generator $a\in CW^\ast_{\varLambda_K}(F,F)$ and consider the moduli space $\overline{\mathcal M}(a)$. The action of $\varPsi_1(a) \in C_{-\ast}(\varOmega_\xi M_K)$ is
			\[
				\mathfrak a(\varPsi_1(a)) = \max_{x\in [0,1]^\ast} L(\varPsi_1(a)(x))\, .
			\]
			Note that the maximum is well-defined by the compactness of $[0,1]^\ast$. Let $x_{\mathrm{max}} \in [0,1]^{\ast}$ be such that $L(\varPsi_1(a)(x_{\mathrm{max}})) = L(\varPsi_1(a))$, and let $\gamma_{\mathrm{max}} \defeq \varPsi_1(a)(x_{\mathrm{max}})$. Since \cref{prp:action_length_estimate} holds for any $\gamma\in \varOmega_\xi M_K$, in particular it holds for $\gamma_{\mathrm{max}}$. Therefore
			\[
				\mathfrak a(a) \geq L(\gamma_{\mathrm{max}}) = \max_{x\in [0,1]^\ast} L(\varPsi_1(a)(x)) = \mathfrak a(\varPsi_1 a)\, .
			\]

			Moreover, the inequality $\mathfrak	a(\varPsi_1 a) \geq \mathfrak	 a((r_\ast \circ \varPsi_1)(a))$ holds because $r_\ast$ does not increase filtration (see the proof of \cref{lma:pw_loops_deformation_retract}).
		\end{proof}
	\subsection{The chain map $\varPsi_1$ is diagonal with respect to the action filtrations}
		\label{sub:the_chain_map_is_diagonal_with_respect_to_the_action_filtrations}
		In this section we prove that $\varPsi_1$ is diagonal with respect to the action filtrations, that is $\mathfrak a(a) = \mathfrak a(\varPsi_1 a)$.

		We first give a brief outline of the proof. We consider the \emph{trivial} $J$-holomorphic half strip $u_0$ whose image is the cone over the Reeb chord $a \in CW^\ast_{\varLambda_K}(F,F)$ and whose tangent space at $(q,p)$ is spanned by $p\dd_q$ and $p\dd_p$ in geodesic normal coordinates. We show that $u_0$ is transversely cut out, and therefore by \cref{lma:d_beta_tau_is_non-negative_outside_zero_sect} together with the proof of \cref{prp:action_length_estimate} we get
		\[
			0 = \int_{u_0^{-1}(W_K)} u_{0}^{\ast}d \lambda_\tau = \mathfrak a(a) - L(\gamma_0)\, .
		\]
		To prove that $u_0$ is transversely cut out, we choose a generic Riemannian metric $g$ on $M_K$ (see \eqref{eq:metric_definition}). Then we have a one-to-one correspondence between Reeb chords of degree $-\lambda$ and geodesics of index $\lambda$ (see \cref{lma:1-1-correspondence_geodeiscs_reeb_of_index_k} below). We consider vector fields $v\in \ker D_{u_0}$ in the kernel of the linearized Cauchy--Riemann operator at $u_0$. Then we show that $v$ restricts to broken Jacobi fields along $\gamma$ for which the Hessian of the energy functional \eqref{eq:energy_functional} is negative definite.

		The following lemma is essentially found and proven in \cite[Prop 6.38]{robbin1995spectral} and \cite{duistermaat1976morse}. Recall that the degree of Reeb chords is defined via the Conley--Zehnder index, see \cref{rmk:description_of_maslov_index_of_reeb_chord_generators} for details.
		\begin{lma}\label{lma:1-1-correspondence_geodeiscs_reeb_of_index_k}
			Let $a < b$ be two real numbers. There is a one-to-one correspondence between Reeb chords $a$ of degree $-\lambda$ with action $\mathfrak a(a) = A$ and geodesics $\gamma$ in $BM_K$ of index $\lambda$ with length $L(\gamma) = A$.
		\end{lma}
		\begin{proof}
			It is a consequence of the first part of \cref{prp:chain_map_is_diagonal_wrt_filtrations} and in particular \eqref{eq:diagonal_wrt_filtration} (which do not depend on the current lemma) that Reeb chords with action $A$ are in one-to-one correspondence with geodesics in $BM_K$ with length $A$. What is left to show is that this one-to-one correspondence also preserves degree/index.

			Let $\gamma\in BM_K$ be a (non-broken) geodesic. By Morse theory on the loop space, the index of $\gamma$ is defined as the Morse index of the energy functional \eqref{eq:energy_functional}. Morse's index theorem \cite[Theorem 15.1]{milnor1963morse} says that the index of $\gamma$ is equal to the number of points $\gamma(t)$ for $t\in (0,1)$, which is conjugate to $\gamma(0)$ along $\gamma$, counted with multiplicity. Recall that $\gamma(t)$ is conjugate to $\gamma(0)$ along $\gamma$ by definition, if there is a Jacobi field $K$ along $\gamma$ so that $K(t) = K(0) = 0$. A Jacobi field is a vector field along $\gamma$ satisfying the Jacobi equation
			\[
				\nabla_{\frac{d}{dt}}\nabla_{\frac{d}{dt}} K + R(\dot \gamma, K) \dot \gamma = 0\, ,
			\]
			where $\nabla$ is the Levi-Civita connection on the bundle $\gamma^\ast TM_K$, and $R$ is the corresponding curvature tensor. The geodesic flow on $M_K$ lifts to the Reeb flow on $ST^\ast M_K$. Therefore Jacobi fields --- which are seen as linearizations of the geodesic flow --- lift to the linearized Reeb flow.

			Assume that $t_1 \in (0,1)$ so that $\gamma(t_1)$ is a point that is conjugate to $\gamma(0)$. Then let $\parenm{e_i(t)}_{i=2}^n$ be a parallel orthonormal frame of $\dot \gamma(t)^\perp \subset T_{\gamma(t)}M_K$, and let $K(t) = \sum_{i=2}^n K_i(t) e_i(t)$ be a Jacobi field so that $K(t_1) = K(0) = 0$. Defining $e_1(t) \defeq \dot \gamma(t)$ we thus have that $\parenm{e_i(t)}_{i=1}^n$ is a parallel orthonormal frame of $T_{\gamma(t)}M_K$ along $\gamma$. Using the notation $\dot K \defeq \nabla_{\frac{d}{dt}} K$, we define $L \defeq \dot K$. Then we have the system
			\[
				\begin{cases}
					\dot K = L \\
					\dot L = -R(\dot \gamma, K)\dot \gamma \, .
				\end{cases}
			\]
			By using $K = \sum_{i=1}^n K_i e_i$ and $L = \sum_{i=1}^n L_i e_i$ with $K_1 = L_1 = 0$, we get the system of differential equations
			\[
				\begin{cases}
					\dot K_i(t) = L_i(t) \\
					\dot L_i(t) = - \sum_{j=1}^n R^i_j(t) K_j(t)
				\end{cases} \text{ for all } i\in \parenm{1,\ldots,n}
			\]
			where $R(t) = \parenm{R^i_j(t)}_{i,j=1}^n = \parenm{\langle R(\dot \gamma, e_j)\dot \gamma, e_i\rangle}_{i,j=1}^n$ is a symmetric matrix. This is equivalent to say that
			\[
				\frac{d}{dt} \begin{pmatrix}
					K \\ L
				\end{pmatrix} = \begin{pmatrix}
					0 & I_{n} \\
					-R & 0
				\end{pmatrix} \begin{pmatrix}
					K \\ L
				\end{pmatrix}\, ,
			\]
			An explicit fundamental solution to this system is
			\[
				\varPhi(t) = \exp \paren{t \begin{pmatrix}
					0 & I_{n} \\ -R & 0
				\end{pmatrix}} = \begin{pmatrix}
					C(-t^2 R) & t S(-t^2 R) \\
					(-tR) S(-t^2 R) & C(-t^2 R)
				\end{pmatrix}\, ,
			\]
			where 
			\[
				C(A) = \sum_{k=0}^\infty \frac{A^k}{(2k)!}, \qquad S(A) = \sum_{k=0}^\infty \frac{A^k}{(2k+1)!}\, ,
			\]
			for $A \in \operatorname{End}(\R^n)$. The fundamental solution $\varPhi$ satisfies $\varPhi(0) = I$ and in particular
			\[
				\begin{pmatrix}
					K(t) \\ L(t) 
				\end{pmatrix} = \varPhi(t) \begin{pmatrix}
					K(0) \\ L(0)
				\end{pmatrix}\, .
			\]
			We use that the Jacobi field $K$ vanishes at $t = 0$. Then we plug in $t = t_1 \in (0,1)$ for which we also have $K(t_1) = 0$. Thus
			\[
				\begin{pmatrix}
					0 \\ L(t_1) 
				\end{pmatrix} = \varPhi(t) \begin{pmatrix}
					0 \\ L(0)
				\end{pmatrix} = \begin{pmatrix}
					t_1 S(-t_1^2 R)L(0) \\
					C(-t_1^2 R)L(0)
				\end{pmatrix}\, .
			\]
			From this, we have that $\gamma(t_1)$ is conjugate to $\gamma(0)$ if and only if $S(-t^2 R)$ is singular at $t = t_1$.

			We consider $(K,L) \in T(T_{\gamma(t)}M_K)$, and using the metric isomorphism and scaling $K$ properly, we consider $(K,L) \in T(ST^\ast_{\gamma(t)}M_K)$ for $\gamma(t) \in M_K$. Since we assumed that $K$ (and hence $L$) was orthogonal to $\gamma$, we regard the lift as $(K,L) \in \xi \subset T(ST^\ast M_K)$ along the lifted geodesic. Since $\xi_z \cong \CC^{n-1} \cong \R^{n-1} \oplus i \R^{n-1}$ is symplectic with the standard symplectic form in these coordinates, we have that $i \R^{n-1} \subset \xi$ is Lagrangian. We then consider the path of Lagrangians
			\[
				\ell(t) = \varPhi(t)(i \R^{n-1}) =\varPhi(t) \begin{pmatrix}
					0 \\ \zeta
				\end{pmatrix} = \begin{pmatrix}
					t S(-t^2 R)\zeta \\
					C(-t^2 R) \zeta
				\end{pmatrix}, \quad \forall \zeta \in i \R^{n-1}\, .
			\]
			Whenever $\gamma(t)$ is conjugate to $\gamma(0)$, the matrix $S(-t^2 R)$ is singular. Hence it has non-trivial kernel which contributes to the Maslov index exactly the dimension of the kernel. The dimension of the kernel also correspond to the multiplicity of $\gamma(t)$ as a conjugate point to $\gamma(0)$. By closing up the loop positively, we find an extra contribution of $n-1$. Hence
			\[
				\mu(\ell) = (n-1) + \sum_{t \; : \; S(-t^2R) \text{ singular}} \dim \ker S(-t^2R) = (n-1) + \ind(\gamma)\, ,
			\]
			from which we conclude $\ind(\gamma) = -\abs a$.
		\end{proof}
		
		\begin{prp}\label{prp:sols_up_to_first_order_respects_filtration}
			Let $a\in CW^\ast_{\varLambda_K}(F,F)$ be any generator. Let $u_0$ be a $J$-holomorphic half strip as in \cref{sub:moduli_space_of_half_strips}, and let $v \in \ker D_{u_0}$, where $D_{u_0}$ is the linearized Cauchy--Riemann operator at $u_0$. Then consider the linearized solution
			\[
				u_\varepsilon \defeq \exp_{u_0}(\varepsilon v)\, ,
			\]
			for any $\varepsilon > 0$. Then we have
			\[
				\mathfrak a(a) > L(\gamma_\varepsilon)\, ,
			\]
			where $\gamma_\varepsilon \defeq \ev(u_\varepsilon)$.
		\end{prp}
		\begin{proof}
			We modify the proof of \cref{prp:action_length_estimate}. Since $u_\varepsilon$ is not $J$-holomorphic, we first prove that the estimate $0 < \int_{u_\varepsilon^{-1}(W_K)} u_\varepsilon^\ast d\lambda_1$ holds.

			In a neighborhood of $(q,p) \in \im u_\varepsilon \subset W_K$, we have a splitting
			\[
				T_{(q,p)}W_K = \mathrm{span}\parenm{p\dd_p, p\dd_q} \oplus \xi \cong \R^2 \times \R^{2n-2}\, .
			\]
			We pick a small ball $B(\rho)$ of radius $\rho > 0$ around $(0,0) \in \R^2 \times \R^{2n-2}$. Let $J_0$ be the product almost complex structure on $\R^2 \times \R^{2n-2}$, which extends $J$ over $B(\rho)$, and let $\abs -$ be the norm determined by $d \lambda_1$ and $J_0$. Then we pick some coordinate $z = (s,t) \in T$ in the domain of $u_\varepsilon$, and define
			\[
				q(z) \defeq (J(u_\varepsilon(z))+J_0(u_\varepsilon(z)))^{-1}(J(u_\varepsilon(z))-J_0(u_\varepsilon(z)))\, .
			\]
			In the operator norm we have $\norm{q(z)} = O(\rho)$ as $\rho \to 0$. Then we have
			\begin{equation}\label{eq:def_of_A}
				2(J_0 + J)^{-1}J \overline{\dd}_Ju_\varepsilon = \overline{\dd}_{J_0}u_\varepsilon + q(z) \dd_{J_0}u_\varepsilon \eqdef A(z)\, .
			\end{equation}
			Let $\rho$ be small enough so that $\abs{\dd_{J_{0}} u_{\varepsilon}}^{2} - 3\abs{q(z)\dd_{J_{0}} u_{\varepsilon}}^{2} > 0$. Then we have
			\begin{align*}
				\abs{\dd_{J_0}u_\varepsilon}^2 + \abs{\overline{\dd}_{J_0}u_\varepsilon}^2 & < \abs{\dd_{J_0} u_\varepsilon}^2 + \abs{\overline{\dd}_{J_0}u_\varepsilon}^2+ \paren{\abs{\dd_{J_{0}} u_{\varepsilon}}^{2} - 3\abs{q(z)\dd_{J_{0}} u_{\varepsilon}}^{2}} + \abs{A(z)}^{2} \\
				&= 2 \abs{\dd_{J_0}u_\varepsilon}^2 + \abs{\overline{\dd}_{J_0}u_\varepsilon}^2 - 3 \paren{\abs{A(z)}^{2} + \abs{q(z)\dd_{J_{0}} u_{\varepsilon}}^{2}} + 4\abs{A(z)}^{2} \\
				&\leq 2 \paren{\abs{\dd_{J_0}u_\varepsilon}^2 - \abs{\overline{\dd}_{J_0}u_\varepsilon}^2} + 4 \abs{A(z)}^{2}\, .
			\end{align*}
			Next we note that
			\[
				u_\varepsilon^\ast d \lambda_1 = \paren{\abs{du_\varepsilon}^2 - 2 \abs{\overline{\dd}_{J_0}u_\varepsilon}^2} \mathrm{dvol}_T = \paren{\abs{\dd_{J_0}u_\varepsilon}^2 - \abs{\overline{\dd}_{J_0}u_\varepsilon}^2} \mathrm{dvol}_T\, ,
			\]
			and hence
			\begin{align*}
				\abs{du_\varepsilon}^2 \mathrm{dvol}_T &= \paren{\abs{\dd_{J_0}u_\varepsilon}^2 + \abs{\overline{\dd}_{J_0}u_\varepsilon}^2}\mathrm{dvol}_T < 2 \paren{\abs{\dd_{J_0} u_\varepsilon}^2 - \abs{\overline{\dd}_{J_0} u_\varepsilon}^2}\mathrm{dvol}_T + 4 \abs{A(z)}^2\mathrm{dvol}_T \\
				&= 2 u_\varepsilon^\ast d \lambda_1 + 4 \abs{A(z)}^2\mathrm{dvol}_T\, .
			\end{align*}
			In view of the definition of $A(z)$ in \eqref{eq:def_of_A}, we have that $\abs{A(z)}^2 = O(\varepsilon^4)$.

			Let $\pi_\xi \co TW_K \longrightarrow TW_K$ be the projection onto the contact plane $\xi$. By \cref{lma:d_beta_tau_is_non-negative_outside_zero_sect} the only contribution to $\int_{u_\varepsilon^{-1}(B(\rho))} u_\varepsilon^\ast d \lambda_1$ comes from the restriction to $\xi$. Summing over all balls $B(\rho)$ covering the image of $u_0$ gives
			\begin{equation}\label{eq:estimate_exp_pseudohol_to_first_order}
				2\int_{u_\varepsilon^{-1}(W_K)} u_\varepsilon^\ast d \lambda_1 = 2\int_{u_\varepsilon^{-1}(W_K)} \pi_\xi(u_\varepsilon^\ast d \lambda_1) \geq \norm{\pi_\xi(du_\varepsilon)}^2 - 4 \norm{\pi_\xi(A(z))}^2 \, .
			\end{equation}
			The Taylor expansion of $u_\varepsilon$ around $\varepsilon = 0$ is
			\[
				u_\varepsilon = u_0 + \varepsilon v + O(\varepsilon^2)\, ,
			\]
			where $v \in \ker D_{u_0}$. Because $v$ is a non-zero solution of the linearized equation $D_{u_0}v = 0$, we rescale $v$ in such a way that
			\begin{equation}\label{eq:rescaled_v_in_W_22}
				\norm v^2_{W^{2,2}_{\kappa}} = \norm v^2_{L^2_\kappa} + \norm{dv}^2_{W^{1,2}_\kappa} = 1
			\end{equation}
			where $W^{k,p}_\kappa$ is the weighted Sobolev space $W^{k,p}([0,\infty) \times [0,1])$ with weight $e^{\kappa s}$ for some small $\kappa > 0$ and where $s$ is the coordinate in the $[0,\infty)$-factor. Let $Z_T = [0,T] \times [0,1] \subset [0,\infty) \times [0,1]$ for some $T > 0$. We use the Poincar\'e inequality $\norm v^2_{L^2_\kappa(Z_T)} \leq C_1 \norm{dv}^2_{L^2_\kappa(Z_T)}$, where $C_1 > 0$ (given that $\kappa > 0$ is small enough), together with \eqref{eq:rescaled_v_in_W_22}. This gives that $\norm{dv}^2_{W^{1,2}_\kappa(Z_T)} \geq C_0$ for some $C_0 > 0$ and some $T > 0$. Hence $\norm{dv}^2_{W^{1,2}_\kappa} \geq C$ for some $C > 0$. 

			The same argument applied to $\pi_\xi(v)$ and $\pi_\xi(dv)$ gives $\norm{\pi_\xi(dv)}_{W^{1,2}_{\kappa}}^2 \geq C'$ for some $C' > 0$. Hence
			\[
			 	\norm{\pi_\xi(du_\varepsilon)}^2 - 4 \norm{\pi_\xi(A(z))}^2 = \norm{\varepsilon \pi_\xi(dv)}^2 + O(\varepsilon^3) \geq C' \varepsilon^2 + O(\varepsilon^3) > 0\, ,
			 \] 
			 for small enough $\varepsilon > 0$. By \eqref{eq:estimate_exp_pseudohol_to_first_order} we therefore have $\int_{u_\varepsilon^{-1}(W_K)} u_\varepsilon^{\ast} d \lambda_1 > 0$.
			
			Next, we show $\int_{u_\varepsilon^{-1}(W_K)} u_\varepsilon^{\ast} d \lambda_1 = \mathfrak a(a) - L(\gamma_\varepsilon)$. The proof is similar to the computation in the proof of \cref{prp:action_length_estimate}. The only difference is the computation of $I_2$ (with notation as in \cref{prp:action_length_estimate}). Since $\overline{\dd}_J u_\varepsilon = O(\varepsilon^2)$, equation \eqref{eq:components_of_u_tilde} becomes
			\begin{equation}\label{eq:components_of_u_tilde_O_epsilon_squared}
				\begin{cases}
					\dd_s Q + \dd_t P + O(\abs x) = O(\varepsilon^2) \\
					\dd_s P - \dd_t P + O(\abs x) = O(\varepsilon^2)\, ,
				\end{cases}
			\end{equation}
			where $x = \tilde u(s,t)$. Then from the second equation we get
			\[
				P(s,t) = s \paren{\dot q(t) + O(\abs x) + O(\varepsilon^2)} + w(s,t)\, .
			\]
			Setting $s = 0$ in \eqref{eq:components_of_u_tilde_O_epsilon_squared} gives
			\[
				\frac{\dd v}{\dd s}(0,t) = O(\abs{\tilde u(0,t)})) + O(\varepsilon^2) = \frac{\dd w}{\dd s}(0,t)\, ,
			\]
			and hence
			\[
				\begin{cases}
					\frac{\dd v}{\dd t}(\delta_0,t) = O(\delta_0) \\
					w(\delta_0, t) = \delta_0 \paren{O(\abs{\tilde u(0,t)}) + O(\varepsilon^2)} + O(\delta_0^2)\, .
				\end{cases}
			\]
			By repeating the same calculation as in \eqref{eq:computations_of_pullback_of_beta_1_close_to_geodesic} we end up at
			\[
				\eval[0]{\tilde u^\ast \beta_1}_{s = \delta_0} =\frac{\ip{ \dot q(t) + O(\abs{\tilde u(0,t)}) + O(\varepsilon^2) + O(\delta_0)}{\dot q(t) + O(\delta_0)}}{\abs{\dot q(t) + O(\abs{\tilde u(0,t)}) + O(\varepsilon^2) + O(\delta_0)}}dt = \paren{\abs{\dot q(t)} + O(\delta_0)}dt \, .
			\]
			The rest of the proof of \cref{prp:action_length_estimate} (which does not require holomorphicity) gives us the result.
		\end{proof}
		\begin{prp}\label{prp:chain_map_is_diagonal_wrt_filtrations}
			Let $a\in CW^\ast_{\varLambda_K}(F,F)$ be any generator and consider $\varPsi_1(a) \in C_{-\ast}(\varOmega_\xi M_K)$. Then 
			\[
				\mathfrak a(a) = \mathfrak a(\varPsi_1 a)\, .
			\]
			The same is also true for the chain map
			\[
				r_\ast \circ \varPsi_1\co CW^\ast_{\varLambda_K}(F,F) \longrightarrow C_{-\ast}(BM_K)\, .
			\]
		\end{prp}
		\begin{proof}
			By \cref{cor:chain_map_preserves_filtrations} we have that 
			\[
				0 \leq \mathfrak a(a) - \mathfrak a(\varPsi_1 a)\, ,
			\]
			and to prove equality, it is enough to show that for any $a\in CW^\ast_{\varLambda_K}(F,F)$, there exists a transversely cut out $J$-holomorphic disk $u\in \overline{\mathcal M}(a)$ with
			\[
				\int_{u^{-1}(W_K)} u^\ast d \lambda_\tau = 0\, .
			\]
			We let
			\[
				u_0 \co T \longrightarrow W_K\, ,
			\]
			be the $J$-holomorphic half strip that is the cone over the Reeb chord $a\in CW^\ast_{\varLambda_K}(F,F)$. In geodesic normal coordinates at $(q,p) = u_0(s,t)$, we have that the tangent space of $\im u_0$ at $(q,p)$ is $T_{u_0(s,t)} \im u_0 = \mathrm{span}\parenm{p\dd_p, p\dd_q}$ which means that $u_0^\ast d \lambda_\tau = 0$ by \cref{lma:d_beta_tau_is_non-negative_outside_zero_sect} and hence by \cref{prp:action_length_estimate} we have 
			\begin{equation}\label{eq:diagonal_wrt_filtration}
				\mathfrak a(a) = L(\gamma)\, .
			\end{equation}
			What is left to show is that $u_0$ is transversely cut out. Consider the following space of vector fields along $u_0$
			\[
				V = \parenm{\eta \in \ker D_{u_0} \suchthat I(\pi_\ast \eta, \pi_\ast \eta) < 0}\, .
			\]
			where $\pi$ is the projection $\pi : W_{K} \longrightarrow M_{K}$ along the Liouville flow, and $I$ is the index form (see \eqref{eq:index_form_hessian_of_energy} below for a definition). By \cref{lma:transversality_of_moduli_space_of_half_strips} we have
			\[
				\ind D_{u_0} = \ind \gamma\, .
			\]
			That is, $\ind D_{u_0}$ is equal to the dimension of the maximal subspace of the space of sections of $\gamma^\ast TM_K$ on which $I$ is negative definite. The projection 
			\[
				\eval[0]{\pi_\ast}_{\ker D_{u_0}} \co \ker D_{u_0} \longrightarrow \gamma^\ast TM_K
			\]
			is injective by unique continuation (cf.\@ \cite[Corollary 2.27]{wendl2016lectures}), which implies that we have
			\begin{equation}\label{eq:dim_jacobi_equals_index_of_lin_op}
				\dim V \leq \ind D_{u_0}\, .
			\end{equation}
			For $v\in \ker D_{u_0}$ we have that $u_\varepsilon = \exp_{u_0}(\varepsilon v)$ is a disk that is near to $u_0$ for small $\varepsilon > 0$. In particular, it is a solution of the Floer equation \eqref{eq:floer_eqn_half_strips} up to the first order. By \cref{prp:sols_up_to_first_order_respects_filtration} and \eqref{eq:diagonal_wrt_filtration} we get
			\[
				0 > L(\gamma_\varepsilon) - \mathfrak a(a) = L(\gamma_\varepsilon) - L(\gamma)\, ,
			\]
			which in turn implies
			\begin{equation}\label{eq:difference_of_energies_negative}
				0 > E(\gamma_\varepsilon) - E(\gamma)\, ,
			\end{equation}
			where $E(\gamma) = \int_0^1 \abs{\dot \gamma(t)}^2 dt$ is the energy of the curve $\gamma$.

			Now, by defining $E(s) \defeq E(\ev(\exp_{u_0}(s v)))$ we compute
			\begin{align}\label{eq:index_form_hessian_of_energy}
				\frac{d^2}{ds^2} \eval[3]{E(s)}_{s=0} &= \frac{d^2}{ds^2}\eval[3]{E(\ev(\exp_{u_0}(s v)))d\ev_{\exp_{u_0}(s v)}(d \exp_{u_0}(s v))}_{s=0} \\
				&= \frac{d^2}{ds^2}E(\ev(u_0)) d\ev_{u_0}(v) = I(\pi_\ast v, \pi_\ast v) \nonumber\, ,
			\end{align} 
			where $I$ is the index form, see e.g.\@ \cite[Section 4.1]{jost2008riemannian}. The Taylor expansion of $E(\varepsilon)$ around $\varepsilon = 0$ is
			\[
				E(\varepsilon) - E(0) = \varepsilon^2 I(\pi_\ast v, \pi_\ast v) + O(\varepsilon^3) \overset{\eqref{eq:difference_of_energies_negative}}{<} 0\, .
			\]
			Hence for small enough $\varepsilon > 0$, we obtain $I(\pi_\ast v, \pi_\ast v) < 0$ and consequently $v \in V$. Therefore we have
			\[
				\dim \ker D_{u_0} \leq \dim V \overset{\eqref{eq:dim_jacobi_equals_index_of_lin_op}}{\leq} \ind D_{u_0}\, ,
			\]
			which concludes $\dim \coker D_{u_{0}} = 0$ and therefore $u_0$ is transversely cut out.

			The same proof shows that $r_\ast \circ \varPsi_1$ is diagonal with respect to the action and length filtrations of Reeb chords and chains of broken geodesic loops, respectively, because of \cite[Lemma 15.4]{milnor1963morse}. Namely, the index of the Hessian $E_{\ast \ast}$ is equal to the index of $E_{\ast \ast}$ restricted to the tangent space $T_\gamma BM_K$.
		\end{proof}
		\subsection{Isomorphism between $CW^\ast_{\varLambda_K}(F,F)$ and $C^{\mathrm{cell}}_{-\ast}(BM_K)$}
			\label{sub:the_filtration_argument}
			The goal of this section is to show that there is a chain isomorphism between $CW^\ast_{\varLambda_K}(F,F)$ and $C^{\mathrm{cell}}_{-\ast}(BM_K)$. The outline of the proof is the following. Given a generator $a \in CW^\ast_{\varLambda_K}(F,F)$ we consider the trivial $J$-holomorphic half strip $u_0 \in \mathcal M(a)$ as in \cref{sub:the_chain_map_is_diagonal_with_respect_to_the_action_filtrations}. By the genericity of the metric as in \eqref{eq:metric_definition} we show that the evaluation map defined in \cref{sub:the_evaluation_map_etc} is transverse to the infinite dimensional stable manifold of the geodesic $\gamma$ in $BM_K$. This gives a chain isomorphism between $CW^\ast_{\varLambda_K}(F,F)$ and $C^{\mathrm{cell}}_{-\ast}(BM_K)$ by identifying a neighborhood of $u_0 \in \mathcal M(a)$ with the unstable manifold of the geodesic $\gamma \in BM_K$ which correspnds to the generator $a$.

			We use the notation
			\[
				\mathcal F_{[x_1,x_2)} C \defeq \mathcal F_{x_1} C / \mathcal F_{x_2}C\, ,
			\]
			and order the generators of $CW^{\ast}_{\varLambda_K}(F,F)$ by their action
			\begin{align*}
				0 &\leq \mathfrak a(a_1) < \mathfrak a(a_2) < \cdots
			\end{align*}
			Pick a strictly increasing sequence of numbers $\parenm{a_i}_{i=1}^\infty$ so that 
			\begin{align*}
				0 \leq \mathfrak a(a_1) < A_1 < \mathfrak a(a_2) < A_2 < \cdots\, ,
			\end{align*}
			and define
			\begin{align*}
				\mathcal F_{A_i} CW^\ast_{\varLambda_K}(F,F) &\defeq \parenm{c\in CW^\ast_{\varLambda_K}(F,F) \suchthat \mathfrak a(c) < A_i} \\
				\mathcal F_{A_i} C_{-\ast}(BM_K) &\defeq \parenm{\sigma \in C_{-\ast}(BM_K) \suchthat \mathfrak a(\sigma) < A_i}\\
				\mathcal F_{A_i} C_{-\ast}^{\text{cell}}(BM_K) &\defeq \parenm{\sigma \in C_{-\ast}^{\text{cell}}(BM_K) \suchthat \mathfrak a(\sigma) < A_i}\, .
			\end{align*}
			We extend the filtration to all of $\Z$ by letting $A_{i} = 0$ for every $i \leq 0$.

			Note that the ordering of the generators of $CW^\ast_{\varLambda_K}(F,F)$ gives an ordering of the generators of $C_{-\ast}(BM_K)$ and $C_{-\ast}^{\text{cell}}(BM_K)$ by \cref{prp:chain_map_is_diagonal_wrt_filtrations}.

			Recall the definition of the retraction
			\[
				r\co \mathcal F_{A_i} \varOmega^{\text{pw}}M_K \longrightarrow \mathcal F_{A_i} BM_K\, ,
			\]
			defined in the proof of \cref{lma:pw_loops_deformation_retract}: Let $\gamma\in \mathcal F_{A_i} \varOmega^{\text{pw}} M_K$ be any loop with $L(\gamma) = \int_0^1 \abs{\dot \gamma} dt < a_i$. Pick a subdivision of the domain of $[0,1]$
			\[
				0 = t_0 < t_1 < \cdots < t_{N-1} < 1 = t_N\, ,
			\]
			which is fine enough so that $\rho(\gamma(t_{i-1}), \gamma(t_i)) < \varepsilon$ for some $\varepsilon > 0$ small enough. Then $r(\gamma)$ is defined so that
			\[
				\eval[0]{r(\gamma)}_{[t_{i-1},t_i]} = \text{ unique minimal geodesic of length} < \varepsilon \text{ from } \gamma(t_{i-1}) \text{ to } \gamma(t_i)\, .
			\]
			Then we define
			\begin{align*}
				r_\ast \co \mathcal F_{A_i} C_{-\ast}^{\mathrm{cell}}(\varOmega^{\text{pw}}M_K) &\longrightarrow \mathcal F_{A_i} C_{-\ast}^{\mathrm{cell}}(BM_K)\\
				\sigma &\longmapsto r \circ \sigma\, .
			\end{align*}
			\begin{thm}\label{thm:isomorphism_cw_and_cell_bmk}
				The map
				\begin{align*}
					r_\ast \circ \varPsi_1 \co CW^{\ast}_{\varLambda_K}(F,F) &\longrightarrow C_{-\ast}^{\mathrm{cell}}(BM_K) \\
					a &\longmapsto r_\ast \circ \ev_{\ast}[\overline{\mathcal M}(a)]\, ,
				\end{align*}
				is an isomorphism
			\end{thm}
			\begin{proof}
				We first show that for any $i\in \Z$ the map
				\begin{align*}
					r_\ast \circ \varPsi_1 \co \mathcal F_{[A_{i-1},A_i)}CW^{\ast}_{\varLambda_K}(F,F) &\longrightarrow \mathcal F_{[A_{i-1},A_i)}C_{-\ast}^{\mathrm{cell}}(BM_K) \\
					a &\longmapsto r_\ast \circ \ev_{\ast}[\overline{\mathcal M}(a)]\, ,
				\end{align*}
				is an isomorphism.

				By the definition of the numbers $\parenm{A_i}_{i=0}^\infty$ there is only one generator $a\in \mathcal F_{[A_{i-1},A_i)} CW^\ast_{\varLambda_K}(F,F)$. Denote its degree by $\lambda$. By \cref{lma:1-1-correspondence_geodeiscs_reeb_of_index_k}, there is exactly one generator $\sigma\in \mathcal F_{[A_{i-1},A_i)} C^{\mathrm{cell}}_{-\lambda}(BM_K)$ that corresponds to $a$. We think of $\sigma$ as the unstable manifold of the geodesic $\gamma$ corresponding to $a$. Since both $\mathcal F_{[A_{i-1},A_i)} CW^\ast_{\varLambda_K}(F,F)$ and $\mathcal F_{[A_{i-1},A_i)} C^{\mathrm{cell}}_{-\ast}(BM_K)$ only contains one generator each, we only need to show that
				\begin{equation}\label{eq:wts_surjection}
					\sigma = (r_\ast \circ \varPsi_1)(a)\, .
				\end{equation}
				By \cref{prp:chain_map_is_diagonal_wrt_filtrations} we already know that the trivial $J$-holomorphic half strip $u_0 \in \overline{\mathcal M}(a)$ over $a$ is so that $\ev(u_0) = \gamma$. To prove that equation \eqref{eq:wts_surjection} holds it is enough to consider the map
				\[
					r\circ \ev \co \overline{\mathcal M}(a) \longrightarrow BM_K\, ,
				\]
				and show that it is locally surjective at $\gamma\in \im \sigma \subset BM_K$. We do this by showing that it is a submersion. That is, we consider
				\[
					d(r \circ \ev)_{u_0} \co T_{u_0} \overline{\mathcal M}(a) \longrightarrow T_\gamma BM_K\, ,
				\]
				and we show that it is surjective onto the image of $\sigma$. As noted above, $\sigma$ should be thought of as the unstable manifold of $\gamma$ inside $BM_K$ with respect to the energy functional $E$. The following composition
				\[
					\begin{tikzcd}[row sep=scriptsize, column sep=scriptsize]
						T_{u_0} \overline{\mathcal M}(a) \rar{d\ev_{u_0}} & T_\gamma \varOmega^{\mathrm{pw}}M_K \rar{dr_{\gamma}} & T_\gamma BM_K
					\end{tikzcd}
				\]
				is described as follows. Pick a subdivision of the domain of $\gamma$
				\[
					0\leq t_0 < t_1 < \cdots < t_N \leq 1\, ,
				\]
				for some $N\in \Z_+$. The tangent space of $\varOmega^{\mathrm{pw}}M_K$ at $\gamma$ has the following splitting
				\[
					T_\gamma \varOmega^{\mathrm{pw}}M_K = T_\gamma BM_K \oplus T'\, ,
				\]
				by \cite[Lemma 15.3, 15.4]{milnor1963morse}. Here $T_\gamma BM_K$ is the space of \emph{broken Jacobi fields} vanishing at the endpoints, and $T'$ is the space of all vector fields $W$ along $\gamma$ so that $W(t_k) = 0$ for every $k\in \parenm{1,\ldots,N}$ (cf.\@ \cite[Section 15]{milnor1963morse}). Furthermore we write
				\[
					T_\gamma BM_K = T_\gamma \sigma \oplus T^+\, ,
				\]
				where $T_\gamma \sigma$ is the (maximal) subspace of $T_\gamma BM_K$ on which the Hessian $E_{\ast \ast}$ is negative definite, and $T^+ \subset T_\gamma BM_K$ is the subspace on which $E_{\ast \ast}$ is positive semidefinite. We will show that for any non-zero $v\in T_{u_0} \overline{\mathcal M}(a)$, its image $d\ev_{u_0}(v)$ does not lie in $T'$, and that the image $d(r\circ \ev)_{u_0}(v)$ does not lie in $T^+$.
				\begin{description}
					\item[$d\ev_{u_0}$ is transverse to $T'$] Consider any non-zero $v \in T_{u_0} \overline{\mathcal M}(a) = \ker D_{u_0}$, where
					\[
						D_{u_0} \co W^{2,2}_\kappa(D_3, u_0^\ast TW_K) \longrightarrow W^{1,2}_\kappa(D_3, \varLambda^{0,1} \otimes_J u_0^\ast TW_K)
					\]
					is the linearization of $\overline{\dd}_J$ at $u_0$. Here $W^{2,2}_{\kappa}$ is the Sobolev space $W^{2,2}$ with weight $e^{\kappa s}$ for some small $\kappa > 0$ at the positive punctures in the domain $D_3$. The differential of the evaluation map $\ev$ is a trace operator on $W^{2,2}_\kappa(D_3, {u_0}^\ast TW_K)$, so $d\ev_{u_0}(v)$ is a vector field in $W_K$ along $\gamma \subset M_K \subset W_K$. Assume that $v$ is such that $v' \defeq d\ev_{u_0}(v) \in T'$, that is $v'(\gamma(t_k)) = 0$ for every $k\in \parenm{1,\ldots,N}$. Since $d \ev_{u_0}$ is a restriction to $\gamma^\ast TM_K$, we assume that $v$ is so that $v(\gamma(t_k)) = 0$ for every $k\in \parenm{1,\ldots,N}$. We consider the subspace
					\[
						A \defeq \parenm{v \in W^{2,2}_\kappa(D_3, {u_0}^\ast TW_K) \suchthat v(\gamma(t_k)) = 0, \, k\in \parenm{1,\ldots,N}}\, .
					\]
					It is closed and has codimension $N$. The restricted linearized operator $\eval[0]{D_{u_0}}_{A}$ is therefore a Fredholm operator with index
					\[
						\ind \eval[0]{D_{u_0}}_A = \ind D_{u_0} - N\, .
					\]
					If we pick $N$ large enough by making the subdivision of the domain of loops fine enough, the index $\ind \eval[0]{D_{u_0}}_A$ is negative. Hence $\ker D_{u_0} \cap A$ is empty for generic choices of almost complex structures on $W_K$. This means that $\im d\ev_{u_0} \cap T' = \parenm{0}$.
					\item[$d(r\circ \ev)_{u_0}$ is transverse to $T^+$] Next, we show that for any non-zero $v \in \ker D_{u_0}$, its projection
					\[
					 	v'' \defeq d(r \circ \ev)_{u_0}(v) \in T_\gamma BM_K
					\]
					 does not lie in $T^+$. We consider the path $s \longmapsto \exp_{u_0}(s v)$ for $s\in (0,\varepsilon)$ with $\varepsilon$ small enough. Then by the proof of \cref{prp:chain_map_is_diagonal_wrt_filtrations} we have for every $s\in (0,\varepsilon)$ that
					\[
						E((r\circ \ev)(u_0)) > E((r\circ \ev)(\exp_{u_0}(s v)))\, .
					\]
					Repeating the argument in the proof of \cref{prp:chain_map_is_diagonal_wrt_filtrations} gives $I(v'', v'') < 0$, which shows that $v'' = d(r\circ \ev)_{u_0}(\xi)$ does not lie in $T^+$.
				\end{description}
				Therefore $r_\ast \circ \varPsi_1 \co \mathcal F_{[A_{i-1},A_i)}CW^{\ast}_{\varLambda_K}(F,F) \longrightarrow \mathcal F_{[A_{i-1},A_i)}C_{-\ast}^{\mathrm{cell}}(BM_K)$ is an isomorphism.

				The filtrations on $CW^\ast_{\varLambda_K}(F,F)$ and $C_{-\ast}^{\text{cell}}(BM_K)$ are both bounded from below which gives an isomorphism $\mathcal F_{A_i} CW^\ast_{\varLambda_K}(F,F) \cong \mathcal F_{A_i} C_{-\ast}^{\text{cell}}(BM_K)$ for every $i\in \Z$. Thus every square in the following diagram commutes.
				\[
					\begin{tikzcd}[row sep=scriptsize, column sep=tiny]
						\cdots \rar[description, phantom]{\subset} & \mathcal F_{A_i} CW^\ast_{\varLambda_K}(F,F) \dar{r_\ast \circ \varPsi_1}[swap]{\cong} \rar[description, phantom]{\subset} & \mathcal F_{a_{i+1}} CW^\ast_{\varLambda_K}(F,F) \dar{r_\ast \circ \varPsi_1}[swap]{\cong} \rar[description, phantom]{\subset} & \cdots \\
						\cdots \rar[description, phantom]{\subset} & \mathcal F_{A_i} C_{-\ast}^{\text{cell}}(BM_K) \rar[description, phantom]{\subset} & \mathcal F_{a_{i+1}} C_{-\ast}^{\text{cell}}(BM_K) \rar[description, phantom]{\subset} & \cdots
					\end{tikzcd}
				\]
				We then pass to colimits to obtain the isomorphism
				\[
					CW^\ast_{\varLambda_K}(F,F) \cong C_{-\ast}^{\text{cell}}(BM_K)\, .
				\]
			\end{proof}
			\begin{proof}[Proof of \cref{thm:quiso}.]
				\cref{thm:quiso} is now an immediate corollary of \cref{thm:isomorphism_cw_and_cell_bmk}, because there is a chain homotopy equivalence $C_{-\ast}^{\text{cell}}(BM_K) \simeq C_{-\ast}(\varOmega_\xi M_K)$. So in particular we have $H_{-\ast}(\varOmega_\xi M_K) \cong H^{\mathrm{cell}}_{-\ast}(BM_K)$ via $s\circ r_\ast$ defined in \eqref{eq:iso_cell_sing} and \eqref{eq:r_deformation_retract_quiso}. Hence
				\[
					\varPsi_1 \co HW^\ast_{\varLambda_K}(F,F) \longrightarrow H_{-\ast}(\varOmega_\xi M_K)\, ,
				\]
				is an isomorphism.
			\end{proof}
\section{Applications}
\label{sec:applications}
	The first goal of this section is to equip $HW^\ast_{\varLambda_K}(F,F)$ and $H_{-\ast}(\varOmega_\xi M_K)$ with the structure of $\Z[\pi_1(M_K)]$-modules. The second goal is then to consider the case when $S = S^n$ and exhibit examples of codimension 2 knots $K \subset S^n$ where the Alexander invariant is related to $CW^\ast_{\varLambda_K}(F,F)$ as $\Z[\pi_1(M_K)]$-modules. From this we draw the conclusion that the unit conormal of $K$ knows about the smooth topology of $K$ beyond the fundamental group.

	After we have discussed the $\Z[\pi_1(M_K)]$-module structures in \cref{sub:Zpi_module_structures}, we will provide background material surrounding the Alexander invariant in \cref{sub:plumbings_and_infinite_cyclic_covers,sub:the_alexander_invariant}. Then, in \cref{sub:using_the_leray--serre_spectral_sequence} we use the Leray--Serre spectral sequence associated with the path-loop fibration to relate the Alexander invariant to $CW^\ast_{\varLambda_K}(F,F)$ as a $\Z[\pi_1(M_K)]$-module.
	\subsection{$\Z[\pi_1(M_K)]$-module structures on $HW^{\ast}_{\varLambda_K}(F,F)$ and $H_{-\ast}(\varOmega M_K)$}
		\label{sub:Zpi_module_structures}
		Consider any homotopy class $[\gamma] \in \pi_1(M_K)$ represented by the unique minimizing geodesic $\gamma$ in the given homotopy class. Via the cell structure of $BM_K$, we associate to $\gamma$ a generator $\sigma_\gamma \in H_{-\ast}(\varOmega M_K)$. Then consider the map
		\begin{align*}
			\pi_1(M_K) \times H_{-\ast}(\varOmega M_K) &\longrightarrow H_{-\ast}(\varOmega M_K) \\
			(\gamma, \sigma) &\longmapsto \gamma \sigma \defeq (-1)^{\sigma_\gamma}P(\sigma \otimes \sigma_\gamma)\, ,
		\end{align*}
		where $P$ denotes the Pontryagin product as in \eqref{eq:def_pontryagin_product}.
		\begin{lma}\label{lma:grp_action_on_cohomology_of_loops}
			The map $(\gamma, \sigma) \longmapsto \gamma \sigma$ defines a group action of $\pi_1(M_K)$ on $H_{-\ast}(\varOmega M_K)$.
		\end{lma}
		\begin{proof}
			Let $[\gamma_1], [\gamma_2]\in \pi_1(M_K)$. As above, we assign to $\gamma_1$ and $\gamma_2$ the cohomology classes $\sigma_{\gamma_1}, \sigma_{\gamma_2} \in H_{-\ast}(\varOmega M_K)$. Assign to the composition $\gamma_1 \gamma_2$ the cohomology class
			\[
				\sigma_{\gamma_1 \gamma_2} \defeq \sigma_{\gamma_1} \circ \sigma_{\gamma_2} = (-1)^{\abs{\sigma_{\gamma_1}}} P(\sigma_{\gamma_2} \otimes \sigma_{\gamma_1}) \in H_{-\ast}(\varOmega M_K)\, .
			\]
			Since $P$ is associative up to a sign in cohomology we have
			\begin{align*}
				\gamma_1(\gamma_2 \sigma) &= P(P(\sigma \otimes \sigma_{\gamma_2}) \otimes \sigma_{\gamma_1}) = (-1)^{\abs{\sigma_{\gamma_1}}}P(\sigma\otimes P(\sigma_{\gamma_2} \otimes \sigma_{\gamma_1})) \\
				&= (-1)^{\abs{\sigma_{\gamma_1}}+\abs{\sigma_{\gamma_2}}} P(\sigma \otimes \sigma_{\gamma_1 \gamma_2}) = (\gamma_1 \gamma_2) \sigma \, .
			\end{align*}
		\end{proof}
		By linearity we extend the action to a $\Z[\pi_1(M_K)]$-module structure on $H_{-\ast}(\varOmega M_K)$. 

		Consider a generator $a_\gamma \in HW^\ast_{\varLambda_K}(F,F)$, and denote by $\gamma$ the geodesic that $a_\gamma$ corresponds to. Via $\gamma$, we let $\sigma_\gamma \in H_{-\ast}(\varOmega M_K)$ be the cohomology class corresponding to $a_\gamma \in HW^\ast_{\varLambda_K}(F,F)$. Then define
		\begin{align*}
			\pi_1(M_K) \times HW^\ast_{\varLambda_K}(F,F) &\longrightarrow HW^\ast_{\varLambda_K}(F,F) \\
			(\gamma,a) &\longmapsto \gamma a\defeq (-1)^{\abs{a_\gamma}}\mu^2(a \otimes a_\gamma)\, .
		\end{align*}
		\begin{lma}\label{lma:grp_action_on_wrapped_floer}
			The map $(\gamma,a) \longmapsto \gamma a$ defines a group action of $\pi_1(M_K)$ on $HW^\ast_{\varLambda_K}(F,F)$.
		\end{lma}
		\begin{proof}
			Let $\sigma_{\gamma_{1}}, \sigma_{\gamma_{2}}$ and $\sigma_{\gamma_{1}\gamma_{2}}$ be as in the proof of \cref{lma:grp_action_on_cohomology_of_loops} above. Let $a\in HW^\ast_{\varLambda_K}(F,F)$ be any generator. Because $\mu^2$ is associative up to a sign in cohomology we have
			\begin{equation}\label{eq:first_part_of_group_action_on_HW}
				\gamma_1(\gamma_2 a) = (-1)^{\abs{a_{\gamma_1}} + \abs{a_{\gamma_2}}}\mu^2(\mu^2(a \otimes a_{\gamma_2}) \otimes a_{\gamma_1}) = (-1)^{\abs{a_{\gamma_2}}} \mu^2(a \otimes \mu^2(a_{\gamma_2} \otimes a_{\gamma_1}))\, .
			\end{equation}
			Because $\parenm{\varPsi_m}_{k=1}^\infty$ is an $A_{\infty}$-homomorphism, we glue the two disks contributing to $\varPsi_1(a_{\gamma_1}) = \sigma_{\gamma_1}$ and $\varPsi_1(a_{\gamma_2}) = \sigma_{\gamma_2}$ to obtain
			\begin{equation}\label{eq:pontryagin_product_giving_group_action}
				P(\varPsi_1(a_{\gamma_2}) \otimes \varPsi_1(a_{\gamma_1})) = \varPsi_1(\mu^2(a_{\gamma_2} \otimes a_{\gamma_1})) \, .
			\end{equation}
			Hence there exists a $J$-holomorphic disk in the symplectization of $\dd W_K$ with two positive punctures $a_{\gamma_1}$ and $a_{\gamma_2}$. Define $a_{\gamma_1 \gamma_2} \defeq (-1)^{\abs{\sigma_{\gamma_1}}}\mu^2(a_{\gamma_2}, a_{\gamma_1})$. Then \eqref{eq:pontryagin_product_giving_group_action} says that
			\[
				\varPsi_1(a_{\gamma_1 \gamma_2}) = (-1)^{\abs{\sigma_{\gamma_1}}}P(\sigma_{\gamma_2} \otimes \sigma_{\gamma_1}) = \sigma_{\gamma_1 \gamma_2}\, .
			\]
			Thus $a_{\gamma_1 \gamma_2} \in HW^\ast_{\varLambda_K}(F,F)$ is the generator corresponding to the concatenation $[\gamma_1 \gamma_2] \in \pi_1(M_K)$. Combining this with \eqref{eq:first_part_of_group_action_on_HW} gives
			\begin{align*}
				\gamma_1(\gamma_2 a) &= (-1)^{\abs{a_{\gamma_1}} + \abs{a_{\gamma_2}}}\mu^2(\mu^2(a \otimes a_{\gamma_2}) \otimes a_{\gamma_1}) \\
				&= (-1)^{\abs{a_{\gamma_2}}} \mu^2(a \otimes \mu^2(a_{\gamma_2} \otimes a_{\gamma_1})) = (-1)^{\abs{a_{\gamma_1 \gamma_2}}} \mu^2(a \otimes a_{\gamma_1 \gamma_2}) = (\gamma_1 \gamma_2)a
			\end{align*}
			Lastly, we need to prove that if $\gamma_{\text{const}} \in \pi_1(M_K)$ is the constant loop, then $\mu^2(a \otimes a_{\gamma_{\text{const}}}) = a$. This follows from the definition of $\varPsi_1$. The generator of $HW^\ast_{\varLambda_K}(F,F)$ which corresponds to the (0-chain of) the constant loop is the unique Lagrangian intersection generator in the compact part of $W_K$ which corresponds to the unique intersection point of $F_i$ with $F_j$, call it $x\in HW^\ast_{\varLambda_K}(F,F)$. Therefore
			\[
				\gamma_{\text{const}}a = \mu^2(a \otimes x) = a\, .
			\]
		\end{proof}
		By linearity we extend the action to a $\Z[\pi_1(M_K)]$-module structure on $HW^\ast_{\varLambda_K}(F,F)$.
		\begin{thm}[\cref{thm:theorem_A}]\label{thm:extending_Y1_to_Zpi_module_iso}
			The isomorphism
			\[
				\varPsi_1 \co HW^\ast_{\varLambda_K}(F,F) \longrightarrow H_{-\ast}(\varOmega M_K)\, ,
			\]
			is an isomorphism of $\Z[\pi_1(M_K)]$-modules.
		\end{thm}
		\begin{proof}
			Let $a\in HW^\ast_{\varLambda_K}(F,F)$ be a generator and let $[\gamma]\in \pi_1(M_K)$ be a homotopy class represented by a unique minimizing geodesic $\gamma$. Then consider a generator $a_\gamma\in HW^\ast_{\varLambda_K}(F,F)$ so that $\varPsi_1(a_\gamma) = \sigma_\gamma$, where $\sigma_\gamma \in H_{-\ast}(\varOmega M_K)$ is the cohomology class corresponding to $\gamma$. Then we have
			\begin{align*}
				\varPsi_1(\gamma a) &= (-1)^{\abs{a_\gamma}}\varPsi_1(\mu^2(a \otimes a_\gamma)) \overset{\eqref{eq:pontryagin_product_giving_group_action}}{=} (-1)^{\abs{a_\gamma}}P(\varPsi_1(a) \otimes \varPsi_1(a_\gamma)) \\
				&= (-1)^{\abs{\sigma_\gamma}}P(\varPsi_1(a)\otimes \sigma_\gamma) = \gamma \varPsi_1(a)\, .
			\end{align*}
		\end{proof}
		\begin{rmk}
			Note that $\Z[\pi_1(M_K)] \cong H_0(\varOmega M_K)$, and consider $C_0^{\text{cell}}(BM_K)$ as an $A_\infty$-algebra with operations $\parenm{m_i}_{i=1}^\infty$ where $m_1 = 0$, $m_2 = P$ is the Pontryagin product, and $m_i = 0$ for $i \geq 3$. We observe that $CW^\ast_{\varLambda_K}(F,F)$ can be equipped with the structure of a left $A_3$-module over $C^{\text{cell}}_0(BM_K)$. More precisely we define this left $A_3$-module structure as a sequence of maps
			\[
				\nu^r_{CW} \co (C^{\text{cell}}_0(BM_K))^{\otimes(r-1)} \otimes CW^\ast_{\varLambda_K}(F,F) \longrightarrow CW^\ast_{\varLambda_K}(F,F)\, ,
			\]
			defined by
			\[
				\begin{cases}
					\nu^1_{CW}(c) \defeq \mu^1(c) \\
					\nu^2_{CW}(x \otimes c) \defeq \mu^2(a_x \otimes c) \\
					\nu^3_{CW}(x_2 \otimes x_1 \otimes c) \defeq \mu^3(a_{x_2} \otimes a_{x_1} \otimes c) \\
					\nu^k_{CW}(x_{k-1} \otimes \cdots \otimes x_1 \otimes c) \defeq 0, & k \geq 4\, ,
				\end{cases}
			\]
			where $a_{x_i}$ is the unique Reeb chord corresponding to $x_i$ via $\varPsi_1$. Then by computation we have that $\parenm{\nu^r_{CW}}_{r=1}^\infty$ satisfies the following equation for $n \in \parenm{1,2,3}$.
			\begin{align}\label{eq:a_infinity_module_structure}
				&\sum_{i=0}^{n-1} m_{n-i}(x_{n-1}\otimes \cdots \otimes x_{i+1} \otimes \nu^{i+1}_{CW}(x_{i}\otimes \cdots \otimes x_{1} \otimes c)) \\
				&+  \sum_{\substack{\ell+k < r \\ \ell \geq 1 \\ k\geq 0}} (-1)^{\abs c + \maltese_k}\nu^{r- \ell+1}_{CW}(x_{r-1}\otimes \cdots \otimes x_{k+ \ell+1} \otimes \mu^\ell(x_{k+\ell}\otimes \cdots \otimes x_{k+1}) \otimes x_k \otimes \cdots \otimes x_1 \otimes c) = 0 \nonumber
			\end{align}
			Note that this means that there is a group action up to homotopy of $C_0^{\text{cell}}(BM_K)$ on $CW^\ast_{\varLambda_K}(F,F)$, but there are no higher coherent homotopies. However, this is enough to directly obtain \cref{lma:grp_action_on_cohomology_of_loops} and \cref{lma:grp_action_on_wrapped_floer}.

			Since $C^{\text{cell}}_{-\ast}(BM_K)$ is an $A_\infty$-algebra, it can be regarded as a left $A_\infty$-module over itself, and therefore also as a left $A_\infty$-module over $C^{\text{cell}}_0(BM_K)$ via the sequence of maps
			\[
				\nu^r_{\text{cell}}\co (C^{\text{cell}}_0(BM_K))^{\otimes (r-1)} \otimes C^{\text{cell}}_{-\ast}(BM_K) \longrightarrow C^{\text{cell}}_{-\ast}(BM_K)\, ,
			\]
			defined by
			\[
				\begin{cases}
					\nu^1_{\text{cell}}(y) \defeq m_1(y) \\
					\nu^2_{\text{cell}}(x \otimes y) \defeq P(x \otimes y) \\
					\nu^k_{\text{cell}}(x_{k-1} \otimes \cdots \otimes x_1 \otimes y) \defeq 0, & k\geq 3\, .
				\end{cases}
			\]
			By a computation we see that $\parenm{\nu^r_{\text{cell}}}_{r=1}^\infty$ satisfies \eqref{eq:a_infinity_module_structure} for every $n \in \Z_+$. For $n\geq 4$ the equation is trivial.

			Furthermore we have that the $A_\infty$-homomorphism $\parenm{\varPsi_k}_{k=1}^\infty$ induces an isomorphism of $A_3$-modules over $C^{\text{cell}}_0(BM_K)$ as follows. The isomorphism of $A_3$-modules over $C^{\text{cell}}_0(BM_K)$ is a sequence of maps
			\[
				\psi_r \co (C^{\text{cell}}_0(BM_K))^{\otimes(r-1)} \otimes CW^\ast_{\varLambda_K}(F,F) \longrightarrow C^{\text{cell}}_{-\ast}(BM_K)
			\]
			defined by
			\[
				\psi_r(x_{r-1} \otimes \cdots \otimes x_1 \otimes c) \defeq \varPsi_r(a_{x_{r-1}} \otimes \cdots \otimes a_{x_{1}} \otimes c)\, ,
			\]
			where $a_{x_i}$ is the unique generator corresponding to $x_i$ via $\varPsi_1$. Then by computation we have that $\parenm{\psi_r}_{r=1}^\infty$ satisfies the following equation for $n\in \parenm{1,2,3}$.
			\begin{align*}
				&\sum_{i=0}^{r-1} \nu^{r-i}_{\text{cell}}(x_{r-1}\otimes \cdots \otimes x_{i+1} \otimes \psi_{i+1}(x_{i}\otimes \cdots \otimes x_{1} \otimes c))\\
				&= \sum_{i=0}^{r-1} \psi_{r-i}(x_{r-1}\otimes \cdots \otimes x_{i+1} \otimes \nu^{i+1}_{CW}(x_{i}\otimes \cdots \otimes x_{1} \otimes c)) \\
				&+ \sum_{\substack{s+t+k = r \\ t,k\geq 1 \\ s\geq 0}} (-1)^{\abs c + \maltese_k} \psi_{r-\ell+1}(x_{r-1} \otimes \cdots \otimes x_{k+\ell+1} \otimes m_\ell(x_{k+\ell} \otimes \cdots \otimes x_{k+1}) \otimes x_k \otimes \cdots \otimes x_1 \otimes c) = 0 \, ,
			\end{align*}
			The fact that this is an $A_3$-module isomorphism directly implies \cref{thm:extending_Y1_to_Zpi_module_iso}.
		\end{rmk}
	\subsection{Plumbings and infinite cyclic covers}
		\label{sub:plumbings_and_infinite_cyclic_covers}
		In this section we review standard background material from \cite{rolfsen1976knots}.
		
		Let $p,q \geq 2$ and $n = p+q+1$. We consider the plumbing of $S^p$ with $S^q$. That is, consider $S^p \times D^q$ and $S^q \times D^p$. By identifying $D^p$ with the upper hemisphere of $S^p$, we have 
		\begin{align*}
			D^p \times D^q \subset S^p \times D^q \\
			D^q \times D^p \subset S^q \times D^p\, .
		\end{align*}
		We then take the disjoint union of $S^p \times D^q$ with $S^q \times D^p$ and identify their common submanifolds $D^p \times D^q \cong D^q \times D^p$ via $f\co (x,y) \longmapsto (y,x)$. We call the resulting space the plumbing of $S^p$ and $S^q$, denoted by $S^p \#_{\text{plumb}} S^q$. In short we write
		\[
			\varSigma = S^p \#_{\text{plumb}} S^q \defeq (S^p \times D^q) \sqcup_{f} (S^q \times D^p)\, .
		\]
		We note that $S^p \vee S^q$ is the deformation retract of $\varSigma$. Let $K \defeq \dd \varSigma$ and note that it is a $(p+q-1)$-dimensional sphere. Embed $\varSigma$ into $S^n$ and consider the complement of its boundary $M_K \defeq S^n \setmin K$; denote its infinite cyclic cover by $\widetilde M_K$. 

		Following \cite[Section 5.C]{rolfsen1976knots} we find the simplicial structure of $\widetilde M_K$ by cutting along $\varSigma$. More precisely, let $\varSigma^{\pm} \cong \overset{\circ}{\varSigma} \times (-1,1)$ be an open bicollar of the interior of $\varSigma$, and let
		\begin{align*}
			\varSigma^+ &\defeq \overset{\circ}{\varSigma} \times (0,1) \subset S^n \\
			\varSigma^- &\defeq \overset{\circ}{\varSigma} \times (-1,0) \subset S^n \\
			M_{\varSigma} &\defeq S^n \setmin \varSigma\, .
		\end{align*}
		Consider infinitely many copies of each of $\varSigma^+$, $\varSigma^-$ and $M_{\varSigma}$. Denote the copies by $M_{\varSigma ;i}$, $\varSigma^+_i$ and  $\varSigma^-_i$ for $i\in \Z$. Then consider the disjoint union of all the $M_{\varSigma; i}$ and glue them together by identifying $\varSigma_i^+ \subset M_{\varSigma; i}$ with $\varSigma_{i+1}^- \subset M_{\varSigma;i+1}$ via the map
		\begin{align*}
			\varSigma_i^+ = \overset{\circ}{\varSigma} \times (0,1) &\longrightarrow \overset{\circ}{\varSigma} \times (-1,0) = \varSigma_{i+1}^- \\
			(\sigma,t) &\longmapsto (\sigma,t-1)\, .
		\end{align*}
		Define
		\begin{equation}\label{eq:construction_of_infinite_cyclic_cover}
			\widetilde M_K \defeq \coprod_{i=-\infty}^\infty M_{\varSigma;i}/(\varSigma^+_i \sim \varSigma^-_{i+1})\, .
		\end{equation}
		\begin{figure}[H]
			\centering
			\includegraphics{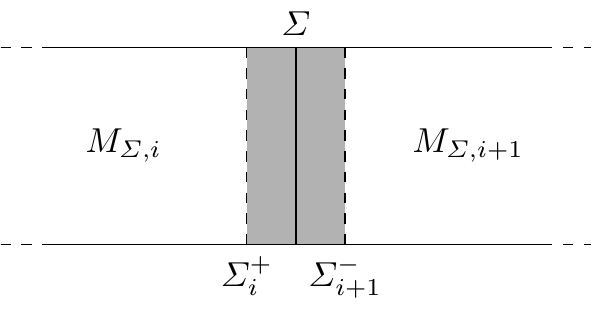}
		\end{figure}
	\subsection{The Alexander invariant}\label{sub:the_alexander_invariant}
		In this section we review standard material on the Alexander invariant from \cite{rolfsen1976knots}.

		Associated to the open cover $\mathcal U = \paren{M_{\varSigma;i}}_{i=-\infty}^\infty$ of $\widetilde M_K$ is the sequence of inclusions
		\[
			\begin{tikzcd}[row sep=scriptsize, column sep=scriptsize]
				\coprod_{i=-\infty}^\infty M_{\varSigma;i} \cap M_{\varSigma; i+1} \rar[shift left]{\iota_{i+1}} \rar[shift right, swap]{\iota_{i}} & \coprod_{i=-\infty}^\infty M_{\varSigma;i} \rar{\kappa} & \widetilde M_K
			\end{tikzcd}
		\]
		from which we get a short exact sequence in singular chains (cf.\@ \cite[Section 8]{bott2013differential})
		\begin{equation}\label{eq:mayer-vietoris_applied_to_infinite_cyclic_cover}
			\begin{tikzcd}[row sep=scriptsize, column sep=scriptsize]
				0 \rar & C_{\ast} \paren{\coprod_{i=-\infty}^\infty M_{\varSigma;i} \cap M_{\varSigma;i+1}} \rar{\alpha_\ast} & C_{\ast}(\coprod_{i=-\infty}^\infty M_{\varSigma;i}) \rar{\beta_\ast} & C_{\ast}(\widetilde M_K) \rar & 0
			\end{tikzcd}\, .
		\end{equation}
		Let $x = (x_i)_{i\in \Z} \in C_{\ast} \paren{\coprod_{i=-\infty}^\infty M_{\varSigma;i} \cap M_{\varSigma;i+1}}$. Then
		\[
			\alpha_\ast x = \paren{(\iota_{i})_\ast(x_i) - (\iota_i)_\ast(x_{i-1})}_{i\in \Z}\, ,
		\]
		and for any $y = (y_i)_{i\in \Z} \in C_{\ast}(\coprod_{i=-\infty}^\infty M_{\varSigma;i})$ we have
		\[
			\beta_\ast y = \sum_{i=-\infty}^\infty \kappa_\ast(y_i)\, .
		\]
		Since
		\[
			M_{\varSigma;i} \cap M_{\varSigma;i+1} = \varSigma_i^+ \simeq S^p \vee S^q\, ,
		\]
		the short exact sequence \eqref{eq:mayer-vietoris_applied_to_infinite_cyclic_cover} induces a long exact sequence in homology
		\[
			\begin{tikzcd}[column sep=scriptsize]
				\cdots \rar & \bigoplus_{i=-\infty}^\infty H_{j}(S^p \vee S^q) \rar{\alpha_\ast} \ar[draw=none]{d}[name=X, anchor=center]{} & \bigoplus_{i=-\infty}^\infty H_{j}(M_{\varSigma;i}) \rar & H_{j}(\widetilde M_K)
				\ar[rounded corners,
					to path={ -- ([xshift=2ex]\tikztostart.east)
					|- (X.center) \tikztonodes
					-| ([xshift=-2ex]\tikztotarget.west)
					-- (\tikztotarget)}]{dll}[at end]{} & {} \\
				{}& \bigoplus_{i= -\infty}^\infty H_{j-1}(S^p \vee S^q) \rar{\alpha_\ast} & \bigoplus_{i=-\infty}^\infty H_{j-1}(M_{\varSigma;i}) \rar & \cdots
			\end{tikzcd}\, .
		\]
		We have
		\[
			\widetilde H_j(S^p \vee S^q) \cong \widetilde H_j(S^p) \oplus \widetilde H_j(S^q)\, ,
		\]
		where $\widetilde H$ denotes reduced homology. Since $M_{\varSigma} = S^n \setmin \varSigma$, Alexander duality gives that
		\[
			\widetilde H_j(M_{\varSigma;i}) \cong \widetilde H^{n-j-1}(\varSigma) \cong \widetilde H^{n-j-1}(S^p \vee S^q) \cong \widetilde H^{n-j-1}(S^p) \oplus \widetilde H^{n-j-1}(S^q)\, .
		\]
		Since $n = p+q+1$, we get
		\[
			H_j(S^p \vee S^q) \cong H_j(M_{\varSigma;i}) \cong \begin{cases}
				\Z, & j = 0,p,q \\
				0, & \text{otherwise,}
			\end{cases}
		\]
		which means that $H_j(\widetilde M_K) \cong 0$ unless $j \in \parenm{0,p,q}$.

		Since the group of deck transformations of $\widetilde M_K$ is infinite cyclic, we choose a generator $\tau\in \mathrm{Aut}(\widetilde M_K, \pi)$ which induces an automorphism
		\[
			\tau_\ast \co H_\ast(\widetilde M_K) \longrightarrow H_\ast(\widetilde M_K)\, .
		\]
		This gives a $\Z[t^{\pm 1}]$-module structure on $H_\ast(\widetilde M_K)$ as follows. Let $p(t) = \sum_{i=-s}^r c_i t^i \in \Z[t^{\pm 1}]$, then for any $\alpha\in H_\ast(\widetilde M_K)$ let
		\[
			p(t)\alpha = \sum_{i=-s}^r c_i \tau_\ast^i(\alpha)\, ,
		\]
		where $\tau_\ast^i$ is the $i$-fold composition power of $\tau_\ast$. The \emph{Alexander invariant} is then defined as $H_{\ast}(\widetilde M_K)$ considered as a $\Z[t^{\pm 1}]$-module.
		\begin{lma}[{\cite[Theorem 7.G.1]{rolfsen1976knots}}]\label{lma:non-trivial_knots_with_pi1_Z}
			There exist non-trivial knots $K \subset S^{n+2}$ with infinite cyclic knot group, $\pi_1(M_K) \cong \Z$.
		\end{lma}
		\begin{proof}
			Let $p, q \geq 2$ and let $n = p+q-1$. We then consider any $K$ obtained as $\dd \varSigma$, where
			\[
				\varSigma = S^p \#_{\text{plumb}} S^q\, .
			\]
			Now we have that $\widetilde M_{\varSigma}$ is simply connected: Every loop in $S^{n+2}$ shrinks missing $\varSigma$ since $\varSigma$ is homotopy equivalent to $S^p \vee S^q$. This is because $\codim(S^p) \geq 3$ and $\codim(S^q) \geq 3$ in $S^{n+2}$. From the construction of $\widetilde M_K$ in \eqref{eq:construction_of_infinite_cyclic_cover} we thus have $\pi_1(\widetilde M_K) \cong 1$. Hence, because the group of deck transformations of $\widetilde M_K \longrightarrow M_K$ is $\Z$, we have $\pi_1(M_K) \cong \Z$.
			\begin{figure}[H]
				\centering
				\includegraphics{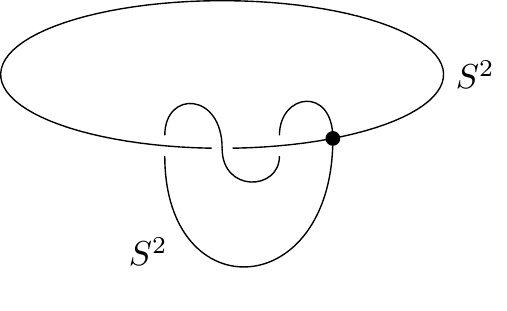}
				\caption{The core of the self plumbing of two knotted $S^2$ embedded in $S^5$.}
				\label{fig:self_plumbing_of_S2_diagram}
			\end{figure}
			To see that such non-trivial $K$ exists, we may consider $K = \dd(S^2 \#_{\text{plumb}} S^2) \subset S^5$, where the core of the plumbing is shown in \cref{fig:self_plumbing_of_S2_diagram}. The Alexander invariant of $K$ is non-trivial by a computation (cf.\@ \cite[Exercise 7.F.5]{rolfsen1976knots}).
		\end{proof}
	\subsection{Using the Leray--Serre spectral sequence}
		\label{sub:using_the_leray--serre_spectral_sequence}
		Consider the knot $K = S^p \#_{\mathrm{plumb}} S^q \subset S^{n}$ with $p+q = n-1$ as in \cref{sub:plumbings_and_infinite_cyclic_covers}. We use the notation $\Z \pi \defeq \Z[\pi_1(M_K)]$.

		Associated to the path-loop fibration
		\[
			\begin{tikzcd}[row sep=scriptsize, column sep=scriptsize]
				\varOmega M_K \rar[hook] & PM_K \dar \\
				& M_K
			\end{tikzcd}
		\]
		is the Leray--Serre spectral sequence. It is first quadrant spectral sequence $\parenm{E^r_{i, j}, d^r_{i,j}}_{i,j \in \N}$ of $\Z \pi$-modules which converges:
		\[
			E^2_{i,j} \cong H_i(M_K; H_j(\varOmega M_K)) \Longrightarrow H_{i+j}(PM_K) = \begin{cases}
				\Z \pi/(t-1), & i+j = 0\\
				0, & \text{otherwise,}
			\end{cases}
		\]
		Note that $\pi_1(M_K)$ is Abelian and hence we can consider $H_i(M_K; H_j(\varOmega M_K))$ as a $\Z \pi$-module.

		Since $C_{\ast}(\widetilde M_K)$ is only supported in degrees $0$, $p$ and $q$ we have the following facts
		\begin{itemize}
			\item Following \cite[Section 3.H]{hatcher2002algebraic} and \cite{shulman2010equivariant} we have the following identification
			\[
				H_i(M_K; H_j(\varOmega M_K)) \cong H_i \paren{C_{\ast}(\widetilde M_K) \otimes_{\Z \pi} H_j(\varOmega M_K)}\, .
			\]
			Assume that $\abs{p-q} \neq 1$. Then we trivially have
			\[
				H_i(M_K; H_j(\varOmega M_K)) = \begin{cases}
					H_{j}(\varOmega M_K), & i = 0 \\
					H_p(\widetilde M_K) \otimes_{\Z \pi} H_j(\varOmega M_K), & i = p \\
					H_q(\widetilde M_K) \otimes_{\Z \pi} H_j(\varOmega M_K), & i = q \\
					0, & \text{otherwise,}
				\end{cases}
			\]
			because $C_\ast(\widetilde M_K)$ is only supported in $\ast \in \parenm{0,p,q}$.
			\item $E^2_{i,j}$ is only supported on the vertical lines $i \in \parenm{0,p,q}$.
			\item The bottom row is $E^2_{i,0} = H_i(\widetilde M_K)$, since $H_0(\varOmega M_K) \cong \Z \pi$.
		\end{itemize}
		
		\begin{ex}\label{ex:plumbing_Sp_Sp}
			Consider the case when $K$ is obtained as the boundary of $S^p \#_{\text{plumb}} S^p \subset S^{2p+1}$ where the core of the plumbing is depicted in \cref{fig:self_plumbing_of_S2_diagram}. In this case, $E^2_{i,j}$ is only supported at the vertical lines $i \in \parenm{0,p}$. For this spectral sequence, the $p$-th page is the first page after page 1 that has non-zero differentials. Namely, the $p$-th page of this spectral sequence is
		\begin{figure}[H]
			\centering
			\includegraphics{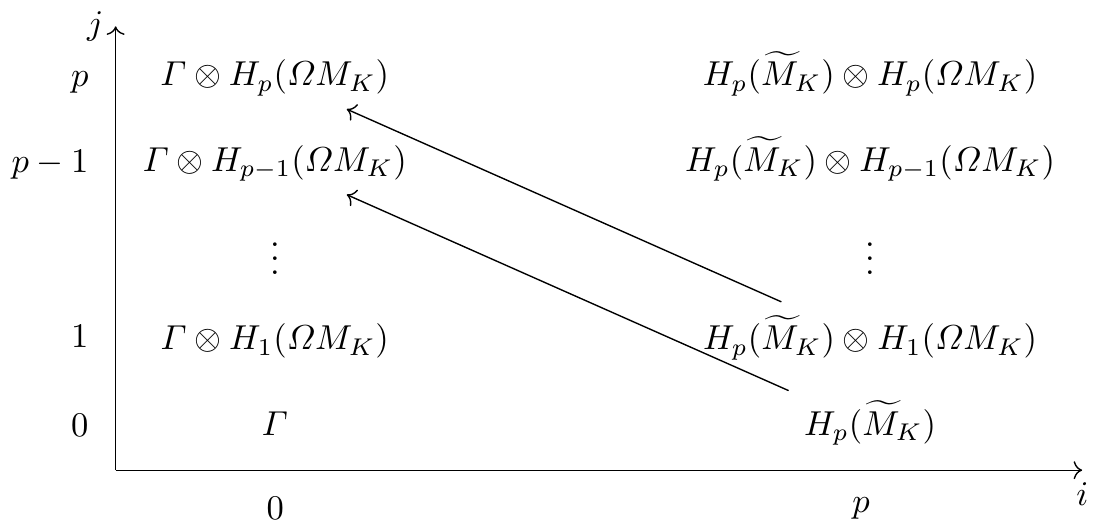}
		\end{figure}
		Where $\varGamma = \Z \pi/(t-1)$ and every tensor product is taken over $\Z \pi$. The differentials at every page succeeding the $p$-th page is zero, so in particular we get 
		\[
			H_p(\widetilde M_K) \cong \Z \pi/(t-1) \otimes_{\Z \pi} H_{p-1}(\varOmega M_K) \cong \Z \pi/(t-1) \otimes_{\Z \pi} HW^{1-p}_{\varLambda_K}(F,F)\, .
		\]
		\end{ex}
		\begin{ex}\label{ex:plumbing_Sp_S2p}
			Suppose $K$ is obtained as the boundary of $S^p \#_{\text{plumb}} S^{2p} \subset S^{3p+1}$ for $p\geq 2$ where the core of the plumbing is depicted in \cref{fig:self_plumbing_of_S2_diagram}. In this case the second page of the spectral sequence is only supported at the lines $i \in \parenm{0,p,2p}$. The $p$-th page of the spectral sequence is the first page after page 1 that has non-zero differentials and it looks as follows:
			\begin{figure}[H]
				\centering
				\includegraphics{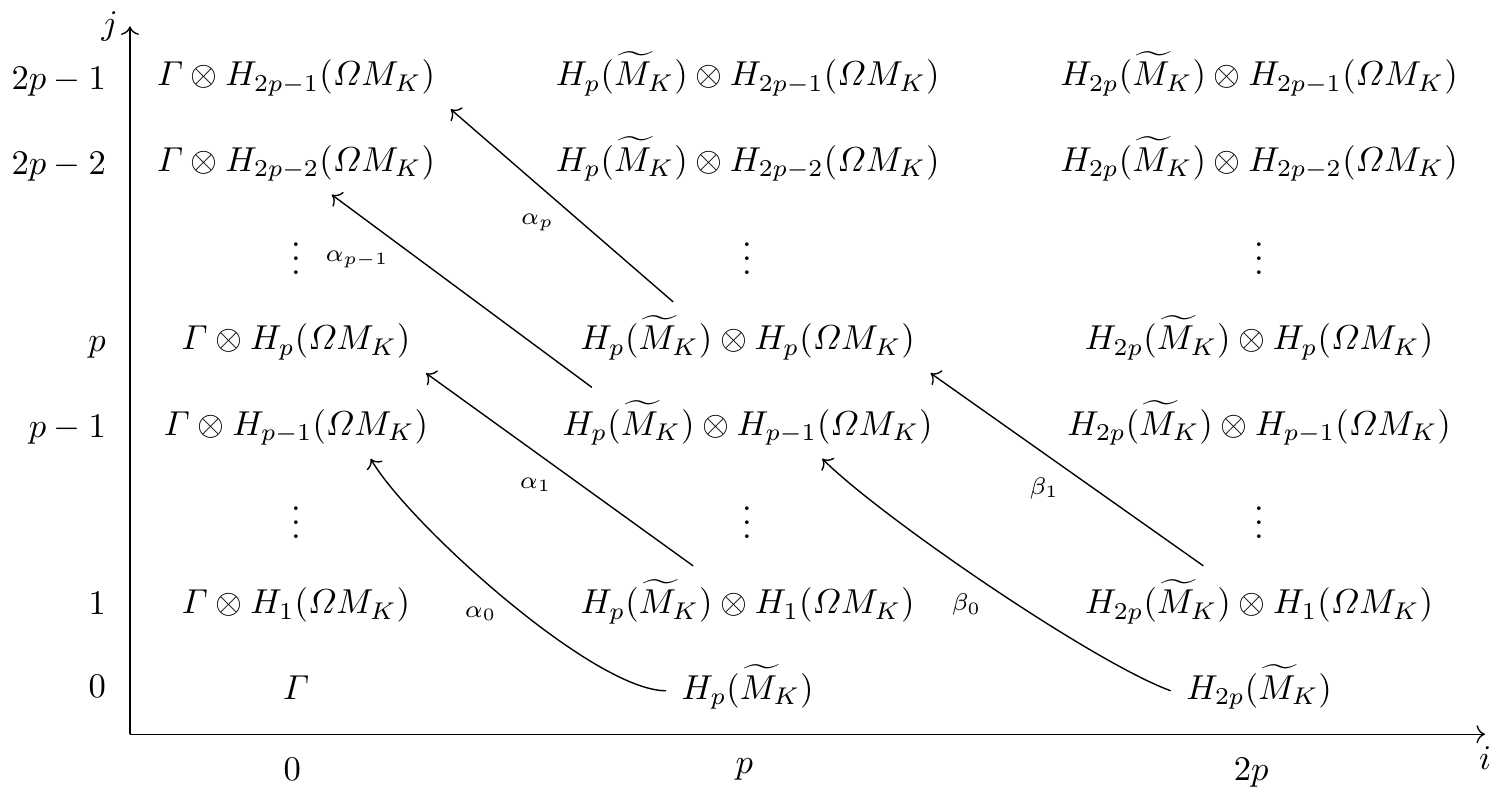}
			\end{figure}
			Immediately from this page, we get an isomorphism of $\Z \pi$-modules
			\begin{equation}\label{eq:computation_of_Hp}
				H_p(\widetilde M_K) \cong \Z \pi/(t-1) \otimes_{\Z \pi} H_{p-1}(\varOmega M_K) \cong \Z \pi/(t-1) \otimes_{\Z \pi} HW^{1-p}_{\varLambda_K}(F,F)\, .
			\end{equation}
			Furthermore, the next page with non-zero differentials is page $2p$, which looks as follows
			\begin{figure}[H]
				\centering
				\includegraphics{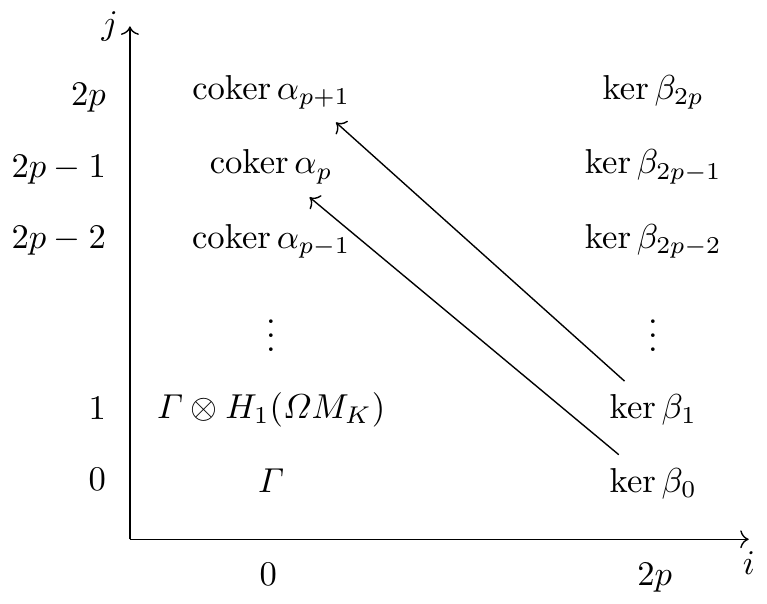}
			\end{figure}
			This is the last page with non-zero differentials, so $\coker \alpha_{p-1} \cong \parenm{0}$ and hence from page $p$, we obtain an exact sequence
			\begin{equation}\label{eq:exact_sequence_involving_H2p}
				\begin{tikzcd}[row sep=scriptsize, column sep=scriptsize]
					H_{2p}(\widetilde M_K) \rar{\beta_0} & H_p(\widetilde M_K) \otimes_{\Z \pi} H_{p-1}(\varOmega M_K) \rar{\alpha_{p-1}} & \varGamma \otimes_{\Z \pi} H_{2p-2}(\varOmega M_K) \rar & 0
				\end{tikzcd}\, ,
			\end{equation}
			and for each $i\in \Z_+$ we have an exact sequence
			\[
				\begin{tikzcd}[row sep=scriptsize, column sep=scriptsize]
					H_{2p}(\widetilde M_K) \otimes H_i(\varOmega M_K) \rar{\beta_i} & H_p(\widetilde M_K) \otimes_{\Z \pi} H_{p-1+i}(\varOmega M_K) \rar{\alpha_{p-1+i}} & \varGamma \otimes_{\Z \pi} H_{2p-2+i}(\varOmega M_K)
				\end{tikzcd}\, .
			\]
			Furthermore, from page $2p$ we get isomorphisms
			\[
				\ker \beta_j \cong \coker \alpha_{p+j}\, ,
			\]
			for every $j\in \N$. In particular, in view of \eqref{eq:exact_sequence_involving_H2p}, we have $\ker \beta_0 \cong \coker \alpha_p$ where
			\[
				\alpha_p \co H_p(\widetilde M_K) \otimes_{\Z \pi} H_p(\varOmega M_K) \longrightarrow \varGamma \otimes_{\Z \pi} H_{2p-1}(\varOmega M_K) \, .
			\]
			So if $H_\ast(\varOmega M_K)$ and $\alpha_p$ is known, we compute $H_p(\widetilde M_K)$ by \eqref{eq:computation_of_Hp} but also a quotient of $H_{2p}(\widetilde M_K)$ by exactness of \eqref{eq:exact_sequence_involving_H2p}
			\[
				H_{2p}(\widetilde M_K)/\coker \alpha_p \cong \ker \alpha_{p-1}\, .
			\]
			Let us summarize what we have.
			\begin{align*}
				\begin{cases}
					H_p(\widetilde M_K) \cong \Z \pi/(t-1) \otimes_{\Z \pi} HW^{1-p}_{\varLambda_K}(F,F)\\
					H_{2p}(\widetilde M_K)/\coker \alpha_p \cong \ker \alpha_{p-1}\, ,
				\end{cases}
			\end{align*}
			where
			\begin{align*}
				\alpha_p \co H_p(\widetilde M_K) \otimes_{\Z \pi} HW^{-p}_{\varLambda_K}(F,F) &\longrightarrow \Z \pi/(t-1) \otimes_{\Z \pi} HW_{\varLambda_K}^{1-2p}(F,F) \\
				\alpha_{p-1} \co H_p(\widetilde M_K) \otimes_{\Z \pi} HW_{\varLambda_K}^{1-p}(F,F) &\longrightarrow \Z \pi/(t-1) \otimes_{\Z \pi} HW_{\varLambda_K}^{2-2p}(F,F)\, .
			\end{align*}
		\end{ex}
		\begin{ex}\label{ex:plumbing_Sp_Sq}
			For a slightly more general case, where $p \geq 2$ and $q > p+1$ we consider again $K$ to be the boundary of $S^p \#_{\text{plumb}} S^q \subset S^{p+q+1}$ where the core of the plumbing is depicted in \cref{fig:self_plumbing_of_S2_diagram}. The Leray--Serre spectral sequence is supported at the lines $i\in \parenm{0,p,q}$. Exactly like in \cref{ex:plumbing_Sp_S2p}, we compute 
			\[
				H_p(\widetilde M_K) \cong \Z \pi/(t-1) \otimes_{\Z \pi} HW^{1-p}_{\varLambda_K}(F,F)\, .
			\]
		\end{ex}
		Let $\varLambda_{\mathrm{unknot}}$ denote the unit conormal of the standard embedded $S^{n-2} \subset S^n$. A consequence of these computations is the following theorem.
		\begin{thm}[\cref{thm:thm_B}]\label{thm:exists_codim_2_knot_with_conormal_not_leg_iso_to_unknot}
			Let $n = 5$ or $n\geq 7$. Let $x \in M_K$ be a point. Then there exists a codimension 2 knot $K \subset S^{n}$ with $\pi_1(M_K) \cong \Z$, such that $\varLambda_{K} \cup \varLambda_x$ is not Legendrian isotopic to $\varLambda_{\mathrm{unknot}} \cup \varLambda_x$.
		\end{thm}
		\begin{proof}
			For the case $n = 5$, consider the knot $K = \dd(S^2 \#_{\mathrm{plumb}} S^2) \subset S^5$ where the core of the plumbing is depicted in \cref{fig:self_plumbing_of_S2_diagram}. In the case $n\geq 7$ we let $p\geq 2$ and $q > p+1$ and consider $K = \dd(S^p \#_{\mathrm{plumb}} S^q) \subset S^{p+q+1}$, where the core of the plumbing is again depicted in \cref{fig:self_plumbing_of_S2_diagram}.

			We note that for dimensional reasons we have $\pi_1(M_K) \cong \Z$, but the Alexander invariant shows that $K$ is non-trivial \cite[Section 7.G]{rolfsen1976knots} (see also \cref{lma:non-trivial_knots_with_pi1_Z}).

			The computations in \cref{ex:plumbing_Sp_Sp} and \cref{ex:plumbing_Sp_Sq} show that in particular
			\[
				H_p(\widetilde M_K) \cong \Z[\pi_1(M_K)]/(t-1) \otimes_{\Z[\pi_1(M_K)]} HW^{1-p}_{\varLambda_K}(F,F)\, .
			\]
		Since we use classical methods to show that $H_p(\widetilde M_K)$ is non-trivial (see \cref{lma:non-trivial_knots_with_pi1_Z}, it follows that $HW^{1-p}_{\varLambda_K}(F,F)$ is non-trivial. Consider the unknot $S^{n-2} \subset S^n$, then the complement $M_{\mathrm{unknot}}$ is homotopy equivalent to a circle, which means that $H_{-\ast}(\varOmega M_{\mathrm{unknot}}) \cong HW^\ast_{\varLambda_{\mathrm{unknot}}}(F,F)$ is only supported in degree 0. Therefore we have $HW^{1-p}_{\varLambda_K}(F,F) \not\cong HW^{1-p}_{\varLambda_{\mathrm{unknot}}}(F,F)$ and so $\varLambda_K \cup \varLambda_x$ is not Legendrian isotopic to $\varLambda_{\mathrm{unknot}} \cup \varLambda_x$.
		\end{proof}
\appendix
\section{Monotonicity of $J$-holomorphic half strips}\label{sec:monotonicity_estimates}
	To establish compactness of the moduli spaces $\mathcal M(\boldsymbol a)$ in \cref{sub:moduli_space_of_half_strips} we need to make sure that $J$-holomorphic half strips in $\mathcal M(\boldsymbol a)$ does not escape to horizontal infinity. Pick a tubular neighborhood of $K \subset M_K$ and call it $N(K)$. Then we decompose $M_K$ as
		\[
			M_K \cong (S \setmin N(K)) \cup_{\dd N(K)} \paren{[0,\infty) \times \dd N(K)}\, ,
		\]
		where we identify $\dd(S \setmin N(K)) \cong \dd N(K)$ with $\parenm{0} \times \dd N(K) \cong \dd N(K)$. Pick a generic Riemannian metric $g$ on $S \setmin N(K)$ such that geodesics are non-degenerate critical points of the length and energy functionals. Define a function
		\[
			f\co [0,\infty) \longrightarrow [0,\infty)
		\]
		so that

		\centerline{\begin{minipage}{0.4\linewidth}
			\begin{equation}\label{eq:metric_definition}
				\begin{cases}
					f(0) = 1 \\
					f(t) > c_0 > 0, & \forall t \in [0,\infty) \\
					f'(0) = -1 \\
					f'(t) < 0, & \forall t \in [0,\infty) \\
					f''(t) \geq 0, & \forall t \in [0,\infty)
				\end{cases}
			\end{equation}
		\end{minipage}
		\begin{minipage}{0.4\linewidth}
			\begin{figure}[H]
				\centering
				\includegraphics{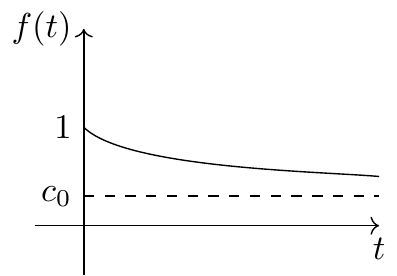}
			\end{figure}
		\end{minipage}}

		Define a metric $h$ on $M_K$ as
		\[
			h = \begin{cases}
				g, & \text{in } S \setmin N(K) \\
				dt^2 + f(t)\eval[0] g_{\dd N(K)}, & \text{in } [0,\infty) \times \dd N(K)\ ,
			\end{cases}
		\]
		where $t$ is the coordinate in the $[0,\infty)$-factor. Similar to the situation in~\cite[Appendix C]{ekholm2017duality}, if $x,y\in \dd N(K)$ are two points and $c\co [0,\ell] \longrightarrow M_K$ a geodesic with $c(s_1) = x \in \dd N(K)$ and $c(s_2) = y \in \dd N(K)$, then there is a unique geodesic $(t(s), c(s)) \in [0,\infty) \times \dd N(K)$ so that
		\begin{itemize}
			\item $(0, c(s_1)) = (0,x)$ and $(t(\ell), c(s_2)) = (0,y)$, and
			\item $t\co [0,\ell] \longrightarrow [0,\infty)$ is a Morse function with a unique maximum at some interior point ${s_0 \in (0,\ell)}$.
		\end{itemize}
		If we define
		\[
			N_i \defeq [0,i] \times \dd N(K)\, ,
		\]
		then
		\[
			N_0 \subset N_1 \subset N_2 \subset \cdots \subset [0,\infty) \times \dd N(K)\, ,
		\]
		is an exhaustion of $[0,\infty) \times \dd N(K)$ by compacts. Then given any geodesic $c \co [0,\ell] \longrightarrow M_K$, there exists some $m \geq 0$ so that $c(t) \in N_m$ for every $t\in [0,\ell]$. In particular, if we restrict to the present situation in this paper, where every geodesic is a loop based at $\xi\in S \setmin N(K)$. To this end fix some constant $L_0 > 0$ and assume $\gamma\in \mathcal F_{L_0} BM_K$, that is $\gamma$ is a piecewise geodesic loop based at $\xi \in S \setmin N(K)$ with length bounded above by $L_0$ (for details see \cref{sub:filtration_on_based_loops}). Then there is some $m = m(L_0, h) > 0$ depending only on $L_0$ and the metric $h$ so that $\gamma(t) \in N_m$ for every $t$. We prove that there exists some $m_0 > 0$ (depending on $m$ and the metric $h$) so that the $J$-holomorphic strips lie inside of $N_{m_0}$ by using the monotonicity lemma \cite[Proposition 4.7.2]{sikorav1994some} (see also \cite[Lemma 3.4]{cieliebak2010compactness}).

		Our metric $h$ defined in \eqref{eq:metric_definition} extends to a metric on $W_K$ such that it has bounded geometry in the terminology of \cite[Section 4]{sikorav1994some}. Furthermore, since $M_K \subset W_K$ is Lagrangian, the tuple $(W_K, J, M_K, h)$ is tame in the sense of \cite[Definition 4.1.1]{sikorav1994some}. Let $r_W, C_W > 0$ be constants so that for any $x,y\in M_K$
		\[
			d_{M_K}(x,y) \leq r_W \imp d_{W_K}(x,y) \leq C_W d_{M_K}(x,y)\, ,
		\]
		where $d_{M_K}$ and $d_{W_K}$ are the metrics induced by $h$ on $M_K$ and $W_K$ respectively. If we denote the lower bound on the injectivity radius by $\rho$, we may assume $r_W \leq \rho$.
		\begin{lma}[{\cite[Proposition 4.7.2 (ii)]{sikorav1994some}}]\label{lma:monotonicity}
			Let $(V, J, W, \mu)$ be tame. Then there exist a positive constant $C_4(W) > 0$ with the following property. Let $u\co T \longrightarrow V$ be a $J$-holomorphic curve so that $u(\dd T) \subset \dd B(x,r) \cup W$ where $x\in u(T)$ and $r < r_W$. Then
			\[
				\operatorname{area}(u(T) \cap B(x,r)) \geq C_4(W) r^2\, .
			\]
		\end{lma}
		We use this lemma with $V = W_K$, $W = F \cup M_K$ and $\mu = h$.
		\begin{thm}\label{thm:holomorphic_strips_stays_in_compact_part}
			Let $A > 0$ be arbitrary and consider a generator $a \in \mathcal F_A CW^\ast_{\varLambda_K}(F,F)$. Then there exists $m > 0$ so that $\im u \subset N_m$ for any $u\in \mathcal M(a)$.
		\end{thm}
		\begin{proof}
			Consider a generator $a \in \mathcal F_A CW^\ast_{\varLambda_K}(F,F)$ and pick some $u\in \overline{\mathcal M}(a)$. Then by \cref{prp:chain_map_is_diagonal_wrt_filtrations} we have
			\[
				L(\ev(u)) = \mathfrak a(a) < A\, ,
			\]
			Because the $J$-holomorphic disk $u\in \mathcal M(a)$ has boundary on the Reeb chord $a$, the exact Lagrangian $F \cong DT^\ast_\xi S$ for $\xi\in M_K$ and the geodesic $\gamma \defeq \ev(u)$. Therefore there is some $m' > 0$ (depending only on $A$) so that $\dd \im u \subset N_{m'}$ for any $u\in \mathcal M(a)$. Then pick some $m >m'> 0$ (which a priori can be equal to $\infty$) and assume that $\im u \subset N_m$. We consider $U\defeq \im u \cap (N_m \setmin N_{m'})$ and then we prove that $m$ is finite. Namely, fix some $r < r_W$ and let $v_1,\ldots,v_\mu \in U$ be the maximal number of points so that $d_{W_K}(v_i,v_j) > 2 r$. Then we apply \cref{lma:monotonicity} to each $U_i \defeq U \cap B(v_i,r)$ so that $\operatorname{area}(U_i) \geq C_4 r^2$ for each $i \in \parenm{1,\ldots,\mu}$. Therefore
			\[
				\operatorname{area}(U) \geq \mu \operatorname{area}(U_1) \geq \mu C_4 r^2 \Leftrightarrow \mu \leq \frac{\operatorname{area}(U)}{C_4 r^2}\, .
			\]
			Since $\mathfrak a(a)$ is bounded by $A$, so is the area of $U$. Hence
			\[
				\mu < \frac{A}{C_4 r^2} < \infty\, .
			\]
			This shows that there is some finite $m > 0$ such that $\im u \subset N_{m}$ for every $u\in \mathcal M(a)$.
		\end{proof}
	\section{Signs, gradings and orientations of moduli spaces}
		\label{sec:signs_gradings_and_orientations_of_moduli_spaces}
		In this section, we use the same conventions and setup as in \cite[Section (11)]{seidel2008fukaya} and \cite[Section 8]{fukaya2010lagrangian}.
		Pick some $T \in \overline{\mathcal H}_m$ and consider the collection of Lagrangian branes $F_0^{\#},\ldots,F_m^{\#}$ of a cotangent fiber $F_0 \cong T^\ast_\xi S \subset W_K$ at $\xi\in M_K$ and a system of parallel copies $\overline F$ as in \cref{sub:moduli_space_of_half_strips} and \cref{sec:wrapped_floer_cohomology_wo_ham}. Pick a word of generators $\boldsymbol a = a_1 \cdot \cdots \cdot a_m$ where $a_k \in CW^\ast(F_{k-1}, F_k)$, and pick abstract perturbation data so that $\mathcal M(\boldsymbol a)$ is regular. Then for some $u\in \mathcal M(\boldsymbol a)$, denote the linearization of the operator $\overline{\dd}_{J_T}$ at the $J$-holomorphic disk $u$ by $D_u$. Then we have the following:
		\begin{lma}[{\cite[Lemma 6.1]{abouzaid2012wrapped}}]\label{lma:det_in_terms_of_orientation_lines}
			With the choice as above there is a canonical up to homotopy isomorphism
			\[
				\det D_u \cong o_\xi \otimes o_{a_1}^\vee \otimes \cdots \otimes o_{a_m}^\vee \otimes o_\xi^\vee
			\]
			and in particular
			\[
				\topp (T \overline{\mathcal M}(\boldsymbol a)) \cong \topp (T \overline{\mathcal H}_m) \otimes o_\xi \otimes o_{a_1}^\vee \otimes \cdots \otimes o_{a_m}^\vee \otimes o_\xi^\vee \, .
			\]
		\end{lma}
		Since the orientation lines $o_x$ are naturally graded by the indices of the linearized operators $D_x$, we have a nautral isomorphism coming from reordering tensor products of orientation lines which produces a Koszul sign
		\[
			o_{x_1} \otimes o_{x_2} \cong (-1)^{\abs{x_1}\abs{x_2}} o_{x_2} \otimes o_{x_1}\, .
		\]
		Furthermore there are natural non-degnerate pairings
		\[
			o_x \otimes o_x^\vee \cong \R\, .
		\]
		From now on we use the following abbreviation: For the word $\boldsymbol a = a_1 \cdots a_m$, we let
		\[
			o_{\boldsymbol a} \defeq o_{a_1} \otimes \cdots \otimes o_{a_m}\, .
		\]
		As in \eqref{eq:stratification_of_deligne_mumford_compactification_1} and \eqref{eq:stratification_of_deligne_mumford_compactification_2} denote by $\mathcal H_{m}$ the moduli space of abstract $J$-holomorphic disks with $m+2$ boundary punctures, and its Deligne--Mumford compactification by $\overline{\mathcal H}_m$. Then the codimension one boundary $\dd \overline{\mathcal H}_{m}$ is covered by the natural inclusions of the following strata
		\begin{align}
			\label{eq:boundary_strata_1}
			\overline{\mathcal H}_{m_1} &\times \overline{\mathcal H}_{m_2}, \, m_1+m_2 = m \\
			\label{eq:boundary_strata_2}
			\overline{\mathcal H}_{m_1} &\times \overline{\mathcal R}_{m_2}, \, m_1+m_2 = m+1\, .
		\end{align}
		Here $\mathcal R_m$ is the Deligne--Mumford space of unit disks in the complex plane with $m+1$ boundary punctures that are oriented counterclockwise. We would like to compare the product orientation of each of the strata with the boundary orientation on $\dd \overline{\mathcal H}_{m}$. The orientation of the boundary is determined as follows. Any orientation on a manifold $X$ induces an orientation on its boundary via the outward normal first-rule. More precisely via the canonical isomorphism
		\[
			\topp TX \cong \nu_{\dd X} \otimes \topp T\dd X\, ,
		\]
		where $\nu_{\dd X}$ is the normal bundle of $\dd X$ which is canonically trivialized by the outwards normal vector along the boundary. Following the conventions in \cite{seidel2008fukaya,abouzaid2010geometric,abouzaid2012wrapped} there is a choice of coherent orientations on $\overline{\mathcal H}_m$ such that the boundary strata \eqref{eq:boundary_strata_1} and \eqref{eq:boundary_strata_2} differs from the boundary orientation on $\dd \overline{\mathcal H}_m$ by a sign $(-1)^{\dag_1}$ and $(-1)^{\dag_2}$ respectively where we have
		\begin{align}
			\label{eq:boundary_strata_1_sign_difference}
			\dag_1 &= m_1\\
			\label{eq:boundary_strata_2_sign_difference}
			\dag_2 &= m_2(m-k)+k+m_2\, ,
		\end{align}
		and $\overline{\mathcal R}_{m_2}$ is attached to the $(k+1)$-th outgoing leaf of $\overline{\mathcal H}_{m_1}$ (cf.\@ \cite[(12.22)]{seidel2008fukaya}). The first sign $\dag_1$ is obtained from \cite[(12.22)]{seidel2008fukaya} by using $m = m_2+1$, $d = m_1+m_2+1$ and $n = d$ since $\overline{\mathcal H}_{m_2}$ is attached to $\overline{\mathcal H}_{m_1}$ at the last outgoing leaf.

		The second sign $\dag_2$ is obtained from \cite[(12.22)]{seidel2008fukaya} by using $m = m_2$, $d = m_1+m_2$ and $n = k$.
		\begin{proof}[Proof of \cref{lma:product_ori_differs_from_boundary_ori}]
			We consider the moduli space $\overline{\mathcal M}(\boldsymbol a)$ and the stratification of its codimension one boundary as in \eqref{eq:compactification_stratification}. We first consider the strata of the form $\overline{\mathcal M}(\boldsymbol a') \times \overline{\mathcal M}(\boldsymbol a'')$ where $\boldsymbol a' \boldsymbol a'' = \boldsymbol a$. Then, using \cref{lma:det_in_terms_of_orientation_lines} we have
			\begin{align*}
				\topp (T \overline{\mathcal M}(\boldsymbol a')) \otimes \topp (T \overline{\mathcal M}(\boldsymbol a'')) &= \topp (T \overline{\mathcal H}_{m_1}) \otimes o_\xi \otimes o_{\boldsymbol a'}^\vee \otimes o_\xi^\vee\\
				&\quad \otimes \topp (T \overline{\mathcal H}_{m_2}) \otimes o_\xi \otimes o_{\boldsymbol a''}^\vee \otimes o_\xi^\vee\, .
			\end{align*}
			Reordering the factors so that $\topp (T \overline{\mathcal H}_{m_2})$ becomes adjacent to $\topp (T \overline{\mathcal H}_{m_1})$ introduces the Koszul sign $(-1)^{\gad_1}$ where
			\[
				\gad_1 = (m_2+1) \paren{\sum_{i=1}^{m_1} \abs{a_i}}\, ,
			\]
			since $\dim \overline{\mathcal H}_{m_2} = m_2 + 1$. Canceling the adjacent factors $o_\xi^\vee$ and $o_\xi$ then gives
			\[
				\topp (T \overline{\mathcal H}_{m_1}) \otimes \topp (T \overline{\mathcal H}_{m_2}) \otimes o_\xi \otimes o_{\boldsymbol a'}^\vee \otimes o_{\boldsymbol a''}^\vee \otimes o_\xi^\vee\, .
			\]
			Then by \eqref{eq:boundary_strata_1_sign_difference} we get a sign $(-1)^{\dag_1}$ when comparing the product orientation of $\overline{\mathcal H}_{m_1} \times \overline{\mathcal H}_{m_2}$ with the boundary orientation of $\dd \overline{\mathcal H}_{m}$. After these reorderings we arrive at
			\[
				\topp (T \dd \overline{\mathcal H}_{m}) \otimes o_\xi \otimes o_{\boldsymbol a'}^\vee \otimes o_{\boldsymbol a''}^\vee \otimes o_\xi^\vee = \topp (T \dd \overline{\mathcal H}_{m}) \otimes o_\xi \otimes o_{\boldsymbol a}^\vee \otimes o_\xi^\vee\, ,
			\]
			which is canonically isomorphic to $\topp(T \overline{\mathcal M}(\boldsymbol a))$. The total sign difference between the product orientation on $\overline{\mathcal M}(\boldsymbol a') \times \overline{\mathcal M}(\boldsymbol a'')$ and the boundary orientation on $\dd \overline{\mathcal M}(\boldsymbol a)$ is therefore
			\[
				\ddag_1 = \gad_1 + \dag_1 = (m_2+1) \paren{\sum_{i=1}^{m_1} \abs{a_i}} + m_1\, .
			\]
			Similarly, we compare the product orientation of $\overline{\mathcal M}(\boldsymbol a \setmin \tilde{\boldsymbol a}) \times \mathcal M^{\mathrm{cw}}(\tilde{\boldsymbol a})$ with the boundary orientation on $\dd \overline{\mathcal M}(\boldsymbol a)$. Recall from \eqref{eq:a_minus_subword_a_tilde_def} that if $\tilde{\boldsymbol a} \subset \boldsymbol a$ is a subword at position $t+1$, then $\boldsymbol a \setmin \tilde{\boldsymbol a}$ denotes the word $\boldsymbol a$ with the subword $\tilde{\boldsymbol a}$ replaced by an auxiliary generator $y$. Again by \cref{lma:det_in_terms_of_orientation_lines} we therefore have
			\begin{align*}
				\topp (T \overline{\mathcal M}(\boldsymbol a \setmin \tilde{\boldsymbol a})) \otimes \topp (T \mathcal M^{\mathrm{cw}}(\tilde{\boldsymbol a})) &= \topp(T \overline{\mathcal H}_{t+1+r}) \otimes o_\xi \otimes o_{\boldsymbol a \setmin \tilde{\boldsymbol a}}^\vee \otimes o_\xi^\vee \\
				&\quad \otimes \topp(T \overline{\mathcal R}_{s}) \otimes o_y \otimes o_{\tilde{\boldsymbol a}}^\vee
			\end{align*}
			Assuming that $\mathcal M^{\mathrm{cw}}(\tilde{\boldsymbol a})$ is rigid means especially that $\abs y = 2 - s + \abs y - \sum_{i=1}^{s} \abs{a_{t+i}}$ and so we move $o_{(\boldsymbol a \setmin \tilde{\boldsymbol a})_2}^\vee \otimes o_\xi^\vee$ past $\topp(T \overline{\mathcal R}_{s}) \otimes o_y \otimes o_{\tilde{\boldsymbol a}}^\vee$ without introducing any sign and arrive at
			\[
				\topp(T \overline{\mathcal H}_{t+1+r}) \otimes o_\xi \otimes o_{(\boldsymbol a \setmin \tilde{\boldsymbol a})_1}^\vee \otimes o_y^\vee \otimes \topp(T \overline{\mathcal R}_{s}) \otimes o_y \otimes o_{\tilde{\boldsymbol a}}^\vee \otimes o_{(\boldsymbol a \setmin \tilde{\boldsymbol a})_2}^\vee \otimes o_\xi^\vee \, .
			\]
			Because $\dim \overline{\mathcal R}_{s} = s$, moving $\topp(T \overline{\mathcal R}_{s})$ to the front and adjacent to $\topp(T \overline{\mathcal H}_{t+1+r})$ gives
			\[
				\topp(T \overline{\mathcal H}_{t+1+r}) \otimes \topp(T \overline{\mathcal R}_{s}) \otimes o_\xi \otimes o_{(\boldsymbol a \setmin \tilde{\boldsymbol a})_1}^\vee \otimes o_y^\vee \otimes o_y \otimes o_{\tilde{\boldsymbol a}}^\vee \otimes o_{(\boldsymbol a \setmin \tilde{\boldsymbol a})_2}^\vee \otimes o_\xi^\vee \, ,
			\]
			with a sign difference of $(-1)^{\gad_2}$ where
			\[
				\gad_2 = s\paren{\abs{\xi} + \abs{y} + \sum_{i=1}^{t}\abs{a_i}} = s \paren{\abs{\xi} + \sum_{i=1}^{t + s}\abs{a_i}}\, .
			\]
			Recall from the assumptions in \cref{lma:product_ori_differs_from_boundary_ori} that $\tilde{\boldsymbol{a}} \subset \boldsymbol a$ is a subword at position $t+1$.
			
			Then using $o_y^\vee \otimes o_y \cong \R$ and $\boldsymbol a = (\boldsymbol a \setmin \tilde{\boldsymbol a})_1\tilde{\boldsymbol a}(\boldsymbol a \setmin \tilde{\boldsymbol a})_2$ this collapses to
			\[
				\topp(T \overline{\mathcal H}_{t+1+r}) \otimes \topp(T \overline{\mathcal R}_{s}) \otimes o_\xi \otimes o_{\boldsymbol a}^\vee \otimes o_\xi^\vee \, ,
			\]
			and using \eqref{eq:boundary_strata_2_sign_difference}, $\topp(T \overline{\mathcal H}_{t+1+r}) \otimes \topp(T \overline{\mathcal R}_{s}) \cong  \topp (T \overline{\mathcal H}_{m})$ with a sign difference of $(-1)^{\dag_2}$. The total sign difference between the product orientation of $\overline{\mathcal M}(\boldsymbol a \setmin \tilde{\boldsymbol a}) \times \mathcal M^{\mathrm{cw}}(\tilde{\boldsymbol a})$ and the boundary orientation on $\dd \overline{\mathcal M}(\boldsymbol a)$ is therefore
			\[
				\ddag_2 = \gad_2 + \dag_2 = s \paren{\abs \xi + \sum_{i=1}^{t+s}\abs{a_i}} + s(m-t)+t+s \, .
			\]
		\end{proof}
		\begin{proof}[Proof of \cref{lma:varpsi_a_infty_morphism} (continued)]
			To confirm that the signs match up in the $A_{\infty}$-relation
			\[
				\dd \varPsi_m + \sum_{m_1+m_2 = m} P(\varPsi_{m_2} \otimes \varPsi_{m_1}) = \sum_{r+s+t = m} (-1)^{\maltese_t} \varPsi_{r+1+t}(\id^{\otimes r} \otimes \mu^s \otimes \id^{\otimes t})\, ,
			\]
			we look at the terms one by one and compute the sign that is in front of each term. In the first term $\dd \varPsi_m$ it is only the sign from $\varPsi_m$ that is taken into account, namely $(-1)^{\S}$ where
			\[
				\S = \sum_{i=1}^m i \abs{a_i} + (\abs \xi + m) \sum_{i=1}^m \abs{a_i}\, .
			\]
			The second term has a sign coming from:
			\begin{enumerate}
				\item The definition of the Pontryagin product $P$ in \eqref{eq:def_pontryagin_product} contributes with a sign $(-1)^{\circ}$ where
				\[
					\circ = \abs{\varPsi_{m_1}(a_{m_1}\otimes \cdots \otimes a_1)} = \dim \overline{\mathcal M}(\boldsymbol a') = -1 + m_1 - \sum_{i=1}^{m_1}\abs{a_i}\, ,
				\]
				\item the difference between the product orientation on $\overline{\mathcal H}_{m_1} \times \overline{\mathcal H}_{m_2}$ and the boundary orientation on $\dd \overline{\mathcal H}_m$ is $(-1)^{\ddag_1}$, where $m_1+m_2 = m$ and
				\[
					\ddag_1 = (m_2+1) \paren{\sum_{i=1}^{m_1} \abs{a_i}} + m_1\, ,
				\]
				as in \cref{lma:product_ori_differs_from_boundary_ori},
				\item the definition of $\varPsi_{m_1}(a_{m_1} \otimes \cdots \otimes a_1)$ in \eqref{eq:def_of_chain_maps} contributes with a sign $(-1)^{\S_1}$ where
				\[
					\S_1 = \sum_{i=1}^{m_1} i\abs{a_i} + (\abs \xi + m_1) \sum_{i=1}^{m_1} \abs{a_i}\, ,
				\]
				and
				\item the definition of $\varPsi_{m_2}(a_{m} \otimes \cdots \otimes a_{m_1+1})$ in \eqref{eq:def_of_chain_maps} contributes with a sign $(-1)^{\S_2}$ where
				\[
					\S_2 = \sum_{i=m_1+1}^m (i-m_1)\abs{a_i} + (\abs \xi + m_2) \sum_{i=m_1+1}^{m} \abs{a_i}\, .
				\]
			\end{enumerate}
			Now it is straightforward to check that $\circ + \ddag_1 + \S_1 + \S_2 = 1+\S \pmod 2$.
			\begin{align*}
				\circ + \ddag_1 + \S_1 + \S_2 &= -1 + m_1 - \sum_{i=1}^{m_1}\abs{a_i} + (m_2+1) \paren{\sum_{i=1}^{m_1} \abs{a_i}} + m_1 + \sum_{i=1}^{m_1} i\abs{a_i} \\
				&\quad + (\abs \xi + m_1) \sum_{i=1}^{m_1} \abs{a_i} + \sum_{i=m_1+1}^m (i-m_1)\abs{a_i} + (\abs \xi + m_2) \sum_{i=m_1+1}^{m} \abs{a_i} \\
				&= 1+ m_2 \sum_{i=1}^{m_1} \abs{a_i} + \sum_{i=1}^m i \abs{a_i} +(\abs \xi + m_1) \sum_{i=1}^{m_1}\abs{a_i} + m_1 \sum_{i=m_1+1}^m \abs{a_i} \\
				&\quad + (\abs \xi + m_2) \sum_{i={m_1+1}}^m \abs{a_i} \\
				&= 1+ \sum_{i=1}^m i \abs{a_i} + (\abs \xi + m) \sum_{i=1}^m \abs{a_i} = 1+\S \pmod 2\, .
			\end{align*}
			Next we consider the term in the right hand side. Let $y\defeq \mu^s(a_{t+s}\otimes \cdots \otimes a_{t+1})$. This sum has a sign coming from:
			\begin{enumerate}
				\item The difference between the product orientation on $\overline{\mathcal H}_{r+1+t} \times \overline{\mathcal R}_{s}$ and the boundary orientation on $\dd \overline{\mathcal H}_{m}$ is $(-1)^{\ddag_2}$ where $r+s+t = m$ and
				\[
					\ddag_2 = s \paren{\abs \xi + \sum_{i=1}^{t+s} \abs{a_i}} + s(m-t) + t + s\, ,
				\]
				\item the definition of $\mu^s(a_{t+s} \otimes \cdots \otimes a_{t+1})$ in \eqref{eq:def_of_a_infty_operations_wo_ham} contributes with a sign $(-1)^{\diamond}$ where
				\[
					\diamond = \sum_{i=t+1}^{t+s}(i-t)\abs{a_i}\, ,
				\]
				\item the definition of $\varPsi_{r+1+t}(a_m \otimes \cdots \otimes a_{t+s+1} \otimes y \otimes a_{t} \otimes \cdots \otimes a_1)$ in \eqref{eq:def_of_chain_maps} contributes with a sign $(-1)^{\tilde \S}$ where
				\begin{align*}
					\tilde \S &= \sum_{i=1}^t i \abs{a_i} + (t+1) \abs y + \sum_{i=t+s+1}^{m}(i-s+1)\abs{a_i} + (\abs \xi + r+t+1)\paren{\sum_{i=1}^t \abs{a_i} + \abs y + \sum_{i=t+s+1}^m \abs{a_i}} \, .
				\end{align*}
				Note that since we assume that $\mathcal M^{\mathrm{cw}}(a_{t+1}\cdots a_{t+s})$ is rigid, we have
				\[
					\abs y = 2-s+ \sum_{i=t+1}^{t+s} \abs{a_i} = s + \sum_{i=t+1}^{t+s} \abs{a_i} \pmod 2\, ,
				\]
				hence we get
				\begin{align*}
					\tilde \S &= \sum_{i=1}^t i \abs{a_i} + (t+1)\paren{s+ \sum_{i=t+1}^{t+s} \abs{a_i}} + \sum_{i=t+s+1}^{m}(i-s+1)\abs{a_i} \\
					&\quad + (\abs \xi + r+t+1)\paren{s + \sum_{i=1}^m \abs{a_i}}\\
					&= \sum_{i=1}^t i \abs{a_i} + (t+1) \sum_{i=t+1}^{t+s}\abs{a_i} + (s+1) \sum_{i=t+s+1}^m \abs{a_i} \\
					&\quad + \sum_{i=t+s+1}^m i \abs{a_i} + (\abs \xi + r)s + (\abs \xi + r+t+1) \sum_{i=1}^m \abs{a_i} \pmod 2\, .
				\end{align*}
				It is then again a straightforward calculation to show that $\ddag_2 + \diamond + \tilde \S + \maltese_t = \S\pmod 2$.
				\begin{align*}
					\ddag_2 + \diamond + \tilde \S + \maltese_t &= s \paren{\abs \xi + \sum_{i=1}^{t+s} \abs{a_i}} + s(m-t) + t + s + \sum_{i=t+1}^{t+s}(i-t)\abs{a_i} \\
					&\quad + \sum_{i=1}^t i \abs{a_i} + (t+1) \sum_{i=t+1}^{t+s}\abs{a_i} + (s+1) \sum_{i=t+s+1}^m \abs{a_i} \\
					&\quad + \sum_{i=t+s+1}^m i \abs{a_i} + (\abs \xi + r)s + (\abs \xi + r+t+1) \sum_{i=1}^m \abs{a_i} + \sum_{i=1}^t \abs{a_i} + t \\
					&= s \sum_{i=1}^{t+s} \abs{a_i} + s(m-t)+t+s + \sum_{i=t+1}^{t+s} i \abs{a_i} + \sum_{i=1}^t i \abs{a_i} \\
					&\quad +\sum_{i=t+1}^{t+s} \abs{a_i}+ (s+1) \sum_{i=t+s+1}^m \abs{a_i} + \sum_{i=t+s+1}^m i \abs{a_i} \\
					&\quad +  rs + (\abs \xi +r+t+1) \sum_{i=1}^m \abs{a_i} + \sum_{i=1}^t \abs{a_i}+t \\
					&= m \sum_{i=1}^m \abs{a_i} + rs+ s^2 +s+\sum_{i=1}^m i\abs{a_i} + rs + \abs \xi \sum_{i=1}^m \abs a_i \\
					&= \sum_{i=1}^m i\abs{a_i} + (\abs \xi + m) \sum_{i=1}^m \abs{a_i} = \S \pmod 2\, .
				\end{align*}
			\end{enumerate}
		\end{proof}
\bibliographystyle{../../bib/alpha2}
\bibliography{../../bib/ref}
\end{document}